\documentclass[letterpaper,11pt]{amsart}
\usepackage{amscd}
\usepackage{amsmath,amssymb,amsthm}
\usepackage{enumerate}
\usepackage{mathtools}
\usepackage[usenames,dvipsnames]{xcolor}
\usepackage{stmaryrd}
\usepackage[margin=1in]{geometry}
\usepackage[mathscr]{eucal}
\usepackage{cite}
\usepackage{upgreek}
\usepackage[bookmarks=false]{hyperref}
\usepackage[T2A]{fontenc}
\usepackage[utf8]{inputenc}
\usepackage{mathrsfs}
\usepackage{comment}

\newtheorem{theorem}{Theorem}[section]
\newtheorem{lemma}[theorem]{Lemma}
\newtheorem{proposition}[theorem]{Proposition}
\newtheorem{corollary}[theorem]{Corollary}

\newtheorem{claim}[theorem]{Claim}
\newtheorem{fact}[theorem]{Fact}

\theoremstyle{definition}
\newtheorem{definition}[theorem]{Definition}
\newtheorem{notation}[theorem]{Notation}

\theoremstyle{remark}

\newtheorem{remark}[theorem]{Remark}

\numberwithin{equation}{section}

\newcommand{\bC}{{\mathbb C}}

\newcommand{\bM}{{\mathbb M}}
\newcommand{\bN}{{\mathbb N}}

\newcommand{\bR}{{\mathbb R}}
\newcommand{\bS}{{\mathbb S}}
\newcommand{\bT}{{\mathbb T}}

\newcommand{\cA}{{\mathcal A}}

\newcommand{\cF}{{\mathcal F}}

\newcommand{\cM}{{\mathcal M}}
\newcommand{\cN}{{\mathcal N}}
\newcommand{\cO}{{\mathcal O}}

\newcommand{\cQ}{{\mathcal Q}}

\newcommand{\cU}{{\mathcal U}}
\newcommand{\cV}{{\mathcal V}}

\newcommand{\rT}{\mathrm{T}}

\DeclareMathOperator{\re}{Re}

\DeclareMathOperator{\tr}{tr}

\DeclareMathOperator{\Tr}{Tr}

\DeclareMathOperator{\vol}{vol}

\DeclareMathOperator{\supp}{supp}

\DeclareMathOperator{\ev}{ev}

\DeclareMathOperator{\op}{op}
\DeclareMathOperator{\orb}{orb}
\DeclareMathOperator{\sa}{sa}

\DeclareMathOperator{\Ent}{Ent}
\DeclareMathOperator{\tp}{tp}
\DeclareMathOperator{\Lip}{Lip}

\DeclareMathOperator{\Th}{Th}
\DeclareMathOperator{\full}{full}
\DeclareMathOperator{\factor}{fact}
\DeclareMathOperator{\dcl}{dcl}

\DeclareMathOperator{\CEP}{CEP}
\DeclareMathOperator{\law}{law}
\DeclareMathOperator{\qf}{qf}

\DeclarePairedDelimiter{\norm}{\lVert}{\rVert}
\DeclarePairedDelimiter{\ip}{\langle}{\rangle}

\begin{document}
	
	\title[Information geometry for types in the large-$n$ limit of random matrices]{Information geometry for types \\ in the large-$n$ limit of random matrices}
	
	\author{David Jekel}
	\address{Department of Mathematical Sciences, University of Copenhagen, Universitetsparken 5, 2100 Copenhagen {\O}, Denmark}
	\email{daj@math.ku.dk}
	\urladdr{https://davidjekel.com}
	
	\maketitle
	
	\begin{abstract}
		We study the interaction between entropy and Wasserstein distance in free probability theory.  In particular, we give lower bounds for several versions of free entropy dimension along Wasserstein geodesics, as well as study their topological properties with respect to Wasserstein distance.  We also study moment measures in the multivariate free setting, showing the existence and uniqueness of solutions for a regularized version of Santambrogio's variational problem.  The role of probability distributions in these results is played by \emph{types}, functionals which assign values not only to polynomial test functions, but to all real-valued logical formulas built from them using suprema and infima.  We give an explicit counterexample showing that in the framework of non-commutative laws, the usual notion of probability distributions using only non-commutative polynomial test functions, one cannot obtain the desired large-$n$ limiting behavior for both Wasserstein distance and entropy simultaneously in random multi-matrix models.
	\end{abstract}
	
	\section{Introduction}
	
	\subsection{Motivation}
	
	This work is part of a continuing project to develop Wasserstein information geometry for free probability, and thus (we hope) for the large-$n$ limit of \emph{invariant random multi-matrix ensembles}, that is, random $m$-tuples of $n \times n$ matrices whose joint distribution is invariant under unitary conjugation.  What I call Wasserstein information geometry concerns the interaction between optimal transport theory and measures of information such as entropy.  In particular, the $L^2$ Wasserstein distance of two probability measures $\mu, \nu \in \mathcal{P}(\bC^m)$ is the infimum of $\norm{\mathbf{X} - \mathbf{Y}}_{L^2}$ over random variables $\mathbf{X} \sim \mu$ and $\mathbf{Y} \sim \nu$ in some diffuse probability space.  The Wasserstein distance is a natural metric, but even better, it arises as the Riemannian distance associated to a certain Riemannian metric on (a dense subset of) $\mathcal{P}(\bC^m)$, allowing the concepts of Riemannian geometry to be applied in optimal transport theory \cite{Arnold1966,Lafferty1988,OV2000,Otto2001}.  Furthermore, the differential entropy $h(\mu) = \int -\rho \log \rho$ of a probability measure $\rho(x)\,dx$ defines a functional on the manifold of probability measures with many natural properties.  For instance, the heat evolution of a probability measure $\mu$ is precisely the upward gradient flow of $h$ \cite{JKO1998,Otto2001}.  The entropy $h$ is also \emph{geodesically concave}, meaning that if $\mu_t$ is a geodesic in the Wasserstein manifold, then $t \mapsto h(\mu_t)$ is concave \cite{McCann1997convexity} \cite[Corollary 17.19]{Villani2008}.\footnote{Of course, according to the opposite sign convention, it would be convex.}
	
	This paper, and the larger project of free information geometry, seek analogous objects and results in the free probabilistic setting that also reflect the large-$n$ behavior of appropriate random matrix models.  We are thus driven by the following questions:
	\begin{itemize}
		\item \emph{What are the correct analogs of entropy and of Wasserstein distance in free probability?}
		\item \emph{How do the free entropy and free Wasserstein distance relate to each other?}
		\item \emph{Under what conditions do the free entropy and free Wasserstein distance describe the large-$n$ limit of classical entropy and Wasserstein distance for invariant random multi-matrix ensembles?}
	\end{itemize}
	
	The invariant multi-matrix ensembles in which we are chiefly interested are those with a probability distribution
	\begin{equation} \label{eq: matrix model formula intro}
		d\mu^{(n)}(\mathbf{X}) = \frac{1}{Z^{(n)}} e^{-n^2 V^{(n)}(\mathbf{X})}\,d\mathbf{X}, \qquad \mathbf{X} \in (\bM_n)^m,
	\end{equation}
	where $d\mathbf{X}$ is Lebesgue measure, $Z^{(n)}$ is a normalizing constant, and $V^{(n)}: (\bM_n)^m \to \bR$ is a potential of the form $V^{(n)}(\mathbf{X}) = \re \tr_n(p(\mathbf{X}))$ for a non-commutative $*$-polynomial $p$ with sufficient growth at $\infty$; see \cite{GMS2006,GS2009,JekelEntropy}.  More generally, $V^{(n)}$ can include products of traces, which leads to a certain notion of tracial non-commutative smooth functions (see \cite{JLS2022} and \cite{DGS2021}).  Later in this work, we will allow $V^{(n)}$ to be given by a formula that also incorporates suprema and infima over the unit ball in auxiliary variables.  For many classes of these multi-matrix ensembles, the trace $\tr_n(q(\mathbf{X}))$ for each non-commutative $*$-polynomial $q$ converges to a deterministic limit.
	
	To describe the large-$n$ behavior of information geometry for these invariant multi-matrix ensembles, we need an appropriate non-commutative analog of probability distributions, as well as analogs of entropy and Wasserstein distance for these distributions.  Regarding distributions, \emph{non-commutative laws} are objects that specify a ``trace'' for any non-commutative polynomial, and every tuple $\mathbf{x}$ in a tracial von Neumann algebra has such a law (see \S \ref{subsubsec: NC law}).  Regarding entropy, the large-$n$ limit of the differential entropy $h$ of $\mu^{(n)}$ with appropriate renormalization should be described by Voiculescu's free entropy $\chi$ in the microstates framework \cite{VoiculescuFE2}.  Regarding Wasserstein distance, Biane and Voiculescu \cite{BV2001} defined an analogous metric for non-commutative laws $\mu$ and $\nu$ as the infimum of $\norm{\mathbf{x} - \mathbf{y}}_{L^2(\cM)^m}$ over tuples $\mathbf{x}$ and $\mathbf{y}$ with laws $\mu$ and $\nu$ respectively.   Moreover, various inequalities from information geometry have analogs in the free setting \cite{Biane2003,HU2006,VoiculescuFE5,Dabrowski2010,Cebron2021,JLS2022,ST2022,Diez2024}.
	
	Guionnet and Shlyakhtenko \cite{GS2014} showed that for certain choices of convex $V$ in \eqref{eq: matrix model formula intro}, the law $\mu$ can be expressed as a pushforward of a semicircular family (the free analog of a Gaussian vector) by the gradient of a convex function.  Approximate transport maps were used by Figalli and Guionnet to establish universality for the local statistics of these multi-matrix models as well \cite{FG2016}, building on prior joint work with Bekerman in the single matrix case \cite{BFG2015}.  See \cite{Nelson2015a,DGS2021,JekelExpectation,JLS2022,BahrBoschert2023} for related results on free transport.  Nonetheless, it has still proved challenging to relate the free Wasserstein distance with the large-$n$ limit of classical Wasserstein distances for invariant multi-matrix ensembles, especially without assuming $V$ is convex.  Indeed, consider two invariant random multi-matrices $\mathbf{X}^{(n)}$ and $\mathbf{Y}^{(n)}$ whose non-commutative laws converge in probability to some deterministic limits $\mu$ and $\nu$.  If $d_{W,\CEP}(\mu,\nu)$ is the Biane--Voiculescu--Wasserstein distance where the couplings are restricted to Connes-embeddable von Neumann algebras (see \cite[\S 5.3, \S 6.1]{GJNS2021}), then there will exist some deterministic matrix tuples $\tilde{\mathbf{X}}^{(n)}$ and $\tilde{\mathbf{Y}}^{(n)}$ whose laws converge to $\mu$ and $\nu$ and whose distance converges to $d_W(\mu,\nu)$ using \cite[Lemma 5.12]{GJNS2021}.  However, we do not know if this distance can be achieved by a coupling of random multi-matrix models $\mathbf{X}^{(n)}$ and $\mathbf{Y}^{(n)}$ given as in \eqref{eq: matrix model formula intro}. For this to happen, the minimal distance has to be achievable by some value of $\mathbf{Y}^{(n)}$ for \emph{most} given values of $\mathbf{X}^{(n)}$ with the same limiting law, i.e., most choices of $\mathbf{X}$ in Voiculescu's \cite{VoiculescuFE2} microstate space $\Gamma_R^{(n)}$ associated to neighborhoods of $\mu$, since invariance means that the probability mass of $\mathbf{X}^{(n)}$ is spread approximately uniformly the microstate spaces.  Besides this, even if we can choose an appropriate value of $\mathbf{Y}^{(n)}$ associated to each $\mathbf{X}^{(n)}$, it is unclear if $\mathbf{Y}^{(n)}$ would be approximately uniformly distributed over its microstate space.  In short, the constraints of minimal distance and of uniform distribution may be incompatible.  In Proposition \ref{prop: counterexample}, we show that there are non-commutative laws $\mu$ and $\nu$ such that no random matrix approximations can simultaneously achieve the ``correct'' entropy and Wasserstein distance.
	
	Analyzing convergence of the Wasserstein distances of random matrix models in the large-$n$ limit thus requires more precise results about the optimal couplings.  In the classical setting, for an optimal coupling $\mathbf{X}$, $\mathbf{Y}$ of sufficiently regular measures $\mu$ and $\nu$, then we have $\mathbf{Y} = \nabla \varphi(\mathbf{X})$ for some convex function $\varphi$.  Moreover, $\mathbf{X} = \nabla \psi(\mathbf{Y})$ where $\psi(y) = \sup_x [\re \ip{x,y} - \varphi(x)]$ is the convex conjugate or Legendre transform of $\varphi$.  Monge-Kantorovich duality and Legendre transforms were studied in the non-commutative setting in \cite{JLS2022,GJNS2021}.\footnote{A distinct version of Legendre transform was used earlier by Hiai to define an analog of free entropy based on free pressure \cite{Hiai2005pressure}.  This Legendre transform was based on the duality between non-commutative laws and non-commutative polynomials rather than the self-duality of $L^2(\cM)$ which we use here.}  This shows, for instance, if the optimal coupling $(\mathbf{x},\mathbf{y})$ of the laws $\mu$ and $\nu$ has the form where $\mathbf{y} = \nabla \varphi(\mathbf{x})$ for some sufficiently smooth convex function $\varphi$, and if this $\varphi$ arises as the large-$n$ limit of corresponding convex functions $\varphi^{(n)}$ associated to the random multi-matrices, then we can conclude convergence of the Wasserstein distance.  This applies when $\mathbf{X}^{(n)}$ are independent Ginibre random matrices (i.e.\ their real and imaginary parts are GUE), and $\mathbf{Y}^{(n)}$ is given by a convex potential.  Such results about convergence of Wasserstein distance for non-convex $V^{(n)}$ remain out of reach.
	
	A fundamental difficulty for Monge-Kantorovich duality is that a natural class of scalar-valued ``non-commutative continuous functions,'' defined as uniform limits of trace polynomials as in \cite{JekelEntropy,JLS2022}, is not closed under partial suprema and infima.  Of course, the Legendre transform is given by such a partial supremum, namely, the supremum over $x$ of $\re \ip{x,y} - \varphi(x)$.  Given non-commutative $*$-polynomials $p_1$, \dots, $p_k$ in $m+1$ variables and $f: \bC^k \to \bR$ continuous, consider the function
	\begin{equation} \label{eq: supremum formula}
		\psi^{\cM}(\mathbf{x}) = \sup_{y \in D_1^{\cM}} f(\tr(p_1(\mathbf{x},y)),\dots, \tr(p_k(\mathbf{x},y))),
	\end{equation}
	where $\mathbf{x} = (x_1,\dots,x_m)$ and $D_1^{\cM}$ denotes the unit ball of $\cM$ with respect to operator norm.  If $\psi$ is evaluated in the matrix algebra $\bM_n$ for a fixed $n$, then $\psi^{\bM_n}$ is a unitarily invariant function and hence can be approximated by trace polynomials.  However, this fails when $\cM$ is a $\mathrm{II}_1$ factor\footnote{That is, infinite-dimensional tracial von Neumann algebras with trivial center.} for general $p_1$, \dots, $p_k$ and $f$.  It is impossible to approximate $\psi^{\cM}$ by functions of traces of polynomials in $\mathbf{x}$ because $\mathcal{M}$ \emph{does not admit quantifier elimination}, which is a notion from model theory.
	
	In the model theory of metric structures \cite{BYBHU2008,BYU2010}, one considers real-valued formulas such as \eqref{eq: supremum formula} as an analog of usual boolean-valued logical formulas, in which $\sup$ and $\inf$ serve as the quantifiers and continuous functions serve as the logical connectives (instead of the usual boolean connectives such as ``and'' and ``or'').  \emph{Quantifier elimination} then means that arbitrary formulas can be approximated uniformly by quantifier-free formulas.  We also remark that for quantifier elimination it would be sufficient for formulas with a \emph{single} quantifier to be approximated by quantifier-free formulas, since the case of multiple nested quantifier could then be handled by induction.  Farah \cite{Farah2024} showed that $\mathrm{II}_1$ factors never admit quantifier elimination, hence the impossibility of approximating \eqref{eq: supremum formula} by trace polynomials.  This in turn implies that even though $\psi^{\bM_n}$ can be approximated by trace polynomials for each $n$, there is no way to do this uniformly for all $n$.  By contrast, atomless classical probability spaces do admit quantifier elimination \cite{BY2012,JekelModelEntropy}.  The lack of quantifier elimination in the non-commutative setting is a challenge not only for non-commutative optimal transport theory, but also for the study of free entropy and large deviations for invariant random matrix ensembles, as the natural choice of large deviations rate function from \cite{BCG2003} is also given by an infimum (in fact, a stochastic control problem, which goes beyond even logical formulas); see \cite[\S 6.2]{JekelModelEntropy}.
	
	We interpret the lack of quantifier elimination as an indication that the class of uniform limits of trace polynomials is not the correct notion of invariant function or observable; rather, we should work over the larger class of formulas involving iterated suprema and infima.  This in turn expands the class of invariant matrix ensembles.  Moreover, the notion of non-commutative probability distribution should be correspondingly modified by expanding the class of test functions from trace polynomials to formulas.  This results in the replacement of the non-commutative law, or quantifier-free type, by the full type.  The \emph{type} of $\mathbf{x} = (x_1,\dots,x_m)$ in a tracial von Neumann algebra $(\cM,\tau)$ is the mapping $\tp^{\cM}(\mathbf{x}): \varphi \mapsto \varphi^{\cM}(\mathbf{x})$ for formulas $\varphi$.  The microstates free entropy $\chi$ and the Wasserstein distance $d_W$ are then replaced by the corresponding versions for full types.  Rather fortuitously, the term ``type'' which we imported from model theory is also used in the theory of Shannon entropy and microstate spaces, and the role of types vis-{\`a}-vis the matricial microstate spaces in this work is analogous that of types for Shannon entropy.
	
	The Wasserstein distance for types in tracial von Neumann algebras has an analog of Monge-Kantorovich duality \cite{JekelTypeCoupling}, where the functions $\varphi$ and $\psi$ are \emph{definable predicates}, i.e., uniform limits of formulas.  The statement is as follows.  For two types $\mu$ and $\nu$, let $C(\mu,\nu)$ denote the maximum inner product $\re \ip{\mathbf{x},\mathbf{y}}_{L^2(\cM)^m}$ over all couplings $(\mathbf{x},\mathbf{y})$.  A pair of convex definable predicates $(\varphi,\psi)$ will be called \emph{admissible} if $\varphi^{\cM}(\mathbf{x}) + \psi^{\cM}(\mathbf{y}) \geq \re \ip{\mathbf{x},\mathbf{y}}_{L^2(\cM)^m}$.  Then we have
	\[
	C(\mu,\nu) = \inf_{(\varphi,\psi) \text{ admissible}} \left[ (\mu,\varphi) + (\nu,\psi) \right],
	\]
	where $(\mu,\varphi)$ denotes the evaluation or dual pairing of a type and a definable predicate, i.e., $(\mu,\varphi) = \varphi^{\cM}(\mathbf{x})$ when $\tp^{\cM}(\mathbf{x}) = \mu$.
	
	The incorporation of model theoretic concepts into free probability is not without drawbacks.  It is not even known whether the large-$n$ limit of $\varphi^{\bM_n}$ exists when $\varphi$ is a formula with no free variables (called a \emph{sentence}).  In fact, in an analogous situation where we consider permutation groups with the Hamming metric instead of $\bM_n$, the theories do \emph{not} converge as $n \to \infty$ \cite{AlTh2024}.  Since the permutation group can be obtained as the normalizer of the diagonal subalgebra in $\bM_n$ modulo its center, this suggests the theories of $\bM_n$ might not converge as $n \to \infty$.\footnote{I thank Andreas Thom, Vadim Alekseev, Ilijas Farah, Sorin Popa, and Ben Hayes for discussions on this topic.}  Even if the theories do converge, it seems intractable by current methods to determine whether an arbitrary formula evaluated (for instance) on a GUE matrix tuple converges as $n \to \infty$, or what the limit would be.  If a method was discovered to achieve this, we perhaps also be able to complete the large deviations programme of \cite{BCG2003}.
	
	However, regardless of whether the limits as $n \to \infty$ exist or not, we can consider limits along a given ultrafilter $\cU$ (see \S \ref{subsubsec: ultraproduct}) and examine free information geometry in the limit as $n \to \cU$.  In this paper, we will use the Monge-Kantorovich duality of \cite{JekelTypeCoupling} to study the relationship between Wasserstein distance and entropy for types.  In particular, our goals are:
	\begin{itemize}
		\item To give estimates for microstates free entropy, free entropy dimension, and $1$-bounded entropy along a Wasserstein geodesic in terms of the endpoints.
		\item To study the topological properties of free entropy and $1$-bounded entropy on the space of types with Wasserstein distance.
		\item To define Gibbs types associated to convex potentials from the space of definable predicates and show they satisfy an analog of the Talagrand inequality.
		\item To obtain for each type $\mu$, a corresponding ``moment type'' $\nu$ that maximizes $\chi_{\full}^{\cU}$ minus the optimal inner product $C_{\full}(\mu,\nu)$ minus a small quadratic term (added for regularization).
	\end{itemize}
	
	\subsection{Results}
	
	Our first result concerns the behavior of several free entropy quantities along Wasserstein geodesics in $\mathbb{S}_m(\rT_{\cU})$, motivated by the geodesic concavity of entropy on $\mathcal{P}(\bC^m)$.  First, let us remark that geodesic concavity is different from concavity under convex combinations of measures.  Classical entropy is both concave along Wasserstein geodesics and concave with respect to convex combinations of measures.  The free entropy $\chi$ in the multivariate and non-self-adjoint settings does \emph{not} satisfy concavity under convex combinations of non-commutative laws, and in fact any nondegenerate convex combination will have entropy $-\infty$ \cite[Corollary 4.3]{VoiculescuFE3}; Hiai used free pressure to define a concave version of entropy in \cite{Hiai2005pressure}, and Biane and Dabrowski also formulated a concavified version of $\chi$ in \cite{BD2013}.  However, from the viewpoint of Wasserstein information geometry, the behavior along geodesics is much more important than the behavior under plain convex combinations, and no one has yet attempted to establish geodesic concavity of entropy in the multivariate free setting.  Although we do not yet know if geodesic concavity or even semiconcavity holds for $\chi_{\full}^{\cU}$ per se, we nonetheless obtain lower bounds for $\chi_{\full}^{\cU}$ on the interior of the geodesic in the setting of full types.
	
	The free entropy quantities under consideration are the following (see \S \ref{subsec: entropy definitions} for precise definitions):
	\begin{itemize}
		\item The microstate free entropy $\chi$ of \cite{VoiculescuFE2} is the analog of the differential entropy $-\int \rho \log \rho$.  It is defined as the exponential growth rate of the Lebesgue measure of certain matricial microstate spaces.  The version for full types was studied in \cite{JekelModelEntropy}.
		\item The microstate free entropy dimension $\delta$ \cite{VoiculescuFE2} is analogous to the Minkowski dimension for the support of a measure \cite{GS2007}.  As shown by Jung \cite{Jung2003FED}, it is obtained by taking the exponential growth rate of $\varepsilon$-covering number of the microstate spaces, normalizing by $1 / \log(1/\varepsilon)$, and taking the $\limsup$ as $\varepsilon \searrow 0$.  The version for full types is discussed for the first time in this paper in \S \ref{subsec: entropy definitions}.
		\item The $1$-bounded entropy $h$ is a metric entropy quantity from \cite{Hayes2018} (based on \cite{Jung2007S1B}).  Here one uses covering numbers up to unitary conjugacy, and does \emph{not} normalize the covering numbers by $\log(1/\varepsilon)$.  While $h$ does not have a nontrivial direct analog for classical measures, the definition instead follows an analogy with dynamical entropy.  The version for full types was studied in \cite{JekelCoveringEntropy}.
	\end{itemize}
	Because we also use $h$ to denote classical differential entropy, we will usually denote the metric entropy by $\Ent$ rather than $h$.
	
	In the following results, $\mathbb{S}_{m,R}(\rT_{\cU})$ is the space of types that arise from $m$-tuples of matrices of operator norm bounded by $R$ in the limit as $n \to \cU$.  Equivalently, it is the space of types in the ultraproduct $\prod_{n \to \cU} \bM_n$.  In the notation $\mathbb{S}_{m,R}(\rT_{\cU})$, the $\rT_{\cU}$ refers to the limit of the theories of $\bM_n$ as $n \to \cU$, or equivalently the theory of $\prod_{n \to \cU} \bM_n$; it is customary to associate the space of types to a \emph{theory} rather than to a particular structure $\cM$.  We have the following estimates for free entropy along Wasserstein geodesics in the type space.
	
	\begin{theorem}[Free entropy along geodesics] \label{thm: entropy along geodesics}
		Fix an ultrafilter $\cU$ and let $\cQ = \prod_{n \to \cU} \bM_n$. Let $\mu, \nu \in \mathbb{S}_m(\rT_{\cU})$ and let $(\mathbf{x},\mathbf{y})$ be an optimal coupling of $(\mu,\nu)$.  Let $\mathbf{x}_t = (1-t) \mathbf{x} + t \mathbf{y}$ and $\mu_t = \tp^{\cQ}(\mathbf{x}_t)$.  Then
		\begin{enumerate}
			\item $\Ent_{\full}^{\cU}(\mu_t) \geq \max(\Ent_{\full}^{\cU}(\mu), \Ent_{\full}^{\cU}(\nu))$.
			\item $\delta_{\full}^{\cU}(\mu_t) \geq \max(\delta_{\full}^{\cU}(\mu),\delta_{\full}^{\cU}(\nu))$.
			\item $\chi_{\full}^{\cU}(\mu_t) \geq \max(\chi_{\full}^{\cU}(\mu) + 2m \log (1-t), \chi_{\full}^{\cU}(\mu) + 2m \log t)$.
		\end{enumerate}
	\end{theorem}
	
	Claim (3) here is the most difficult, and its proof requires a lifting lemma (Lemma \ref{lem: lifting}) which allows one to extend a given random matrix model of $\mathbf{x}$ to a matrix model of $(\mathbf{x},\mathbf{x}_t)$; this crucially relies on Monge-Kantorovich duality with definable predicates, which is not available for quantifier-free types, or non-commutative laws.  In fact, we show in Remark \ref{rem: qf failure} that the analog of Lemma \ref{lem: lifting} fails for quantifier-free types.  We do not necessarily expect (1), (2), or (3) to hold for the plain $1$-bounded entropy, free entropy dimesion, and free entropy laws because they do not take account of the ambient algebra properly.  One could also attempt to prove an analogous result using entropy in the presence, or equivalently the entropy of existential types, but the correct analog of Wasserstein distance in this setting is not yet clear, and hence we focus on the setting of full types.
	
	It is natural to hope for concavity of $\chi_{\full}^{\cU}$ along the geodesic, but we are currently unable to prove this due to a lack of smoothness for the definable predicates in the optimal couplings.  Since claim (3) gives us upper bounds for the entropy at the endpoints of the geodesic, we can reduce the general problem to showing concavity of $t \mapsto \chi_{\full}(\mathbf{x}_t)$ for $t \in (0,1)$, and we know that there are bi-Lipschitz transport maps between $\mathbf{x}_t$ and $\mathbf{x}_s$ (see \S \ref{subsec: displacement interpolation}). One would like to compute and show concavity for the entropy of $\mathbf{x}_t$ with the change of variables formula, but that would require evaluating the log-determinant of the derivative of the transport map, and in fact we only know the transport map is bi-Lipschitz, so its derivative may not be well-defined.  It is unknown whether definable functions can be approximated by some sort of ``$C^1$ definable functions'' in the non-commutative setting.
	
	The proof of (3) also shows local Lipschitz continuity of $\chi_{\full}^{\cU}$ along Wasserstein geodesics as follows.
	
	\begin{proposition}[Modulus of continuity of $\chi$ along geodesics] \label{prop: geodesic continuity}
		Consider the same setup as Theorem \ref{thm: entropy along geodesics}.  Let $0 \leq s < t \leq 1$.  Then
		\[
		2m \log \frac{1-t}{1-s} \leq \chi_{\full}^{\cU}(\mu_t) - \chi_{\full}^{\cU}(\mu_s) \leq 2m \log \frac{t}{s}.
		\]
	\end{proposition}
	
	In light of this result, we also investigate more generally how free entropy relates with the topology on the space of laws.  Here we must be careful:  In the non-commutative setting the Wasserstein topology is much stronger than the weak-$*$ topology on the space of types $\mathbb{S}_{m,R}(\rT_{\cU})$ since this space is weak-$*$ compact, but not separable with respect to $d_W$; see \cite[Proposition 2.4.9]{AGKE2022} and also see \cite[\S 5.5]{GJNS2021} for an analogous result for non-commutative laws.  This provides a sharp contrast with the setting of classical probability where the weak-$*$ and Wasserstein topology agree for probability measures on a compact subset of $\bC^m$.  Thus, continuity properties for various versions of free entropy need to be considered separately for each of these topologies.
	
	Voiculescu's free entropy $\chi$ is upper semi-continuous on the set of non-commutative laws with the weak-$*$ topology \cite[Proposition 2.6]{VoiculescuFE2} (and hence also with respect to the Wasserstein topology).  Analogously, $\chi_{\full}^{\cU}$ is weak-$*$ upper semi-continuous on $\mathbb{S}_{m,R}(\rT_{\cU})$ \cite[Lemma 3.6]{JekelModelEntropy}.  The free entropy quantities $\delta_0$ and analogously $\delta_{\full}^{\cU}$ fail to be upper semi-continuous in general, even if we use the stronger Wasserstein topology rather than the weak-$*$ topology, because a tuple with $\delta_{\full}^{\cU}(\mathbf{X}) = 1$ is a limit of tuples with $\delta_{\full}^{\cU} = n$.  Similarly, $\Ent_{\full}^{\cU}$ fails to be upper semi-continuous, even though $\Ent_{\full}^{\cU}$ is the supremum of upper semi-continuous functions $\Ent_{\full,\varepsilon}^{\cU}$ for $\varepsilon > 0$.  However, surprisingly, $\Ent_{\full}^{\cU}$ turns out to be \emph{lower semi-continuous} with respect to $d_W$, and as a consequence, we can deduce that the property $\Ent_{\full}^{\cU} = \infty$ is generic in $(\mathbb{S}_{m,R}(\rT_{\cU}),d_W)$.
	
	\begin{proposition}[Topological properties of free entropy] \label{prop: topological properties}
		Fix a free ultrafilter $\cU$ on $\bN$.
		\begin{enumerate}[(1)]
			\item The metric entropy $\Ent_{\full}^{\cU}$ is lower semi-continuous on $(\mathbb{S}_{m,R}(\rT_{\cU}),d_W)$.
			\item $\{\mu \in \mathbb{S}_{m,R}(\rT_{\cU}): \Ent_{\full}^{\cU}(\mu) = \infty\}$ is a dense $G_\delta$ set in $(\mathbb{S}_{m,R}(\rT_{\cU}),d_W)$.
			\item The free entropy $\chi_{\full}^{\cU}$ is weak-$*$ upper semi-continuous on $\mathbb{S}_{m,R}(\rT_{\cU})$.
			\item $\{\mu \in \mathbb{S}_{m,R}(\rT_{\cU}): \chi_{\full}^{\cU}(\mu) = -\infty\}$ is a dense $G_\delta$ set both with respect to the weak-$*$ topology and the Wasserstein topology.
			\item $\{\mu \in \mathbb{S}_{m,R}(\rT_{\cU}): \chi_{\full}^{\cU}(\mu) > -\infty\}$ is dense in $(\mathbb{S}_{m,R}(\rT_{\cU}),d_W)$.
		\end{enumerate}
	\end{proposition}
	
	It is an interesting open question whether $\delta_{\full}^{\cU}$ has any such lower semi-continuity property, and what its generic behavior is in $(D_R^{\cQ})^m$.
	
	Proposition \ref{prop: topological properties} shows Wasserstein-generic types for $\rT_{\cU}$ will have $\chi_{\full}^{\cU} = -\infty$ and $\Ent_{\full}^{\cU} = +\infty$.  While the conditions $\chi > -\infty$ and $\Ent = +\infty$ can be used to prove many of the properties of von Neumann algebras such as absence of Cartan subalgebras and non-Gamma (see \cite[\S 1.2]{HJNS2021} for discussion), Proposition \ref{prop: topological properties} (2) and (4) give an indication that $\Ent = +\infty$ may apply to much broader families of examples than $\chi > -\infty$.
	
	Next, we turn our attention to the invariant multi-matrix ensembles associated to definable predicates.  If $\mathbf{X}^{(n)}$ has probability density $e^{-n^2 \varphi^{\bM_n}} / Z^{(n)}$ for some definable predicate $\varphi$, then any type which describes the large-$n$ limit should be the one that maximizes the entropy $\chi_{\full}^{\cU}$ minus the evaluation of $\mu$ on $\varphi$.  Although Gibbs types may not be unique in general, they are unique when $\varphi$ is strongly convex.  We focus on the strongly convex case for simplicity, and because this is the setting needed for our quasi-moment types later on.  For a type $\mu$ and a definable predicate $\varphi$, we denote the evaluation or dual pairing by $(\mu,\varphi)$.
	
	\begin{proposition}[Gibbs types for strongly convex definable predicates] \label{prop: convex Gibbs type}
		Let $c > 0$ and let $\varphi$ be a definable predicate such that $\varphi^{\cM}(\mathbf{x}) - \frac{c}{2} \norm{\mathbf{x}}_{L^2(\cM)}^2$ is convex for tracial von Neumann algebras $\cM$.  Then there exists a unique Gibbs type for $\varphi$, that is, a type $\mu$ that maximizes $\chi_{\full}^{\cU}(\mu) - (\mu,\varphi)$.
	\end{proposition}
	
	The Gibbs types associated to strongly convex definable predicates $\varphi$ satisfy a non-commutative analog of the Talagrand inequality, as a consequence of the Talagrand inequality for the associated random matrix models (see \S \ref{subsec: Talagrand}); this follows a similar method as \cite[Theorem 2.2]{HU2006}.
	
	Our last result is about moment types.  First, let us recall moment measures in the setting of classical probability.  Many probability measures on $\bR^m$ have a canonical realization as the pushforward by $\nabla \varphi$ of a Gibbs measure associated to some convex potential $\varphi$.  Specifically, Cordero-Erausquin and Klartag \cite{CEK2015moment} showed that, given $\mu \in \mathcal{P}(\bR^d)$ with finite expectation and barycenter zero and support not contained in any hyperplane, there is a unique lower semi-continuous convex $V: \bR \to (-\infty,+\infty]$, such that $V$ is infinite almost everywhere on $\partial \{ V < +\infty\}$, such that $\mu = (\nabla V)_* \mu_V$ where $d\mu_V(x) = (1/Z) e^{-V(x)}\,dx$.  An elegant approach to moment measures in terms of Wasserstein geometry was given by Santambrogio \cite{Santambrogio2016}:  The measure $\mu_V$ can be obtained as the maximizer of
	\[
	\nu \mapsto h(\nu) - C(\mu,\nu).
	\]
	The analogous construction for non-commutative laws was studied by Bahr and Boschert \cite{BahrBoschert2023}.  In the case of a single self-adjoint operator, they were able to follow Santambrogio's variational approach, and in the multivariate setting, they were able to handle the case where $V$ is a perturbation of a quadratic potential, by solving the Jacobian equation similarly to Guionnet and Shlyakhtenko's transport results for free Gibbs laws \cite{GS2014}.  In addition, Diez has taken two constructions of Fathi related to moment measures and adapted them to the free setting: \cite{Diez2024} gives a symmetrized version of Talagrand's inequality in the free setting analogously to \cite{Fathi2018}, and \cite{Diez2024Stein} describes the relationship between free moment measures and free Stein kernels analogously to \cite{Fathi2019Stein}.
	
	Here we study the analog of Santambrogio's variational problem on the type space $\mathbb{S}_{m}(\rT_{\cU})$.  As often happens, there are additional difficulties in the non-commutative setting.  For instance, fix a type $\mu$, and suppose we want to maximize
	\[
	\nu \mapsto \chi_{\full}^{\cU}(\nu) - C_{\full}(\mu,\nu)
	\]
	simply over $\mathbb{S}_{m,R}(\rT_{\cU})$ for some $R > 0$.  In the classical case, the existence of a maximizer would be immediate if the support is restricted to a compact set; this is because the entropy is upper semi-continuous and the optimal inner product $C(\mu,\nu)$ is continuous (in fact, a key part of Santambrogio's argument is to show semi-continuity of the entropy beyond the case of compact support \cite[\S 2, pp.\ 424-426]{Santambrogio2016}).  However, in the non-commutative setting, the $C_{\full}(\mu,\nu)$ is only upper semi-continuous with respect to the weak-$*$ topology, so $-C_{\full}(\mu,\nu)$ is lower semi-continuous, and thus as far as we know $\chi_{\full}^{\cU}(\nu) - C_{\full}(\mu,\nu)$ might not be weak-$*$ upper semi-continuous, even when we restrict to non-commutative random variables bounded by a constant $R$.  Of course, $C_{\full}(\mu,\nu)$ is continuous with respect to the \emph{Wasserstein} topology, but $\mathbb{S}_{m,R}(\rT_{\cU})$ is not compact in the Wasserstein topology, so again the existence of a maximizer is unclear.
	
	As suggested by Santambrogio's approach, we can use Monge-Kantorovich duality to write $-C_{\full}(\mu,\nu)$ as the supremum of $-(\mu,\varphi) - (\nu,\psi)$ for admissible pairs of convex definable predicates $\varphi$ and $\psi$, so we now are studying
	\[
	\sup_{\nu} \sup_{(\varphi,\psi)} \chi_{\full}^{\cU}(\mu) - (\mu,\varphi) - (\nu,\psi).
	\]
	Here of course, if $(\varphi,\psi)$ is fixed, then the maximizing $\nu$ (if it exists) is the Gibbs type associated to $\varphi$.  We would hope to obtain a maximizing $(\nu,\varphi,\psi)$ using compactness, but unlike equicontinuous and pointwise bounded functions on a compact subset of $\bR^m$, the set of definable predicates on the $R$-ball that satisfy a given modulus of continuity and pointwise bound is \emph{not} precompact.
	
	Nonetheless, we are able to obtain a maximizing $\nu$ if we first add a quadratic perturbation to $\varphi$ to make it strongly convex.  The Talagrand inequality then aids in obtaining convergence in Wasserstein distance of a sequence of almost maximizers for the perturbed problem.  We thus obtain the following result.
	
	\begin{theorem}[Quasi-moment types] \label{thm: quasi-moment type}
		Let $\mu \in \mathbb{S}_m(\rT_{\cU})$, and let $t > 0$.  Let $q(x_1,\dots,x_m) = \frac{1}{2} \sum_{j=1}^m \tr(x_j^*x_j)$ be the norm squared viewed as a formula for tracial von Neumann algebras.  Then there exists a unique type $\nu_t \in \mathbb{S}_m(\rT_{\cU})$ that maximizes
		\[
		\nu \mapsto \chi_{\full}^{\cU}(\nu) - C_{\full}(\mu,\nu) - t\, (\nu,q).
		\]
		Moreover, $\nu_t$ is the Gibbs type associated to $\varphi_t + tq$ for some convex definable predicate $\varphi_t$.
	\end{theorem}
	
	One can consider the quadratically perturbed problem as analogous to the original but with the background measure a Gaussian instead of Lebesgue measure.  However, it would still be of great interest to discover what happens when $t = 0$ in the above problem.
	
	As for classical moment measures, one would also like to show that $\mu$ is something like a pushforward of the maximizer $\nu$.  For non-commutative laws, the most general version of pushforward would be that $\mu$ is realized by some tuple in the von Neumann algebra generated by some $\mathbf{x}$ with law $\mu$.  In the setting of types, the von Neumann algebra generated by $\mathbf{x}$ is replaced by the definable closure (see \cite{JekelTypeCoupling}).  Thus we ask, more precisely, if $(\mathbf{x},\mathbf{y})$ is an optimal coupling of $(\mu,\nu)$ as above, then is $\mathbf{y}$ in the definable closure of $\mathbf{x}$?
	
	Generally, the answer is no, and the reason again has to do with the stark difference between the weak-$*$ and Wasserstein topologies in the non-commutative setting.  For each $t > 0$, the set of Gibbs types associated to $t$-strongly convex $\varphi$ is \emph{separable} with respect to Wasserstein distance; this follows from the Talagrand inequality and the separability of the space of definable predicates (see \S \ref{subsec: separability}  Hence, the set of types that can be realized as definable pushforwards of such Gibbs types is also separable.  However, as already mentioned, $\mathbb{S}_{m,R}(\rT_{\cU})$ is not separable with respect to Wasserstein distance \cite[Proposition 2.4.9]{AGKE2022}.  Hence, \emph{most} types cannot be realized as pushforwards of the Gibbs types for strongly convex definable predicates.
	
	\subsection{Organization}
	
	The rest of the paper is organized as follows:
	\begin{itemize}
		\item[\S \ref{sec: preliminaries}] covers preliminaries on random matrices and operator algebra (\S \ref{subsec: operator algebras}) and classical convex functions (\S \ref{subsec: classical convex}).
		\item[\S \ref{sec: types and optimal couplings}] explains the model-theoretic notions of formulas and types (\S \ref{subsec: model theory}), recalls the Wasserstein distance and Monge-Kantorovich duality for types in tracial von Neumann algebras (\S \ref{subsec: optimal couplings}), and then shows that there are Lipschitz definable functions transforming between the types along the interior of a geodesic (\S \ref{subsec: displacement interpolation}).
		\item[\S \ref{sec: entropy along geodesics}] reviews several notions of free entropy, proves Theorems \ref{thm: entropy along geodesics} about entropy and Wasserstein geodesics (\S \ref{subsec: entropy along geodesics}), establishes the topological properties of free entropy in Proposition \ref{prop: topological properties} (\S \ref{subsec: topological properties}), and also shows a counterexample to simultaneous convergence of Wasserstein distance and entropy in the setting of laws (\S \ref{subsec: counterexample}).
		\item[\S \ref{sec: Gibbs types}] shows the existence of Gibbs types for strongly convex definable predicates (\S \ref{subsec: Gibbs existence}), establishes the Talagrand inequality for them (\S \ref{subsec: Talagrand}), and deduces separability of the space of such Gibbs types (\S \ref{subsec: separability}).
		\item[\S \ref{sec: quasi-moment types}] proves Theorem \ref{thm: quasi-moment type} establishing the existence of quasi-moment types solving a version of Santambrogio's variational problem.
	\end{itemize}
	
	\section{Preliminaries} \label{sec: preliminaries}
	
	\subsection{Random matrices and operator algebras} \label{subsec: operator algebras}
	
	\subsubsection{Invariant random matrix models}
	
	Our results are motivated by the study of random multi-matrices, or random matrix $m$-tuples, or random elements of $\bM_n^m$.  Here we equip $\bM_n^m$ with the inner product
	\[
	\ip{\mathbf{X},\mathbf{Y}}_{\tr_n} = \sum_{j=1}^m \tr_n(X_j^*Y_j),
	\]
	where $\tr_n = (1/n) \Tr_n$ is the normalized trace; we also denote the corresponding normalized Hilbert-Schmidt norm by $\norm{\cdot}_{\tr_n}$.  As a complex inner-product space $\bM_n^m$ may be transformed by a linear isometry to $\bC^m$.  By the \emph{Lebesgue measure} on $\bM_n^m$, we mean the Lebesgue measure on $\bC^m$ transported by such an isometry, which is independent of the particular choice of isometry.  Note that many authors identify $\bM_n$ with $\bC^{n^2}$ entrywise to obtain Lebesgue measure, but our choice of Lebesgue measure differs by a factor depending on $n$ since we use $\tr_n$ rather than $\Tr_n$ to define the inner product.  See e.g.\ \cite[\S 3.3]{JekelPi2024} for further discussion of the normalization.  We will be concerned especially in \S \ref{sec: Gibbs types} with probability measures $\mu^{(n)}$ on $\bM_n^m$ given as
	\[
	d\mu^{(n)}(\mathbf{X}) = \frac{1}{Z^{(n)}} e^{-n^2 \varphi^{(n)}(\mathbf{X})}\,d\mathbf{X},
	\]
	where $d\mathbf{X}$ is Lebesgue measure, $\varphi^{(n)}: \bM_n^m \to \bR$ is an appropriate potential, and $Z^{(n)}$ is a normalizing constant.  Given a probability measure $\mu^{(n)}$ on $\bM_n^m$, there is a corresponding random variable $\mathbf{X}^{(n)}$ in $\bM_n^m$ whose distribution is $\mu^{(n)}$ meaning that $\mu^{(n)}(A) = P(\mathbf{X}^{(n)} \in A)$ for Borel $A \subseteq \bM_n$.
	
	Note that many works study the case where $\mathbf{X}^{(n)}$ is in the real subspace of self-adjoint matrices $(\bM_n)_{\sa}^m$. There is an isometry between $\bM_n^m$ and $(\bM_n)_{\sa}^{2m}$ given by
	\[
	(X_1,\dots,X_m) \mapsto \left( \frac{X_1 + X_1^*}{2}, \frac{X_1 - X_1^*}{2i}, \dots, \frac{X_m + X_m^*}{2}, \frac{X_m - X_m^*}{2i} \right),
	\]
	and thus many results that apply in the self-adjoint case also apply in the non-self-adjoint case after a simple change of notation.  We focus on the non-self-adjoint case for ease of applying the model-theoretic definitions in \S \ref{subsec: model theory}, though the results could also be adapted to the self-adjoint setting.
	
	\subsubsection{Tracial von Neumann algebras}
	
	The natural objects to describe the large-$n$ limits of $\bM_n$ are \emph{tracial von Neumann algebras}.  The algebra $\bM_n$ is generalized to an algebra $M$ of operators on a Hilbert space and $\tr_n$ is generalized to a linear functional $\tau: M \to \bC$ satisfying analogous properties.  Usually, we do not need to consider classically random elements of the von Neumann algebra since $\mathbb{E} \circ \tr_n$ can also be viewed as special case of a trace on a von Neumann algebra.  In this paper, a \emph{tracial von Neumann algebra}, or equivalently \emph{tracial $\mathrm{W}^*$-algebra} refers to a finite von Neumann algebras with a specified tracial state.  We recommend \cite{Ioana2023} for a quick introduction to the topic, as well as the following standard reference books \cite{KadisonRingroseI,Dixmier1969,Sakai1971,TakesakiI,Blackadar2006,Zhu1993}.
	
	The abstract definitions / characterizations of $\mathrm{C}^*$-algebras and $\mathrm{W}^*$-algebras are as follows.
	\begin{enumerate}[(1)]
		\item A (unital) algebra over $\bC$ is a unital ring $A$ with a unital inclusion map $\bC \to A$.
		\item A (unital) $*$-algebra is an algebra $A$ equipped with a conjugate linear involution $*$ such that $(ab)^* = b^* a^*$.
		\item A unital $\mathrm{C}^*$-algebra is a $*$-algebra $A$ equipped with a complete norm $\norm{\cdot}$ such that $\norm{ab} \leq \norm{a} \norm{b}$ and $\norm{a^*a} = \norm{a}^2$ for $a, b \in A$.
		\item A \emph{$\mathrm{W}^*$-algebra} is a $\mathrm{C}^*$-algebra $\cA$ that $\cA$ is isomorphic as a Banach space to the dual of some other Banach space (called its predual).
	\end{enumerate}
	The work of Sakai (see e.g.\ \cite{Sakai1971}) showed that the abstract definition given here for a $\mathrm{W}^*$-algebra is equivalent to several other definitions and notions.  Sakai also showed that the predual is unique (and hence so is the weak-$*$ topology on $\cA$).
	
	\begin{notation}
		A \emph{tracial $\mathrm{W}^*$-algebra} is a pair $(M,\tau)$ where $M$ is a $\mathrm{W}^*$-algebra and $\tau: M \to \bC$ is a faithful normal tracial state, that is, a linear map satisfying
		\begin{itemize}
			\item \emph{positivity}: $\tau(x^*x) \geq 0$ for all $x \in M$
			\item \emph{unitality}: $\tau(1) = 1$
			\item \emph{traciality}: $\tau(xy) = \tau(yx)$ for $x, y \in A$
			\item \emph{faithfulness}:  $\tau(x^*x) = 0$ implies $x = 0$ for $x \in A$.
			\item \emph{weak-$*$ continuity}:  $\tau: M \to \bC$ is weak-$*$ continuous.
		\end{itemize}
		We will often denote $\tau$ by $\tr^{\cM}$ for consistency with the model-theoretic notation introduced below.
	\end{notation}
	
	\begin{fact}
		If $\cM = (M,\tau)$ is a tracial $\mathrm{W}^*$-algebra, then $\re \ip{x,y}_{L^2(\cM)} := \tau(x^*y)$ defines an inner product on $\cM$.  The Hilbert space completion is denoted $L^2(\cM)$.  The map $\cM \to L^2(\cM)$ is injective because $\tau$ is faithful.  Moreover, for $r > 0$, the ball $D_r^{\cM} = \{x \in \cM: \norm{x} \leq r\}$ is a closed subset of $L^2(\cM)$.
	\end{fact}
	
	%\begin{fact} \label{fact: J}
	%	If $\cM = (M,\tau)$ is a tracial $\mathrm{W}^*$-algebra, then $\norm{x^*}_{L^2(\cM)} = \norm{x}_{L^2(\cM)}$.  In particular, the $*$-operation extends to an anti-linear isometry $J: L^2(\cM) \to L^2(\cM)$.  Thus, it makes sense to refer to ``self-adjoint'' elements of $L^2(\cM)$.
	%\end{fact}
	
	\begin{definition}
		A \emph{$*$-homomorphism} is a map between $*$-algebras that respects the addition, multiplication, and $*$-operations.  For tracial $\mathrm{W}^*$-algebras, a $*$-homomorphism $\rho: \cM \to \cN$ is said to be \emph{trace-preserving} if $\tr^{\cM}(\rho(x)) = \tr^{\cN}(x)$.
	\end{definition}
	
	\begin{fact} \label{fact: conditional expectation}
		A $*$-homomorphism between $\mathrm{C}^*$-algebras is automatically contractive with respect to the norm.  Moreover, if $\rho: \cM \to \cN$ is a trace-preserving $*$-homomorphism of tracial $\mathrm{W}^*$-algebras, then $\rho$ is also contractive with respect to the $L^2$-norm and hence extends uniquely to a contractive map $L^2(\cM) \to L^2(\cN)$.  Moreover, a trace-preserving $*$-homomorphism is automatically continuous with respect to the weak-$*$ topology.
	\end{fact}
	
	\begin{notation}
		If $\rho: \cM \to \cN$ is a trace-preserving $*$-homomorphism of tracial $\mathrm{W}^*$-algebras, we will also denote its extension $L^2(\cM) \to L^2(\cN)$ by $\rho$.  Furthermore, for a tuple $\mathbf{x} = (x_i)_{i \in I}$, we use the notation $\rho(\mathbf{x}) = (\rho(x_i))_{i \in I}$.
	\end{notation}
	
	\begin{fact}
		Let $\cM \subseteq \cN$ be a trace-preserving inclusion of tracial $\mathrm{W}^*$-algebras.  Let $E_{\cM}: L^2(\cN) \to L^2(\cM)$ be the orthogonal projection.  Then $E_{\cM}$ restricts to a map $\cN \to \cM$ that is contractive with respect to the operator norm.  Moreover, $E_{\cM}$ is the unique conditional expectation (positive $\cN$-$\cN$ bimodule map that restricts to the identity on $\cN$) that preserves the trace.
	\end{fact}
	
	\begin{notation}
		For a $\mathrm{W}^*$-algebra, or more generally, a unital $*$-algebra,
		\begin{enumerate}[(1)]
			\item $\cM_{\sa} = \{x \in \cM: x^* = x\}$ denotes the set of \emph{self-adjoints}.
			%		\item $L^2(\cM)_{\sa} = \{x \in L^2(\cM): J(x) = x\}$ similarly denotes the set of self-adjoints in $L^2(\cM)$.
			\item $U(\cM) = \{u \in \cM: u^*u = uu^* = 1\}$ denotes the set of \emph{unitaries}.
			\item $P(\cM) = \{ p \in \cM: p^* = p = p^2 \}$ denotes the set of \emph{projections}.
			\item $Z(\cM) = \{x \in \cM: \forall y \in \cM, xy = yx\}$ denotes the \emph{center}.
			\item For $A \subseteq \cM$, we write $A' \cap \cM = \{x \in \cM: \forall y \in A, xy = yx\}$ for the \emph{relative commutant} of $A$ in $\cM$.
		\end{enumerate}
	\end{notation}
	
	%\begin{fact}
	%	A tracial $\mathrm{W}^*$-algebra has trivial center $Z(\cM) = \bC 1$ if and only if for any two projections $p$ and $q$ with $\tr^{\cM}(p) = \tr^{\cM}(q)$, there exists a unitary $u \in U(\cM)$ such that $upu^* = q$.
	%\end{fact}
	
	\begin{notation}[Factors]
		A von Neumann algebra $M$ is said to be a \emph{factor} if $Z(M) = \bC$.  This terminology comes from its role in Murray and von Neumann's direct integral decomposition of general von Neumann algebras.
	\end{notation}
	
	\subsubsection{Non-commutative laws} \label{subsubsec: NC law}
	
	Next, we recall the notion of \emph{non-commutative laws}, which is a na{\"\i}ve analog of joint probability distribution in the non-commutative setting, and is based on non-commutative $*$-polynomial test functions.
	
	\begin{notation}[Non-commutative $*$-polynomials]
		$\bC^* \ip{x_1,\dots,x_m}$ denotes the free unital $*$-algebra generated by indeterminates $x_1$, \dots, $x_m$.  Equivalently, it is the free unital algebra generated by $x_1$, \dots, $x_{2m}$ equipped with the unique $*$-operation sending $x_j$ to $x_{j+m}^*$ for $j = 1$, \dots, $m$.  As a vector space, it has a basis given by $*$-monomials in $x_1$, \dots, $x_{2m}$.
	\end{notation}
	
	\begin{notation}[Non-commutative $*$-law of a tuple]
		Let $\cM = (M,\tau)$ be a tracial von Neumann algebra and let $\mathbf{y} \in \cM^m$.  Then there is a unique $*$-homomorphism $\ev_{\mathbf{x}}: \bC^*\ip{x_1,\dots,x_m}$ sending the indeterminate $x_j$ to the element $y_j \in \cM$.  The \emph{non-commutative $*$-law of $\mathbf{y}$} is the linear functional
		\[
		\law(\mathbf{y}) = \tau \circ \ev_{\mathbf{y}}: \bC^*\ip{x_1,\dots,x_m} \to \bC.
		\]
	\end{notation}
	
	\begin{notation}[Non-commutative $*$-laws in the abstract]
		In general, a \emph{non-commutative $*$-law} is a linear map $\lambda: \bC^*\ip{x_1,\dots,x_m} \to \bC$ satisfying the following properties:
		\begin{itemize}
			\item \emph{Unital:} $\lambda(1) = 1$.
			\item \emph{Positive:} $\lambda(p^*p) \geq 0$ for $p \in \bC^*\ip{x_1,\dots,x_m}$.
			\item \emph{Tracial:} $\lambda(pq) = \lambda(qp)$ for $p, q \in \bC^*\ip{x_1,\dots,x_m}$.
			\item \emph{Exponentially bounded:} There exists $R > 0$, such that $|\lambda(p)| \leq R^d$ when $p$ is a $*$-monomial of degree $d$.
		\end{itemize}
	\end{notation}
	
	Every abstract non-commutative $*$-law as described above can be realized as the law of some tuple.  This is proved in the self-adjoint case in \cite[Proposition 5.2.14]{AGZ2009}, and the general case follows by taking operator-valued real and imaginary parts.
	
	\begin{proposition}
		A linear functional $\lambda: \bC^*\ip{x_1,\dots,x_m} \to \bC$ is the law of some tuple $\mathbf{y} \in \cM^m$ with $\max_j \norm{y_j}_{\op} \leq R$ if and only if $\lambda$ is unital, positive, tracial, and exponential bounded using this given $R$.
	\end{proposition}
	
	\begin{notation} \label{not: space of laws}
		We denote by $\Sigma_{m,R}^*$ the space of non-commutative $*$-laws of $m$-tuples that are bounded by $R$ in operator norm.  We equip $\Sigma_{m,R}^*$ with the weak-$*$ topology.  We denote by $\Sigma_m^*$ the union of $\Sigma_{m,R}^*$ for $R > 0$, equipped with the inductive limit topology.
	\end{notation}
	
	It is well-known that $\Sigma_{m,R}^*$ is compact in the weak-$*$ topology.  We also remark that $\Sigma_{m,R}^*$ agrees with a model-theoretic object, the space of \emph{quantifier-free types} for tracial von Neumann algebras (see Remark \ref{rem: qf type}).  In this work, since we always use non-self-adjoint tuples, by default ``polynomial'' and ``law'' will refer to non-commutative $*$-polynomials and $*$-laws.
	
	\subsubsection{Ultrafilters and ultraproducts} \label{subsubsec: ultraproduct}
	
	Next, we explain the definitions of ultrafilters and ultraproducts, especially since these concepts are less familiar to many researchers in random matrix theory.  See also \cite[Appendix A]{BrownOzawa2008}, \cite[\S 2]{Capraro2010}, \cite[\S 5.4]{ADP}, \cite[\S 5.3]{GJNS2021}.
	
	One motivation for ultrafilters is the process of taking limits along a subsequence of a given sequence.  In various analysis arguments, one may take subsequences of subsequences, and so forth, in order to arrange that all given quantities have a limit.  An ultrafilter is intuitively like a sub-net of $\bN$ that is a maximally refined, so that it assigns a limit to \emph{all} given sequences simultaneously.  The precise definition is as follows.
	
	\begin{definition}[Ultrafilters on $\bN$]
		An \emph{ultrafilter} on $\bN$ is a collection $\cU$ of subsets of $\bN$ satisfying the following properties:
		\begin{itemize}
			\item \emph{Nontriviality:} $\varnothing \not \in \cU$.
			\item \emph{Finite intersection property:}  If $A, B \in \cU$, then there exists $C \subseteq A \cap B$ with $C \in \cU$.
			\item \emph{Directedness:}  If $A \in \cU$ and $A \subseteq B$, then $B \in \cU$.
			\item \emph{Maximality:}  For every $A \subseteq \bN$, we have either $A \in \cU$ or $A^c \in \cU$.
		\end{itemize}
	\end{definition}
	
	For each $n \in \bN$, there is an associated \emph{principal ultrafilter} $\{A \subseteq \bN: n \in A\}$.  We are primarily concerned with the \emph{non-principal} or \emph{free} ultrafilters on $\bN$.  If $\cU$ is an ultrafilter on $\bN$, we say that a set $A$ is \emph{$\cU$-large} if $A \in \cU$.  Similarly, a property is said to hold for \emph{$\cU$-many $n$} if the \emph{set} of $n$ that satisfies this property is an element of $\cU$.
	
	\begin{definition}[Ultralimits]
		Let $\cU$ be an ultrafilter on $\bN$, and let $(x_n)_{n \in \bN}$ be a sequence in some topological space.  We say that $\lim_{n \to \cU} x_n = x$ if for every neighborhood $O$ of $x$, we have $\{n: x_n \in O \} \in \cU$.
	\end{definition}
	
	\begin{fact}
		If $\cU$ is an ultrafilter on $\bN$ and $x_n$ is a sequence in a compact Hausdorff topological space, then $\lim_{n \to \cU} x_n$ exists and is unique.
	\end{fact}
	
	We remark that ultrafilters on $\bN$ are in bijection with characters on $\ell^\infty(\bN)$; here ``character'' refers to a multiplicative linear functional $\ell^\infty(\bN) \to \bC$, or equivalently a pure state on $\ell^\infty(\bN)$ as a $\mathrm{C}^*$-algebra.  For each ultrafilter $\cU$, the character $\varphi$ is given by $\varphi(f) = \lim_{n \to \cU} f(n)$.  Conversely, given a character $\varphi$, the corresponding ultrafilter $\cU = \{A \subseteq \bN: \varphi(\mathbf{1}_A) = 1\}$.  This in turn gives an identification between ultrafilters and points in the Stone-{\v C}ech compactification $\beta \bN$ of $\bN$, where for each $\omega \in \beta \bN$, the corresponding ultrafilter is $\cU = \{A \subseteq \bN: \omega \in \overline{A}\}$, or equivalently $A \in \cU$ if and only if $A$ is the intersection of $\bN$ with some neighborhood $O$ of $\omega$.
	
	Ultraproducts of tracial von Neumann algebras are defined as follows.  For $n \in \bN$, let $\cM_n = (M_n,\tau_n)$ be a sequence of tracial von Neumann algebras.  Let $\prod_{n \in \bN} M_n$ be the set of sequences $(x_n)_{n \in \bN}$ such that $\sup_n \norm{x_n}_{\op} < \infty$, which is a $\mathrm{C}^*$-algebra.  Let
	\[
	I_{\mathcal{U}} = \left\{(x_n)_{n \in \bN} \in \prod_{n \in \bN} M_n: \lim_{n \to {\mathcal{U}}} \norm{x_n}_{L^2(\cM_n)} = 0 \right\}.
	\]
	Using the non-commutative H\"older's inequality for $L^2$ and $L^\infty$, one can show that $I_{\mathcal{U}}$ is a two-sided ideal in $\prod_{n \in \bN} M_n$, and therefore, $\prod_{n \in \bN} M_n / I_{\mathcal{U}}$ is a $*$-algebra.  We denote by $[x_n]_{n \in \bN}$ the equivalence class in $\prod_{n \in \bN} M_n / I_{\mathcal{U}}$ of a sequence $(x_n)_{n \in \bN}$. Furthermore, we define a trace $\tau_{\cU}$ on $\prod_{n \in \bN} M_n / I_\mathcal{U}$ by
	\[
	\tau_{\cU}([x_n]_{n \in \bN}) = \lim_{n \to \cU} \tau_n(x_n);
	\]
	the limit exists because of the boundedness of the sequence and it is independent of the particular representative of the equivalence class $[x_n]_{n \in \bN}$ because $|\tau_n(x_n) - \tau(y_n)| \leq \norm{x_n - y_n}_{L^2(\cM_n)}$.  It turns out the pair $(\prod_{n \in \bN} A_n / I_{\mathcal{U}},\tau_{\mathcal{U}})$ is automatically a tracial $\mathrm{W}^*$-algebra; see \cite[Proposition 5.4.1]{ADP}.  We call the tracial von Neumann algebra $(\prod_{n \in \bN} M_n / I_{\mathcal{U}},\tau_{\mathcal{U}})$ the \emph{ultraproduct of $(\cM_n)_{n \in \bN}$ with respect to $\mathcal{U}$} and we denote it by
	\[
	\prod_{n \to \mathcal{U}} \cM_n := \left(\prod_{n \in \bN} M_n / I_{\mathcal{U}},\tau_{\mathcal{U}} \right).
	\]
	
	\subsubsection{The Connes embedding problem and Wasserstein distance}
	
	We also recall the notion of Connes-embeddability.  A tracial von Neumann algebra $\cM$ is said to be \emph{Connes-embeddable} if any $L^2$-separable subalgebra of $\cM$ admits a trace-preserving embedding into some ultraproduct $\prod_{n \to \cU} \bM_n$ of matrix algebras.  Such embeddability turns out to be independent of the choice of $\cU$.  It is also equivalently phrased in terms of embeddings into the ultrapower of the hyperfinite $\mathrm{II}_1$ factor $\mathcal{R}$, which is not needed for this work.  The Connes embedding problem asks whether every tracial von Neumann algebra has the property of Connes-embeddability, and a negative solution has been announced in \cite{JNVWY2020}.
	
	We also recall from \cite[Lemma 5.10]{GJNS2021} that $\cM$ is generated by $\mathbf{x} = (x_1,\dots,x_m)$, then $\cM$ is Connes-embeddable if and only if $\law(\mathbf{x})$ can be approximated in the weak-$*$ topology on $\Sigma_{m,R}^*$ by the laws of matrix tuples $\mathbf{X}^{(n)} \in \bM_n^m$.  It was shown in \cite[\S 5.3]{GJNS2021} that a negative solution to the Connes embedding problem presents an obstruction to relating the Wasserstein distance of random matrix models with the Biane--Voiculescu--Wasserstein distance of \cite{BV2001}.  Given $\mu, \nu \in \Sigma_{m,R}^*$, the Biane--Voiculescu--Wasserstein distance is defined by
	\[
	d_W(\mu,\nu) = \inf \{ \norm{\mathbf{x} - \mathbf{y}}_{L^2(\cM)^m}: \law(\mathbf{x}) = \mu, \law(\mathbf{y}) = \nu, \, \cM \text{ tracial von Neumann algebra} \}.
	\]
	Here the tracial von Neumann algebra $\cM$ is allowed to vary.  In the case when $\mu$ and $\nu$ are laws arising from Connes-embeddable tracial von Neumann algebras, we also define $d_{W,\CEP}(\mu,\nu)$ by the same formula but with $\cM$ restricted to Connes-embeddable tracial von Neumann algebras.\footnote{Although we sometimes refer to $d_{W,\CEP}$ as a Wasserstein distance, we do not know whether $d_{W,\CEP}$ satisfies the triangle inequality since the most natural method of proving this would require knowing that amalgamated free products preserve Connes embeddability.}
	
	Then \cite[\S 5.3]{GJNS2021} showed that based on a negative solution to the Connes embedding problem, $d_{W}$ can be strictly less than $d_{W,\CEP}$ even when $\mu$ and $\nu$ are the laws of tuples from \emph{matrix algebras}.  Hence, $d_W$ has no hope of giving the large-$n$ limit of the Wasserstein distances for multi-matrix models in general, and hence our attention should be given to $d_{W,\CEP}$.  We will show in \S \ref{subsec: counterexample} that even $d_{W,\CEP}$ cannot give a good description of the large-$n$ limit of the Wasserstein distance of matrix models in a way that is compatible with the correct limiting behavior of free entropy.
	
	\subsection{Classical Monge-Kantorovich duality and convex functions} \label{subsec: classical convex}
	
	This section records some elementary properties of convex functions on inner-product spaces, Legendre transforms, and the Hopf-Lax semigroup.  These will be used to study the convex functions associated to the random variables $\mathbf{x}_s$ and $\mathbf{x}_t$ along a Wasserstein geodesic.  As motivation, we first recall the classical Monge-Kantorovich duality for the $L^2$-Wasserstein distance.
	
	\subsubsection{Classical Monge-Kantorovich duality}
	
	Let $\mathcal{P}_2(\bR^m)$ be the set of probability measures on $\bR^m$ with finite second moment, and for $\mu, \nu \in \mathcal{P}_2(\bR^m)$, write
	\[
	C(\mu,\nu) = \sup \{ \ip{X,Y}_{L^2(\Omega;\bR^m)}: X \sim \mu, Y \sim \nu \text{ random variables on probability space } \Omega\}.
	\]
	This is related to the $L^2$-Wasserstein distance by the formula
	\[
	d_W(\mu,\nu)^2 = \int_{\bR^m} |x|^2\,d\mu(x) + \int_{\bR^m} |y|^2\,d\mu(y) - 2 C(\mu,\nu).
	\]
	Monge-Kantorovich duality allows us to express $C(\mu,\nu)$ through a dual optimization problem as follows.  (This is a special case of a much more general theory, but we focus on the $L^2$ case because we use the Hilbert space structure.)
	
	\begin{theorem}[{Classical Monge-Kantorovich duality; see e.g.\ \cite[Theorem 5.10, Case 5.17]{Villani2008}}] \label{thm: classical MK duality}
		Let $\mu, \nu \in \mathcal{P}_2(\bR^m)$.  Then
		\begin{multline*}
			C(\mu,\nu) = \inf \biggl\{\int_{\bR^m} \varphi\,d\mu + \int_{\bR^m} \psi\,d\nu: \\ \varphi, \psi \text{ convex on } \bR^m \text{ with } \varphi(x) + \psi(y) \geq \ip{x,y} \text{ for } x, y \in \bR^m  \biggr\},
		\end{multline*}
		and there exist convex functions $\varphi$ and $\psi$ that achieve the infimum.
	\end{theorem}
	
	Note that if $X$ and $Y$ are random variables with distribution $\mu$ and $\nu$ that are optimally coupled, then $\ip{X,Y} \leq \varphi(X) + \psi(Y)$, and also
	\[
	\mathbb{E} \ip{X,Y} = C(\mu,\nu) = \mathbb{E} \varphi(X) + \mathbb{E} \psi(Y).
	\]
	Therefore, equality $\ip{X,Y} = \varphi(X) + \psi(Y)$ must hold almost surely.  This relates closely to the notion of subdifferentials of convex functions, which we recall briefly here.
	
	Let $H$ be a real inner-product space and let $\varphi: H \to (-\infty,\infty]$.  We say that $y \in \underline{\nabla} \varphi(x)$ if for all $x' \in H$,
	\[
	\varphi(x') \geq \varphi(x) + \ip{y,x' - x}_H + o(\norm{x'-x}_H).
	\]
	Symmetrically, we say that $y \in \overline{\nabla} \varphi(x)$ if the same relation holds with $\leq$.  It is well known (see e.g.\ \cite[Fact 4.3]{JekelTypeCoupling}) that if $\varphi: H \to (-\infty,\infty]$ is convex, then $y \in \underline{\nabla} \varphi(x)$ if and only if
	\[
	\varphi(x') \geq \varphi(x) + \ip{y,x'-x} \text{ for } x' \in H,
	\]
	or in other words, the $o(\norm{x' - x}_H)$ term becomes unnecessary.
	
	Now suppose that $\varphi$ and $\psi$ are given satisfying $\ip{x,y} \leq \varphi(x) + \psi(y)$ everywhere, and consider a point $(x,y)$ where equality is achieved (as happens for $(X,Y)$ almost surely above).  Then for $x' \in H$,
	\[
	\varphi(x') \geq \ip{x',y}_H - \psi(y) = \ip{x',y}_H - \ip{x,y}_H + \varphi(x) = \varphi(x) + \ip{y,x'-x}_H,
	\]
	and hence $y \in \underline{\nabla} \varphi(x)$.  Symmetrically, $x \in \underline{\nabla} \psi(y)$.  Hence for the convex functions $\varphi$ and $\psi$ associated to an optimal coupling as in Theorem \ref{thm: classical MK duality}, we have $Y \in \underline{\nabla} \varphi(X)$ and $X \in \underline{\nabla} \psi(Y)$ almost surely.

	Another viewpoint on this relationship concerns the Legendre transforms.  For a real inner-product space $H$ and $\varphi: H \to [-\infty,\infty]$, the \emph{Legendre transform} is
	\[
	\mathcal{L}\varphi(y) = \sup_{x \in H} \left[ \ip{x,y}_H - \varphi(x) \right].
	\]
	The function $\mathcal{L}\varphi$ is always convex, being a supremum of affine functions.  Moreover, the inequality $\varphi(x) + \psi(y) \geq \ip{x,y}$ holds for all $x$ if and only if $\psi(y) \geq \mathcal{L}\varphi(x)$; it follows that in Theorem \ref{thm: classical MK duality}, we can always replace $\psi$ by $\mathcal{L}\varphi$ and the pair $(\varphi,\mathcal{L}\varphi)$ must still be optimal.

	%We also have $\varphi(x) + \mathcal{L} \varphi(y) \geq \ip{x,y}$ with equality if and only if $x$ achieves the supremum in the definition of $\mathcal{L}\varphi(y)$, which in turn is equivalent to $x \in \underline{\nabla} \mathcal{L}\varphi(y)$.  And if $\varphi$ is convex, this is equivalent to 
	
	These convex functions is very useful for studying the displacement interpolation $(1-t)X + tY$ for $t \in (0,1)$ when $(X,Y)$ is an optimal coupling, in both the classical and the non-commutative setting.  In fact, we can show that if $t \in [0,1]$ and $s \in (0,1)$, then $X_t$ can be expressed as a Lipschitz function of $X_s$ (which we will use later to analyze the free entropy and free entropy dimension in the non-commutative setting).  To prove this, we want to show that $(X_s,X_t)$ is an optimal coupling and find associated convex functions $\varphi_{s,t}$ and $\psi_{s,t}$.  We will show that $\varphi_{s,t}$ is differentiable and its gradient is a Lipschitz function, hence $\underline{\nabla} \varphi_{s,t}$ reduces to a single point $\{\nabla \varphi_{s,t}(x)\}$, and we have $X_t = \nabla \varphi_{s,t}(X_s)$.
	
	How do we obtain the functions $\varphi_{s,t}$ and $\psi_{s,t}$?  For simplicity, suppose that $s = 0$, and assume that $\psi = \mathcal{L} \varphi$.  Note that $X_t = (1-t)X + tY \in \underline{\nabla}((1-t) q + t \varphi)$ where $q(x) = (1/2) |x|^2$.  Hence, we can take $\varphi_{0,t} = (1-t)q + t \varphi$.  Then take $\psi_{s,t}$ to be the Legendre transform of $\varphi_{0,t}$.  The key relation is that
	\[
	\mathcal{L}[\varphi + tq](y) = \inf_{y'} \left[ \mathcal{L}\varphi(y') + \frac{1}{2t} |y' - y|^2 \right],
	\]
	which together with rescalings allows us to compute $\psi_{0,t}$.  The function on the right-hand side is an inf-convolution of $\mathcal{L}\varphi$ with a quadratic function, or the application of the Hopf-Lax semigroup to $\mathcal{L} \varphi$.  We will show that $\mathcal{L}[\varphi + tq]$ has Lipschitz gradient by analyzing the optimization problem on the right-hand side, and showing that it defines a semi-concave function.
	
	\subsubsection{Classical results from convex analysis}
	
	Here we recall the definition of semi-concavity, and the dual notion of strong convexity, and show the properties of the Hopf-Lax semigroup that we need for this work.  We work on a general real inner-product space $H$; later, these statements will be applied with $H = \mathcal{M}^m$ for some tracial von Neumann algebra with the real inner product given by
	\[
	\re \ip{\mathbf{x},\mathbf{y}}_{L^2(\cM)^m} = \sum_{j=1}^m \re \tr^{\cM}(x_j^*y_j).
	\]
	
	\begin{definition}
		Let $H$ be a real inner-product space and $\varphi: H \to [-\infty,\infty]$.
		\begin{itemize}
			\item For $c > 0$, we say $\varphi$ is $c$-strongly convex if $\varphi(x) - (c/2) \norm{x}_H^2$ is convex.
			\item For $c > 0$, we say that $\varphi$ is $c$-semiconcave if $\varphi(x) - (c/2) \norm{x}_H^2$ is concave.
		\end{itemize}
	\end{definition}
	
	The following characterizations are well-known and follow from direct computation with inner products.
	
	\begin{fact}
		Let $H$ be a real inner-product space and let $\varphi: H \to [-\infty,\infty]$.
		\begin{itemize}
			\item $\varphi$ is $c$-strongly convex if and only if for $x, y \in H$ and $\alpha \in [0,1]$, we have
			\[
			\varphi((1-\alpha)x + \alpha y) \leq (1 - \alpha) \varphi(x) + \alpha \varphi(y) - \frac{c}{2} \alpha (1 - \alpha) \norm{x - y}_H^2.
			\]
			\item $\varphi$ is $c$-semiconcave if and only if for $x, y \in H$ and $\alpha \in [0,1]$, we have
			\[
			\varphi((1-\alpha)x + \alpha y) \geq (1 - \alpha) \varphi(x) + \alpha \varphi(y) - \frac{c}{2} \alpha (1 - \alpha) \norm{x - y}_H^2.
			\]
		\end{itemize}
	\end{fact}
	
	The next fact concerns the semi-concave regularization of functions on an inner product space by the Hopf-Lax semigroup.
	
	\begin{fact} \label{lem: inf-convolution convexity}
		Let $H$ be a real inner-product space and let $\varphi: H \to [-\infty,\infty]$.  Let $t, u > 0$.  Let
		\[
		\varphi_t(x) = \inf_y \left[ \varphi(y) + \frac{1}{2t} \norm{x - y}_H^2 \right].
		\]
		\begin{itemize}
			\item $\varphi_t$ is $1/t$-semiconcave.
			\item If $\varphi$ is $1/u$-semiconcave, then $\varphi_t$ is $1 / (u + t)$-semiconcave.
			\item If $\varphi$ is convex, then $\varphi_t$ is convex.
			\item If $\varphi$ is $1/u$ strongly convex, then $\varphi_t$ is $1 / (u+t)$-strongly convex.
		\end{itemize}
	\end{fact}
	
	\begin{proof}
		(1) Note that $\varphi(y) + \frac{1}{2t} \norm{x - y}_H^2$ is a $1/t$-semiconcave function of $x$.  Since $\varphi_t$ is the infimum of this collection, it is also $1/t$-semiconcave.
		
		(2) Fix two points $x$ and $x'$ and fix $\alpha \in [0,1]$.  Let $x_\alpha = (1 - \alpha)x + \alpha x'$.  Let $y_\alpha \in H$ be a candidate for the infimum defining $\varphi_t(x_\alpha)$.  Let
		\[
		y = y_\alpha + \alpha u(t + u)^{-1} (x - x'), \qquad y' = y_\alpha + (1 - \alpha) u(t + u)^{-1} (x' - x).
		\]
		Observe that
		\[
		y_\alpha = (1 - \alpha) y + \alpha y'.
		\]
		Now
		\begin{align*}
			\varphi(y_\alpha) + \norm{x_\alpha - y_\alpha}_H^2 &= \varphi((1-\alpha)y + \alpha y') + \frac{1}{2t} \norm{(1 - \alpha)x + \alpha x' - (1 - \alpha)y - \alpha y'}_H^2 \\
			&\geq (1 - \alpha) \varphi(y) + \alpha \varphi(y') - \frac{1}{2u} \alpha(1 - \alpha) \norm{y - y'}_H^2 \\
			&\quad +  \frac{1}{2t} (1 - \alpha) \norm{x - y}_H^2 
			+ \frac{1}{2t} \alpha \norm{x' - y'}_H^2 - \frac{1}{2t} \alpha(1 - \alpha) \norm{(x - y) - (x' - y')}_H^2 \\
			&\geq (1 - \alpha)\varphi_t(x) + \alpha \varphi_t(x') - \frac{1}{2} \alpha(1 - \alpha) \left( u^{-1} \norm{y - y'}_H^2 + t^{-1} \norm{x - y - x' + y'}_H^2 \right).
		\end{align*}
		Observe that
		\[
		y - y' = \alpha u(t + u)^{-1} (x - x') - (1 - \alpha) u(t + u)^{-1} (x' - x) = u(t+u)^{-1} (x - x').
		\]
		Therefore,
		\begin{align*}
			u^{-1} \norm{y - y'}_H^2 + t^{-1} \norm{x - y - x' + y'}_H^2 &= u^{-1} [u(t+u)^{-1} \norm{x - x'}_H]^2 + t^{-1}[(1 - u(t+u)^{-1}) \norm{x - x'}_H]^2 \\
			&= \left( u^{-1} [u(t + u)^{-1}]^2 + t^{-1}[t(t + u)^{-1}]^2 \right) \norm{x - x'}_H^2 \\
			&= \left( u(t + u)^{-2} + t(t + u)^{-2} \right) \norm{x - x'}_H^2 \\
			&= (u + t)^{-1} \norm{x - x'}_H^2.
		\end{align*}
		Therefore,
		\[
		\varphi_t(x_\alpha) \geq (1 - \alpha) \varphi_t(x) + \alpha \varphi_t(x') - \frac{1}{2(u+t)} \alpha(1 - \alpha) \norm{x - x'}_H^2
		\]
		as desired.
		
		(3) We consider (3) as a special case of (4) by taking $u = +\infty$ in the argument for (4) below.
		
		(4) Fix two points $x$ and $x'$ and fix $\alpha \in [0,1]$.  Let $y, y' \in H$.  Then
		\begin{align*}
			\varphi_t((1-\alpha)x + \alpha x') &\leq \varphi((1 - \alpha)y + \alpha y') + \frac{1}{2t} \norm{(1 - \alpha)x + \alpha x' - (1 - \alpha)y - \alpha y' }_H^2 \\
			&\leq (1 - \alpha) \varphi(y) + \alpha \varphi(y') - \frac{1}{2u} \alpha (1 - \alpha) \norm{y - y'}_H^2 \\
			&\quad + \frac{1}{2t} (1 - \alpha) \norm{x - y}_H^2 
			+ \frac{1}{2t} \alpha \norm{x' - y'}_H^2 - \frac{1}{2t} \alpha(1 - \alpha) \norm{(x - y) - (x' - y')}_H^2.
		\end{align*}
		For $h, k \in H$, we have when $u \in (0,+\infty)$ that
		\begin{align*}
			\frac{1}{2u} \norm{h}_H^2 & + \frac{1}{2t} \norm{h - k}_H^2 \\ &= \frac{1}{2} \left( u^{-1} + t^{-1} \right) \norm{h}_H^2 - t^{-1} \ip{h,k}_H + \frac{1}{2} t^{-1} \norm{k}_H^2 \\
			&= \frac{1}{2} \norm{(u^{-1} + t^{-1})^{1/2} h - t^{-1}(u^{-1} + t^{-1})^{-1/2} k}_H^2 - \frac{1}{2} t^{-2} (u^{-1} + t^{-1})^{-1} \norm{k}_H^2 +  \frac{1}{2} t^{-1} \norm{k}_H^2 \\
			&\geq \frac{1}{2} \left( -\frac{u}{t(t + u)} + \frac{t + u}{t(t + u)} \right) \norm{k}_H^2 = \frac{1}{2(t + u)} \norm{k}_H^2;
		\end{align*}
		note that when $u = \infty$, the overall inequality reduces to $(1/2t) \norm{h - k}_H^2 \geq 0$ which is trivially true. Applying this with $ h = y' - y$ and $k = x' - x$, we get
		\[
		\varphi_t((1-\alpha)x + \alpha x') \leq (1 - \alpha) \left[ \varphi(y) + \frac{1}{2t} \norm{x - y}_H^2 \right] + \alpha \left[ \varphi(y') + \frac{1}{2t} \norm{x' - y'}_H^2 \right] - \frac{1}{2(u+t)} \alpha (1 - \alpha) \norm{x - x'}_H^2,
		\]
		and since $y$ and $y$' were arbitrary,
		\[
		\varphi_t((1-\alpha)x + \alpha x') \leq (1 - \alpha)\varphi_t(x) + \alpha \varphi_t(x') + \frac{1}{2(u+t)} \alpha(1 - \alpha) \norm{x - x'}_H^2.  \qedhere
		\]
	\end{proof}
	
	In Monge-Kantorovich duality, we will be concerned with pairs of convex functions $\varphi$, $\psi$ with $\varphi(x) + \psi(y) \geq \ip{x,y}_H$ as well as with the cases where equality is achieved for some particular $x$ and $y$.  In order to study the interpolation $x_t = (1-t)x + ty$, we also want to construct associated pairs of convex functions $\varphi_{s,t}$ and $\psi_{s,t}$ which satisfy similar properties with respect to $(x_s,x_t)$ for $0 \leq s \leq t \leq 1$.  The construction of such functions is well-known in optimal transport theory.  Important for our applications is that the functions $\varphi_{s,t}$ and $\psi_{s,t}$ are uniformly convex and semiconcave when $s,t \in (0,1)$.  We state this result here for a general real inner-product space.

	\begin{proposition} \label{prop: abstract interpolation 2}
		Let $H$ be a real inner-product space and let $\varphi$ and $\psi$ be functions $H \to (-\infty,\infty]$ satisfying $\varphi(x) + \psi(y) \geq \ip{x,y}$.  For $0 \leq s \leq t \leq 1$, let
		\begin{align*}
			\varphi_{s,t}(x) &= \inf_{x' \in H} \left[ \frac{t}{2s} \norm{x}_H^2 - \frac{t-s}{s} \ip{x,x'}_H + \frac{(t-s)(1-s)}{2s} \norm{x'}^2 + (t-s) \varphi(x') \right] \text{ when } s > 0, \\
			\varphi_{0,t}(x) &= \frac{1-t}{2} \norm{x}_H^2 + t \varphi(x),
		\end{align*}
		and
		\begin{align*}
			\psi_{s,t}(y) &= \inf_{y' \in H} \left[ \frac{1-s}{2(1-t)} \norm{y}_H^2 - \frac{t-s}{1-t} \ip{y,y'}_H + \frac{(t-s)t}{2(1-t)} \norm{y'}_H^2 + (t-s) \psi(y') \right] \text{ when } t < 1 \\
			\psi_{s,1}(y) &= \frac{s}{2} \norm{y}_H^2 + (1-s) \psi(y).
		\end{align*}
		Then
		\begin{enumerate}
			\item $\varphi_{s,t}(x) + \psi_{s,t}(y) \geq \ip{x,y}_H$.
			\item $\varphi_{s,t}$ is $t/s$-semiconcave for $s > 0$ and $(1-t) / (1-s)$-strongly convex for $t < 1$.
			\item $\psi_{s,t}$ is $(1-s)/(1-t)$-semiconcave for $t < 1$ and $s/t$-strongly convex for $s > 0$.
			\item Suppose $\varphi(x_0) + \psi(x_1) = \ip{x_0,x_1}$, and let $x_t = (1-t)x_0+ tx_1$.  Then $\varphi_{s,t}(x_s) + \psi_{s,t}(x_t) = \ip{x_s,x_t}$.
		\end{enumerate}
	\end{proposition}
	
	\begin{proof}
		(1) We consider first the generic case when $0 < s < t < 1$. Fix $x$ and $y$ in $H$. For every $x'$, $y'$, we have $\varphi(x') + \psi(y') \geq \ip{x',y'}_H$, and thus
		\begin{align*}
			&\frac{t}{2s} \norm{x}_H^2 - \frac{t-s}{s} \ip{x,x'}_H + \frac{(t-s)(1-s)}{2s} \norm{x'}_H^2 + (t-s) \varphi(x') \\
			+&\frac{1-s}{2(1-t)} \norm{y}_H^2 - \frac{t-s}{1-t} \ip{y,y'}_H + \frac{(t-s)t}{2(1-t)} \norm{y'}_H^2 + (t-s) \psi(y')
		\end{align*}
		is bounded below by the function
		\begin{align*}
			f(x',y') &= \frac{t}{2s} \norm{x}_H^2 - \frac{t-s}{s} \ip{x,x'}_H + \frac{(t-s)(1-s)}{2s} \norm{x'}^2 \\
			&+ \frac{1-s}{2(1-t)} \norm{y}_H^2 - \frac{t-s}{1-t} \ip{y,y'}_H + \frac{(t-s)t}{2(1-t)} \norm{y'}^2 + (t-s) \ip{x',y'}_H.
		\end{align*}
		Write
		\[
		T = \begin{bmatrix} \frac{1-s}{s} I & I \\
			I & \frac{t}{1-t}I
		\end{bmatrix} = (t-s) \begin{bmatrix} \frac{1}{s} I & 0 \\ 0 & \frac{1}{1-t} I \end{bmatrix} + \begin{bmatrix} \frac{1-t}{s} I & I \\
			I & \frac{s}{1-t} I
		\end{bmatrix},
		\]
		which is positive-definite.  Note that
		\begin{align*}
			f(x',y') = \frac{t}{2s} \norm{x}_H^2 + \frac{1-s}{2(1-t)} \norm{y}_H^2 - (t-s) \ip{s^{-1}x \oplus (1-t)^{-1}y,x'\oplus y'}_{H \oplus H} \\ + (t-s) \frac{1}{2} \ip{T(x' \oplus y'),x'\oplus y'}_{H \oplus H}.
		\end{align*}
		Hence, the function is bounded below and the minimum is achieved when $x' \oplus y' = T^{-1}(s^{-1}x \oplus (1-t)^{-1}y)$, and it is
		\[
		\frac{1}{2s} \norm{x}_H^2 + \frac{1}{2(1-t)} \norm{y}_H^2 - \frac{1}{2} (t-s) \ip{s^{-1}x \oplus (1-t)^{-1}y, T^{-1}(s^{-1}x \oplus (1-t)^{-1}y)}_{H \oplus H}
		\]
		Note
		\begin{align*}
			T^{-1} &= \left( \frac{1-s}{s} \frac{t}{1-t} - 1 \right)^{-1} \begin{bmatrix} \frac{t}{1-t} I & -I \\
				-I & \frac{1-s}{s}I
			\end{bmatrix} \\
			&= \left( (1-s)t - s(1-t) \right)^{-1} \begin{bmatrix} st I & -s(1-t)I \\
				-s(1-t)I & (1-s)(1-t)I
			\end{bmatrix} \\
			&= \frac{1}{t-s} \begin{bmatrix} st I & -s(1-t)I \\
				-s(1-t)I & (1-s)(1-t)I
			\end{bmatrix},
		\end{align*}
		so 
		\[
		(t-s) T^{-1}(s^{-1}x \oplus (1-t)^{-1}y) = \left(  tx - sy \right) \oplus \left( (1-s) y - (1-t)x \right),
		\]
		and taking the inner product with $s^{-1}x \oplus (1-t)^{-1}y$ yields
		\[
		ts^{-1} \norm{x}_H^2 + (1-s)(1-t)^{-1} \norm{y}_H^2 - 2 \ip{x,y}_H,
		\]
		and plugging this into the equations above shows that $f(x',y')$ reduces to $\ip{x,y}_H$ at the minimizer $x' \oplus y'$.  Hence,
		\[
		\varphi_{s,t}(x) + \psi_{s,t}(y) \geq \inf_{x',y' \in H} f(x',y') \geq \ip{x,y}_H.
		\]
		In the case that $s = t$, then $\varphi_{s,t} = \psi_{s,t} = q$.  In the case that $s = 0$ and $t = 1$, $\varphi_{s,t} = \varphi$ and $\psi_{s,t} = \psi$, so the result is trivial.  The remaining cases when $0 < s < t = 1$, or when $0 = s < t < 1$, can be handled by a similar argument as above, but with the minimization problem only occurring over one of the variables $x'$ or $y'$.
		
		(2) When $s > 0$, then $\varphi_{s,t}$ is the infimum of a family of $t/s$-semiconcave functions of $x$, hence is $t/s$-semiconcave.   For the second part, if $s = 0$, then $\varphi_{s,t}$ is $1-t$-uniformly convex by inspection.  For $s \leq t < 1$, observe also that
		\[
		\varphi_{s,t}(x) = \left( \frac{t}{2s} - \frac{t-s}{2s(1-s)} \right) \norm{x}_H^2 + \frac{t-s}{1-s} \inf_{x' \in H} \left[ \frac{1}{2s} \norm{x - (1-s)x'}_H^2  + (1-s) \varphi(x') \right].
		\]
		The infimum on the right-hand side is convex by Fact \ref{lem: inf-convolution convexity} applied to the function $(1-s) \varphi((1-s)^{-1}x'')$ via the substitution $x'' = (1-s)x'$.  Also,
		\[
		\frac{t}{s} - \frac{t-s}{s(1-s)} = \frac{t(1-s) - (t-s)}{s(1-s)} = \frac{1-t}{(1-s)},
		\]
		so that $\varphi_{s,t}$ is $(1-t)/(1-s)$-strongly convex.
		
		(3) The proof is symmetrical to (2).
		
		(4) Consider the generic case where $0 < s < t < 1$.  In the definition of $\varphi_{s,t}(x_s)$, we plug in $x_s$ for $x$ and $x_0$ for $x'$, the candidate for the infimum.  Similarly, in the definition of $\psi_{s,t}$, take $y = x_t$ and $y' = x_1$.  Since $(t-s) \varphi(x_0) + (t-s) \varphi(x_1) = (t-s) \ip{x_0,x_1}_H$, we obtain
		\[
		\varphi_{s,t}(x_s) + \psi_{s,t}(x_t) \leq f(x_0,x_1),
		\]
		where $f$ is the function from the proof of (1).  Note that
		\begin{align*}
			(t-s) T^{-1}(s^{-1}x_s \oplus (1-t)^{-1}x_t) &= (tx_s - sx_t) \oplus ((1-s)x_t - (1-t)x_s) \\
			&= [t(1-s) x_0 + ts x_1 - s(1-t) x_0 - st x_1] \\
			&\quad \oplus [(1-s)(1-t)x_0 + (1-s)tx_1 - (1-t)(1-s)x_0 + (1-t)sx_1] \\
			&= (t-s)[x_0 \oplus x_1].
		\end{align*}
		Thus, the $x_0 \oplus x_1 = T^{-1}[x_s \oplus x_t]$ is the minimizer of $f$ and the minimum value reduces to $\ip{x_s,x_t}_H$.  Therefore, combining this with the conclusion of (1),
		\[
		\ip{x_s,x_t}_H \leq \varphi_{s,t}(x_s) + \psi_{s,t}(x_t) \leq \ip{x_s,x_t}_H,
		\]
		hence equality is achieved.  The other cases of $s$ and $t$ are handled by similar considerations as we noted in the proof of (1).
	\end{proof}

	\section{Types and Optimal couplings} \label{sec: types and optimal couplings}
	
	\subsection{Background on model theory for von Neumann algebras} \label{subsec: model theory}
	
	In this paper, we need the notions of formulas, definable predicates, theories, and types for tracial von Neumann algebras.  However, to minimize the model-theoretic background needed, we will not present the proper and general definitions of these concepts, but rather minimal working definitions for the setting of tracial von Neumann algebras.  For general background on continuous model theory, see \cite{BYBHU2008,Hart2023,GH2023}.  Exposition for the non-commutative probability setting is also given in the precursors to the present work \cite{JekelCoveringEntropy,JekelModelEntropy,JekelTypeCoupling}.
	
	\begin{definition}[Formulas]
		\emph{Formulas} for tracial von Neumann algebras are formal expressions in free  variables $(x_i)_{i \in I}$ defined recursively as follows:
		\begin{itemize}
			\item A \emph{basic formula} is an expression of the form $\re \tr(p(x_1,\dots,x_n))$ where $p$ is a non-commutative $*$-polynomial.\footnote{Technically, since formulas are defined \emph{before} one states the algebra axioms, $p$ should be a formal expression composed of the addition, multiplication, scalar multiplication, and $*$-operations.  But as we work with tracial von Neumann algebras throughout, this distinction is not important here.}
			\item If $\varphi(x_1,\dots,x_n,y)$ is a formula, and $R > 0$, then
			\[
			\psi(x_1,\dots,x_n) = \sup_{y \in D_R} \varphi(x_1,\dots,x_n,y)
			\]
			is a formula in free variables $(x_1,\dots,x_n)$, and the same holds with $\inf$ rather than $\sup$.  The variable $y$ is \emph{bound} to the quantifier $\sup_{y \in D_R}$.  The $D_R$ is a \emph{domain of quantification} corresponding to the operator norm ball of radius $R$.
			\item If $\varphi_1$, \dots, $\varphi_k$ are formulas in free variables $x_1$, \dots, $x_n$, and $f: \bR^k \to \bR$ is a continuous function, then
			\[
			\varphi(x_1,\dots,x_n) = f(\varphi_1(x_1,\dots,x_n),\dots,\varphi_k(x_1,\dots,x_n)).
			\]
			is a formula.
		\end{itemize}
	\end{definition}
	
	\begin{definition}[Interpretation of formulas]
		Let $\cM = (M,\tau)$ be a tracial von Neumann algebra.  The \emph{interpretation} or \emph{evaluation} of a formula $\varphi$ in $\cM$ is defined recursively as well.  Each formula $\varphi$ in $n$-variables defines a function $\cM^n \to \bR$ as follows.
		\begin{itemize}
			\item If $\varphi(x_1,\dots,x_n)$ is the formula $\tr(p(x_1,\dots,x_n))$, then $\varphi^{\cM}(x_1,\dots,x_n) = \tau(p(x_1,\dots,x_n))$ for $\mathbf{x} = (x_1,\dots,x_n)$ in $\cM^n$.
			\item If $\varphi(x_1,\dots,x_n) = \sup_{y \in D_R} \psi(x_1,\dots,x_n,y)$ as a formula, then we have
			\[
			\varphi^{\cM}(x_1,\dots,x_n) = \sup_{y \in D_R^{\cM}} \psi(x_1,\dots,x_n,y),
			\]
			where $D_R^{\cM}$ is the operator-norm ball of radius $R$ in $\cM$.
			\item If $f: \bR^k \to \bR$ is a continuous function, then $f(\varphi_1,\dots,\varphi_k)^{\cM} = f(\varphi_1^{\cM},\dots,\varphi_k^{\cM})$.
		\end{itemize}
	\end{definition}
	
	These formulas are a continuous analog of first-order logical formulas.  While formulas in discrete logic take the values true and false, formulas in continuous model theory take real values.  The quantifiers $\sup$ and $\inf$ are used instead of $\forall$ and $\exists$.  Continuous functions $f: \bR^k \to \bR$ as logical connectives instead of the usual operations such as $\vee$, $\wedge$, $\neg$.  Besides the motivation from mathematical logic, formulas as defined above are natural objects to consider from the viewpoint of non-commutative analysis, since operations like the Legendre transform and Hopf-Lax semigroup (see \S \ref{subsec: classical convex}) are defined in terms of suprema and infima.
	
	\begin{definition} ~
		\begin{itemize}
			\item A \emph{sentence} is a formula with no free variables. 
			\item A \emph{theory} is a collection of sentences $\varphi$.
			\item $\cM$ \emph{models} a theory $\rT$ if $\varphi^{\cM} = 0$ for all $\varphi \in \rT$, and in this case we write $\cM \models \rT$.
			\item The theory $\Th(\cM)$ of a tracial von Neumann algebra $\cM$ is the set of sentences such that $\varphi^{\cM} = 0$.
			\item Two tracial von Neumann algebras are \emph{elementarily equivalent} if they have the same theory.
		\end{itemize} 
	\end{definition}
	
	Farah, Hart, and Sherman \cite[Proposition 3.3]{FHS2014} showed that tracial von Neumann algebras can be axiomatized by a theory $\rT_{\tr}$.  Moreover, tracial factors can be axiomatized by a theory $\rT_{\tr,\factor}$ by \cite[Proposition 3.4(1)]{FHS2014}.  Another approach to axiomatizating tracial von Neumann algebras is given in \cite[Proposition 29.4]{BYIT2024}.
	
	There is also a natural topology on the set of theories $\Th(\cM)$ for tracial von Neumann algebras $\cM$.  Note that for every sentence $\varphi$, there is a unique constant $c$ such that $(\varphi - c)^{\cM} = \varphi^{\cM} - c = 0$.  Hence, $\{ \varphi: \varphi^{\cM} = 0\}$ is equivalent information to the linear map $\varphi \mapsto \varphi^{\cM}$ on the vector space of sentences.  Thus, we can also view the set of theories in the vector space dual of the set of sentences.  Then we say that $\Th(\cM_i) \to \Th(\cM)$ if $\varphi^{\cM_i} \to \varphi^{\cM}$ for all sentences $\varphi$.
	
	In general, if $\cM$ is the ultraproduct $\prod_{n \to \cU} \cM_n$ of tracial von Neumann algebras, then $\Th(\cM) = \lim_{n \to \cU} \Th(\cM_n)$; this is a special case of {\L}o{\'s}'s theorem \cite[Theorem 5.4]{BYBHU2008}, which also implies that if $\mathbf{x} \in \cM^n$ is given by a sequence $\mathbf{x}_n \in \cM_n^m$, then $\tp^{\cM}(\mathbf{x}) = \lim_{n \to \cU} \tp^{\cM_n}(\mathbf{x}_n)$.  As mentioned in the introduction, it is unknown whether $\lim_{n \to \infty} \Th(\bM_n)$ exists for the matrix algebras $\bM_n$.  More concretely, this means that we do not know whether an expression obtained by iterated suprema and infima over the operator norm ball of $\bM_n$ will have a limit as $n \to \infty$.  Or in other words, we do not know whether $\prod_{n \to \cU} \bM_n$ and $\prod_{n \to \cV} \bM_n$ are elementarily equivalent for two different ultrafilters $\cU$ and $\cV$ on the natural numbers.  Therefore, we write
	\[
	\rT_{\cU} := \lim_{n \to \cU} \Th(\bM_n) = \Th\left( \prod_{n \to \cU} \bM_n \right).
	\]
	Now we come to the definition of \emph{types} which is most important for this work.  For further background on types, see also \cite[\S 8]{BYBHU2008}, \cite[\S 4.1]{AKE2021}, \cite[\S 7]{Hart2023}.
	
	\begin{definition}[Types]
		Let $\mathcal{F}_n$ denote the set of formulas for tracial von Neumann algebras in $n$ free variables.  Let $\cM = (M,\tau)$ be a tracial von Neumann algebra.  For $\mathbf{x} = (x_1,\dots,x_m) \in \cM^m$, the \emph{(full) type} $\tp^{\cM}(\mathbf{x})$ is the linear map $\cF_n \to \bR: \varphi \mapsto \varphi^{\cM}(\mathbf{x})$.
		
		For every formula $\varphi$ and type $\mu \in \mathbb{S}_{m,R}(\rT)$, we denote by $(\mu,\varphi)$ the dual pairing or evaluation of $\mu$ on $\varphi$ (since by definition $\mu$ is a linear functional on $\mathcal{F}_n$).
	\end{definition}
	
	\begin{definition}[Spaces of types]
		For a theory $\rT$ (especially for $\rT_{\tr}$ and $\rT_{\tr,\factor}$ and $\rT_{\cU}$ discussed above), we write
		\[
		\mathbb{S}_{m,R}(\rT) = \{\tp^{\cM}(\mathbf{x}): \mathbf{x} \in (D_R^{\cM})^m, \cM \models \rT \}.
		\]
		We equip $\mathbb{S}_{m,R}(\rT)$ with the weak-$*$ topology as linear functionals on $\mathcal{F}_n$.  We write $\mathbb{S}_{m}(\rT) = \bigcup_{R > 0} \mathbb{S}_{m,R}(\rT)$ equipped with the inductive limit topology.
	\end{definition}
	
	\begin{remark} \label{rem: realization as sets}
	Although the class of models of $\mathrm{T}$ is not a \emph{set}, $\mathbb{S}_{m,R}(\mathrm{T})$ is a set because it is defined as a subset of the set of linear functionals on $\mathcal{F}_n$; moreover, for a linear functional $\mu$ on $\mathcal{F}_n$, the condition that there exists a model $\cM$ of $\mathrm{T}$ and $\mathbf{x} \in (D_R^{\mathcal{M}})^m$ with $\tp^{\cM}(\mathbf{x}) = \mu$ can be expressed in terms of sets.  Indeed, if $\mathrm{T}$ does not have any model (i.e., it is not consistent), then there are no types relative to $\mathrm{T}$.  Otherwise, fix a model $\cM$ of $\mathrm{T}$ and a free ultrafilter on $\mathbb{N}$.  Then the ultraproduct $\cM^{\cU}$ will be countably saturated, and so all types of finite tuples will be realized in $\cM^{\cU}$ (see \cite[\S 4.4]{FHS2014}, \cite[p.~33ff.]{BYBHU2008}, \cite[\S 6.1]{JekelModelEntropy}), so that $\mathbb{S}_{m,R}(\rT) = \{\tp^{\cM}(\mathbf{x}): \mathbf{x} \in (D_R^{\cM^{\cU}})^m \}$.
	\end{remark}
	
	\begin{remark}[Laws and quantifier-free types] \label{rem: qf type}
		This definition resembles that of non-commutative laws (Notation \ref{not: space of laws}) except that now a larger class of test functions is used.  Note that a \emph{quantifier-free formula}, that is, a formula without any suprema or infima, can always be expressed as
		\[
		f(\re \tr(p_1(x_1,\dots,x_m)), \dots, \re \tr(p_k(x_1,\dots,x_m)))
		\]
		for some continuous function $f: \bC^k \to \bR$ and non-commutative $*$-polynomials $p_1$, \dots, $p_k$.  In particular, for tuples $\mathbf{x}_n \in \cM_n^m$ bounded by $R$ in operator norm, convergence of $\law(\mathbf{x}_n)$ is equivalent to convergence of $\varphi^{\cM_n}(\mathbf{x}_n)$ for every quantifier-free formula.  The \emph{quantifier-free type} of $\mathbf{x}$ is defined as the linear functional on quantifier-free formulas given by evaluation at $\mathbf{x}$, and the space $\mathbb{S}_{m,R,\qf}(\rT_{\tr})$ of quantifier-free types can also be equipped with a weak-$*$ topology.  One then sees that $\mathbb{S}_{m,R,\qf}(\rT_{\tr})$ is weak-$*$ homeomorphic to $\Sigma_{m,R}^*$ in the natural way.  At the same time, $\mathbb{S}_{m,R,\qf}(\rT_{\tr})$ is a quotient space of $\mathbb{S}_{m,R}(\rT_{\tr})$ via the natural restriction map.  This connection is explained in more detail in \cite[\S 3.4-3.5]{JekelCoveringEntropy}.
	\end{remark}
	
	Although every formula defines a weak-$*$ continuous function on $\mathbb{S}_{m,R}(\rT)$ for each theory $\rT$, the converse is not true.  The objects that correspond to continuous functions on $\mathbb{S}_{m,R}(\rT)$ are a certain completion of the set of formulas, called \emph{definable predicates}; see e.g.\ \cite[\S 5.2]{Hart2023}.
	
	\begin{definition}[Definable predicates]
		Let $\rT$ be a theory in the language of tracial von Neumann algebras.  A \emph{$m$-variable definable predicate relative to $\rT$} is a collection $\varphi = (\varphi^{\cM})_{\cM \models \rT}$ of functions $\varphi^{\cM}: \cM^m \to \bR$ for $\cM \models \rT$, such that for every $\varepsilon > 0$ and $R > 0$, there is a formula $\psi$ with
		\[
		\sup_{\cM \models \rT} \sup_{\mathbf{x} \in (D_r^{\cM})^m} |\varphi^{\cM}(\mathbf{x}) - \psi^{\cM}(\mathbf{x})| < \varepsilon.
		\]
		(This supremum can be expressed in terms of sets similarly as in Remark \ref{rem: realization as sets}.)
	\end{definition}
	
	\begin{remark} \label{rem: formula completion}
		The set of definable predicates is easily seen to be the completion of $\mathcal{F}_n$ with respect to the family of seminorms
		\[
		\norm{\varphi}_R = \sup_{\cM \models \rT} \sup_{\mathbf{x} \in (D_r^{\cM})^m} |\varphi^{\cM}(\mathbf{x})|,
		\]
		and as such it is a Fr{\'e}chet topological vector space.  For a type $\mu \in \mathbb{S}_m$ and a definable predicate $\varphi$, we denote by $(\mu,\varphi)$ the dual pairing satisfying $(\mu,\varphi) = \varphi^{\cM}(\mathbf{x})$ when $\tp^{\cM}(\mathbf{x}) = \mu$.
	\end{remark}
	
	\begin{remark} \label{rem: definable predicates and continuous functions}
		$\mathbb{S}_{m,R}(\rT)$ is a compact Hausdorff space (and in fact metrizable in the setting of tracial von Neumann algebras).  Moreover, every formula and every definable predicate yields a continuous function on $\mathbb{S}_{m,R}(\rT)$.  Conversely, for every continuous function $f: \mathbb{S}_{m,R} \to \bR$, there exists a definable predicate $\varphi$ such that $f(\mu) = (\mu,\varphi)$ for $\mu \in \mathbb{S}_{m,R}(\rT)$ (see \cite[Lemma 2.16]{JekelTypeCoupling}).  In particular, for every $\mu \in \mathbb{S}_{m,R}(\rT)$, there exists a definable predicate $\varphi$ such that $(\nu,\varphi) \geq 0$ for $\nu \in \mathbb{S}_{m,R}(\rT)$ with equality if and only if $\nu = \mu$.
	\end{remark}
	
	\begin{fact} \label{fact: definable predicate operations}
		If $\varphi(x_1,\dots,x_n,y)$ is a definable predicate for some theory $\rT$ $R > 0$, then so are $\sup_{y \in D_R} \varphi(x_1,\dots,x_n,y)$ and $\inf_{y \in D_R} \varphi(x_1,\dots,x_n,y)$.  Similarly, definable predicates are closed under application of continuous connectives.  See \cite[Proposition 9.3]{BYBHU2008}, \cite[Lemma 3.12]{JekelCoveringEntropy}.
	\end{fact}
	
	While definable predicates provide the analog of scalar-valued continuous functions in the model-theoretic setting, we will also be concerned with \emph{definable functions} from $\cM^m \to \cM$.
	
	\begin{definition}[Definable function]
		Let $\rT$ be some theory of tracial von Neumann algebras.  A \emph{definable function with respect to $\rT$} is a collection of functions $(f^{\cM})_{\cM \models \rT}$ where $f^{\cM}: \cM^m \to \cM$ such that
		\begin{itemize}
			\item For each $R > 0$, there exists $R' > 0$ such that $f^{\cM}$ maps $(D_R^{\cM})^m$ into $D_{R'}^{\cM}$.
			\item There exists a definable predicate $\varphi$ in $m+1$ variables such that $\varphi^{\cM}(\mathbf{x},y) = \norm{f(\mathbf{x}) - y}_{L^2(\cM)}$ for each $\cM \models \rT$ and $\mathbf{x} \in \cM^m$ and $y \in \cM$.
		\end{itemize}
		Moreover, we also use ``definable function'' to refer to a tuple $\mathbf{f} = (f_1,\dots,f_{m'})$ of definable functions in the above sense.
	\end{definition}
	
	We also use the fact that definable predicates and definable functions behave well under composition.  For proof, see \cite[Proposition 3.17]{JekelCoveringEntropy}.
	
	\begin{lemma} \label{lem: definable composition}
		Let $\rT$ be some theory of tracial von Neumann algebras, $\mathbf{f} = (f_1,\dots,f_{m'})$ in $m$ variables, and $\varphi$ a definable predicate in $m'$ variables.  Then $\varphi \circ \mathbf{f}$ is a definable predicate.
	\end{lemma}
	
	Pushforwards of types by definable functions are defined analogously to pushforwards of measures by continuous functions.  We recall the following from \cite[Lemma 3.20]{JekelCoveringEntropy}:  Suppose $\rT$ is some theory and $\mathbf{f} = (f_1,\dots,f_{m'})$ is an $m'$-tuple of $m$-variable definable functions with respect to $\rT$.  Let $\cM \models \rT$ and $\mathbf{x} \in \cM^m$.  Then $\tp^{\cM}(\mathbf{f}(\mathbf{x}))$ is uniquely determined by $\tp^{\cM}(\mathbf{x})$, and depends weak-$*$ continuously on $\tp^{\cM}(\mathbf{x})$.  If $\tp^{\cM}(\mathbf{x}) = \mu$, then we denote $\tp^{\cM}(\mathbf{f}(\mathbf{x}))$ by $\mathbf{f}_* \mu$.
	
	We finally recall the notion of \emph{definable closure} which turns out to be closely related to definable functions.  In the model-theoretic setting, the definable closure of some tuple $\mathbf{x}$ (or more generally some subset) is an appropriate analog of the von Neumann subalgebra generated by $\mathbf{x}$.
	
	\begin{definition}[Definable closure] \label{def: definable closure}
		Let $\cM$ be a tracial von Neumann algebra and let $\mathbf{x} \in \cM^m$.  We say that $z$ is in the \emph{definable closure} of $\mathbf{x}$, or $z \in \dcl^{\cM}(\mathbf{x})$, if there exists a definable predicate $\varphi$ with respect to $\rT_{\tr}$ in $m+1$ variables, such that $\norm{z - y}_{L^2(\cM)} = \varphi(\mathbf{x},y)$ for all $y \in \cM$.
	\end{definition}
	
	For tracial von Neumann algebras, any element in the definable closure of $\mathbf{x}$ can be expressed as a definable function of $\mathbf{x}$.
	
	\begin{theorem}[{\cite[Theorem 1.4]{JekelTypeCoupling}}] \label{thm: definable closure}
		Let $\cM$ be a tracial von Neumann algebra and let $\mathbf{x} \in \cM^m$.  Then $z \in \dcl^{\cM}(\mathbf{x})$ if and only if there exists a definable function $f$ with respect to $\rT_{\tr}$ such that $z = f(\mathbf{x})$.
	\end{theorem}
	
	\subsection{Wasserstein distance and optimal couplings for types} \label{subsec: optimal couplings}
	
	The Wasserstein distance on the type space $\mathbb{S}_m(\rT)$ is defined in a similar way to Biane and Voiculescu's Wasserstein distance for non-commutative laws \cite{BV2001}.  Its definition is in fact a special case of the $d$-metric in the model theory of metric structures \cite[\S 8, p.\ 44]{BYBHU2008}, and the $L^1$ Wasserstein distance on classical atomless probability spaces was studied from a model theory viewpoint in the thesis of Song \cite{SongThesis}.  Further discussion of the Wasserstein distance for types in tracial von Neumann algebras can be found in \cite[\S 6.1]{JekelModelEntropy}.  The definition is as follows.  In the notation, we include the subscript ``$\full$'' in order to make clear that is the Wasserstein distance for full or complete types in distinction to the classical Wasserstein distance or the Biane--Voiculescu distance.
	
	\begin{definition}[Wasserstein distance for types]
		Let $\rT$ be the theory of some tracial von Neumann algebra.  For $\mu, \nu \in \mathbb{S}_m(\rT)$, let
		\[
		d_{W,\full}(\mu,\nu) = \inf \{ \norm{\mathbf{x} - \mathbf{y}}_{L^2(\cM)^m}: \cM \models \rT, \mathbf{x}, \mathbf{y} \in \cM^m, \tp^{\cM}(\mathbf{x}) = \mu, \tp^{\cM}(\mathbf{y}) = \nu \}.
		\]
		Similarly, define
		\[
		C_{\full}(\mu,\nu) = \sup \{ \re \ip{\mathbf{x},\mathbf{y}}_{L^2(\cM)^m}: \cM \models \rT, \mathbf{x}, \mathbf{y} \in \cM^m, \tp^{\cM}(\mathbf{x}) = \mu, \tp^{\cM}(\mathbf{y}) = \nu \}.
		\]
		A pair $(\mathbf{x},\mathbf{y})$ that achieves the optimum in either of these equations is called an \emph{optimal coupling} of $(\mu,\nu)$ in $\cM$.
	\end{definition}
	
	\begin{remark} \label{rem: optimal coupling exists}
		Note that we do not fix $\cM$ from the beginning, but we allow $\cM$ to vary in order to witness the optimum.  However, as in Remark \ref{rem: realization as sets}, the supremum can be expressed as the supremum over a set because any $\cM$ which is countably saturated will realize all the possible types of $2m$-tuples $(\mathbf{x},\mathbf{y})$.  In particular, if $\cM$ is countably saturated, then $\cM$ will contain an optimal coupling of any two given types (see \cite[\S 6.1]{JekelModelEntropy}). Moreover, since any ultraproduct with respect to a free ultrafilter on $\bN$ is countably saturated (see \cite[\S 4.4]{FHS2014}, \cite[p.~33ff.]{BYBHU2008}), we conclude that for types with respect to the theory $\rT_{\cU}$ of $\cQ = \prod_{n \to \cU} \bM_n$, there always exists an optimal coupling in $\cQ$.
	\end{remark}
	
	The Wasserstein distance is weak-$*$ lower semi-continuous on $\mathbb{S}_{m,R}(\rT) \times \mathbb{S}_{m,R}(\rT)$.  This is a special case of the lower semi-continuity of the $d$-metric for types in metric structures \cite{BY2008topometric}.  This is also analogous to Biane and Voiculescu's observation for the Wasserstein distance for non-commutative laws \cite[Proposition 1.4(b)]{BV2001}.  As a consequence, we have in general that the Wasserstein distance of types is at most the limit of the Wasserstein distance of random matrix models for that type.
	
	\begin{lemma} \label{lem: Wasserstein distance of matrix models}
		Let $\mu, \nu \in \mathbb{S}_{m,R}(\rT_{\tr,\factor})$ be types.  Let $\cU$ be a free ultrafilter on $\bN$.  Let $\mathbf{X}^{(n)}$ and $\mathbf{Y}^{(n)}$ be tuples of random matrix $m$-tuples with distributions $\mu^{(n)}$ and $\nu^{(n)}$ respectively.  Suppose that for $R' > R$,
		\[
		\lim_{n \to \cU} P(\norm{\mathbf{X}^{(n)}} \geq R') = 0, \qquad \lim_{n \to \cU} P(\norm{\mathbf{Y}^{(n)}} \geq R') = 0,
		\]
		and assume that $\lim_{n \to \cU} \tp^{\bM_n}(\mathbf{X}^{(n)}) = \mu$ in probability in $\mathbb{S}_{m,R'}(\rT_{\tr,\factor})$, meaning that for every weak-$*$ neighborhood $\cO$ of $\mu$ and $R' > 0$, we have
		\[
		\lim_{n \to \cU} P(\tp^{\bM_n}(\mathbf{X}^{(n)}) \in \cO) = 1,
		\]
		and similarly assume that $\lim_{n \to \cU} \tp^{\bM_n}(\mathbf{Y}^{(n)}) = \nu$ in probability.  Let $d_{W,\operatorname{class}}(\mu^{(n)},\nu^{(n)})$ be the classical Wasserstein distance of $\mu^{(n)}$ and $\nu^{(n)}$ as probability distributions on $\bM_n^m$ with the inner product associated to $\tr_n$.  Then
		\[
		d_{W,\full}(\mu,\nu) \leq \lim_{n \to \cU} d_{W,\operatorname{class}}(\mu^{(n)},\nu^{(n)}).
		\]
	\end{lemma}
	
	\begin{proof}
		Let
		\[
		c = \lim_{n \to \cU} d_{W,\operatorname{class}}(\mu^{(n)},\nu^{(n)}).
		\]
		Fix $R' > R$.  Since $\mathbb{S}_{m,R'}(\rT_{\tr,\factor})$ is metrizable, there is a sequence of neighborhoods $\cO_k$ such that $\overline{\cO_{k+1}} \subseteq \cO_k$ and $\bigcap_{k \in \bN} \cO_k = \{\mu\}$.  Similarly, fix such a sequence of neighborhoods $\cO_k'$ for $\nu$.  
		
		Assume without loss of generality that $\mathbf{X}^{(n)}$ and $\mathbf{Y}^{(n)}$ are random variables on the same probability space which provide a classical optimal coupling of $\mu^{(n)}$ and $\nu^{(n)}$.  By applying Markov's inequality to the nonnegative random variable $\norm{\mathbf{X}^{(n)} - \mathbf{Y}^{(n)}}_{\tr_n}^2$,
		\[
		P(\norm{\mathbf{X}^{(n)} - \mathbf{Y}^{(n)}}_{\tr_n} \geq c + 1/k) \leq \frac{\mathbb{E} \norm{\mathbf{X}^{(n)} - \mathbf{Y}^{(n)}}_{\tr_n}^2}{(c + 1/k)^2} = \frac{d_{W,\operatorname{class}}(\mu^{(n)},\nu^{(n)})^2}{(c + 1/k)^2}.
		\]
		Then observe that
		\begin{multline*}
			P\bigl(\tp^{\bM_n}(\mathbf{X}^{(n)}) \in \cO_k, \tp^{\bM_n}(\mathbf{Y}^{(n)}) \in \cO_k', \norm{\mathbf{X}^{(n)} - \mathbf{Y}^{(n)}}_{\tr_n} \leq c + 1/k \bigr) \\ \geq 1 - P(\tp^{\bM_n}(\mathbf{X}^{(n)}) \not \in \cO_k) - P(\tp^{\bM_n}(\mathbf{Y}^{(n)}) \not \in \cO_k') - P(\norm{\mathbf{X}^{(n)} - \mathbf{Y}^{(n)}}_{\tr_n} \geq c + 1/k).
		\end{multline*}
		Hence,
		\[
		\lim_{n \to \cU} P\bigl(\tp^{\bM_n}(\mathbf{X}^{(n)}) \in \cO_k, \tp^{\bM_n}(\mathbf{Y}^{(n)}) \in \cO_k', \norm{\mathbf{X}^{(n)} - \mathbf{Y}^{(n)}}_{\tr_n} \leq c + 1/k \bigr) \geq 1 - \frac{c^2}{(c+ 1/k)^2} > 0.
		\]
		In particular, for some $n$ and some outcomes in the probability space
		\[
		\tp^{\bM_n}(\mathbf{X}^{(n)}) \in \cO_k \text{ and } \tp^{\bM_n}(\mathbf{Y}^{(n)}) \in \cO_k' \text{ and } \norm{\mathbf{X}^{(n)} - \mathbf{Y}^{(n)}}_{\tr_n} \leq c + 1/k.
		\]
		
		We need to conclude using compactness of $\mathbb{S}_{2m,R'}(\rT_{\tr,\factor})$.  First, let $\pi_1, \pi_2: \mathbb{S}_{2m,R'}(\rT_{\tr,\factor})$ be the restriction maps to formulas in the first $m$ variables and the last $m$ variables respectively, or equivalently $\pi_1(\tp^{\cM}(\mathbf{x},\mathbf{y})) = \tp^{\cM}(\mathbf{x})$ and $\pi_2(\tp^{\cM}(\mathbf{x},\mathbf{y})) = \tp^{\cM}(\mathbf{y})$ when $(\mathbf{x},\mathbf{y}) \in \cM^{2m}$ for a tracial factor $\cM$.  Let
		\[
		\mathcal{S}_k = \left\{ \sigma \in \mathcal{S}_{2m,R'}(\rT_{\tr,\factor}): \pi_1(\sigma) \in \overline{\cO_k}, \pi_2(\sigma) \in \overline{\cO_k'}, \sigma\left[\sum_{j=1}^n \tr(|x_j - x_{m+j}|^2)\right] \leq (c + 1/k)^2 \right\},
		\]
		or equivalently
		\[
		\mathbb{S}_k = \{ \tp^{\cM}(\mathbf{x},\mathbf{y}): \cM \text{ tracial factor}, \tp^{\cM}(\mathbf{x}) \in \overline{\cO_k}, \tp^{\cM}(\mathbf{y}) \in \overline{\cO_k'}, \norm{\mathbf{x} - \mathbf{y}}_{L^2(\cM)} \leq c + 1/k\}.
		\]
		The foregoing argument shows that $\mathcal{S}_k$ is nonempty for each $k$, since it contains $\tp^{\bM_n}(\mathbf{X}^{(n)},\mathbf{Y}^{(n)})$ for some $n$ and some outcome in the probability space.  Also, $\mathcal{S}_k \supseteq \mathcal{S}_{k+1}$ and $\mathcal{S}_k$ is closed.  Since $\mathbb{S}_{2m,R'}(\rT_{\tr,\factor})$ is compact, $\bigcap_{k \in \bN} \mathcal{S}_k$ is nonempty.  Therefore, there exists some $2m$-tuple $(\mathbf{x},\mathbf{y})$ in a tracial factor $\cM$ such that for all $k$,
		\[
		\tp^{\cM}(\mathbf{x}) \in \overline{\cO_k}, \quad \tp^{\cM}(\mathbf{y}) \in \overline{\cO_k'}, \quad \norm{\mathbf{x} - \mathbf{y}}_{L^2(\cM)^m} \leq c + 1/k,
		\]
		hence
		\[
		\tp^{\cM}(\mathbf{x}) = \mu, \quad \tp^{\cM}(\mathbf{y}) = \nu, \quad \norm{\mathbf{x} - \mathbf{y}}_{L^2(\cM)^m} \leq c.
		\]
		Therefore, $d_{W,\full}(\mu,\nu) \leq c$ as desired.
	\end{proof}
	
	In \cite{JekelTypeCoupling}, the present author gave an analog of Monge-Kantorovich duality (Theorem \ref{thm: classical MK duality}) for types in tracial von Neumann algebras.
	
	\begin{theorem}[{\cite[Theorem 1.1]{JekelTypeCoupling}}] \label{thm: MK duality}
		Fix a complete theory $\rT$ of a tracial von Neumann algebra.  Let $\mu$ and $\nu \in \mathbb{S}_m(\rT)$ be types.  Then there exist convex $\mathrm{T}_{\tr}$-definable predicates $\varphi$ and $\psi$ such that
		\begin{equation} \label{eq: admissibility}
			\varphi^{\cM}(\mathbf{x}) + \psi^{\cM}(\mathbf{y}) \geq \re \ip{\mathbf{x},\mathbf{y}}_{L^2(\cM)} \text{ for all } \mathbf{x}, \mathbf{y} \in \cM^m \text{ for all } \cM \models \mathrm{T}_{\tr},
		\end{equation}
		and such that equality is achieved when $(\mathbf{x},\mathbf{y})$ is an optimal coupling of $(\mu,\nu)$.  Hence, $C_{\full}(\mu,\nu)$ is the infimum of $(\mu,\varphi) + (\nu,\psi)$ over all pairs $(\varphi,\psi)$ of convex definable predicates satisfying \eqref{eq: admissibility}.
	\end{theorem}
	
	Through studying the convex definable predicates $\varphi$ and $\psi$ more closely, one can also show the following.
	
	\begin{theorem}[{\cite[Theorem 1.3]{JekelTypeCoupling}}]
		Let $\rT$ be the theory of some tracial von Neumann algebra.  Let $(\mathbf{x},\mathbf{y})$ be an optimal coupling of $\mu, \nu \in \mathbb{S}_m(\rT)$ in some $\cM$.  Let $\mathbf{x}_t = (1-t) \mathbf{x} + t \mathbf{y}$.  Then $\dcl^{\cM}(\mathbf{x}_t) = \dcl^{\cM}(\mathbf{x},\mathbf{y})$.
	\end{theorem}
	
	This means that $\mathbf{x}$ and $\mathbf{y}$ can be expressed as definable functions of $\mathbf{x}_t$ (see Theorem \ref{thm: definable closure}).  In this paper, in order to obtain estimates on the free entropy and free entropy dimension, we will show that they can be expressed as Lipschitz definable predicates of $\mathbf{x}_t$, following a similar method as holds in the classical case (see \S \ref{subsec: classical convex}) and also the non-commutative setting in \cite[\S 4]{GJNS2021}.
	
	\subsection{The displacement interpolation} \label{subsec: displacement interpolation}
	
	Our goal in this section is, for a given optimal coupling $(\mathbf{x},\mathbf{y})$ and to study the interpolation $\mathbf{x}_t = (1 - t) \mathbf{x} + t \mathbf{y}$.  It was shown in \cite{JekelTypeCoupling} that $\dcl^{\cM}(\mathbf{x}_t) = \dcl^{\cM}(\mathbf{x},\mathbf{y})$ for $t \in (0,1)$.  We will now go further and show that $\mathbf{x}_s$ is a Lipschitz definable function applied to $\mathbf{x}_t$ when $t \in (0,1)$, which will be essential for our applications to entropy.  The proof proceeds similarly to \cite[\S 4]{GJNS2021} by studying pairs of convex functions $\varphi_{s,t}$ and $\psi_{s,t}$ that witness Monge-Kantorovich duality for $(\mathbf{x}_s,\mathbf{x}_t)$.  Using strong convexity and semiconcavity, we will show that $\nabla \varphi_{s,t}$ is Lipschitz with $\mathbf{x}_s = \nabla \varphi_{s,t}(\mathbf{x}_t)$; this builds on the results of \cite[\S 5]{JekelTypeCoupling}.
	
	Motivated by the classical results sketched in \S \ref{subsec: classical convex}, our goal is obtain a similar result in the setting of types and definable predicates for tracial von Neumann algebras.  The existence of Lipschitz transport functions will enable estimates of free entropy and free entropy dimension of $\mathbf{x}_t$ in terms of that of $\mathbf{x}_s$ for Theorem \ref{thm: entropy along geodesics}.  There are several technical points we must pay careful attention to.  First, for our applications to entropy, it is crucial that $\nabla \varphi_{s,t}$ should be a definable function, since this means that it plays well with matrix approximations for types in the weak-$*$ topology (notably this is \emph{not} the case for the functions studied in \cite{GJNS2021}).  In \cite[\S 5.1]{JekelTypeCoupling}, it was shown that $\nabla \varphi$ is a definable function if $\varphi$ is semiconvex and semiconcave and $\nabla \varphi$ satisfies certain operator-norm bounds.  Moreover, if we restrict our attention to \emph{factors} $\cM$, the Lipschitzness of the gradient automatically implies the needed operator-norm bounds \cite[Corollary 5.6]{JekelTypeCoupling}.
	
	Second, in general, if $\varphi$ is a convex definable predicate, it is not clear whether the Legendre transform, given by
	\[
	(\mathcal{L} \varphi)^{\cM}(\mathbf{y}) = \sup_{\mathbf{x}\in \cM^m} \left[ \ip{\mathbf{x},\mathbf{y}}_{L^2(\cM)} - \varphi^{\cM}(\mathbf{x}) \right],
	\]
	is actually a definable predicate (even if we assume it is finite everywhere).   This is because the model-theoretic setup only allows taking suprema over operator norm balls.  Thus, in \cite[\S 5.2]{JekelTypeCoupling}, suprema and infima over operator norm balls were used for various operations on convex definable predicates.  In general,
	\[
	(\mathcal{L} \varphi)^{\cM}(\mathbf{y}) = \sup_{R > 0} \sup_{\mathbf{x} \in (D_R^{\cM})^m} \left[ \ip{\mathbf{x},\mathbf{y}}_{L^2(\cM)} - \varphi^{\cM}(\mathbf{x}) \right],
	\]
	will be the supremum of a countable family of definable predicates, and hence will define a weak-$*$ lower semi-continuous function on the type space (while definable predicates would define weak-$*$ continuous functions).  We will show that if $\varphi$ is strongly convex, then $\mathcal{L} \varphi$ will be a definable predicate.  We accomplish this by showing that the supremum is actually achieved in an operator norm ball with radius depending on the operator norm of the input, which in turn follows because the maximizer is described in terms of the gradient of $\mathcal{L} \varphi$, and its operator norm can be estimated using \cite[Corollary 5.6]{JekelTypeCoupling}.
	
	Now we begin the technical proofs for the results on convex definable predicates, inf-convolutions, and Legendre transforms that we need for our applications.  The first is a basic estimate for $\nabla \varphi^{\cM}(\mathbf{0})$ which is needed in order to estimate the $\varphi$ on various operator norm balls (and will also be used in our study of Gibbs types in \S \ref{sec: Gibbs types}).
	
	\begin{lemma} \label{lem: gradient at zero}
		Let $\varphi$ be a convex definable predicate with respect to $\rT_{\tr,\factor}$, and fix $\cM \models \rT_{\tr,\factor}$.  Then there exists some $\mathbf{y} \in (y_1,\dots,y_m) \in \underline{\nabla} \varphi^{\cM}(\mathbf{0}) \cap \bC^m$.  Moreover, for every such $\mathbf{y}$ and for every $R > 0$, we have
		\[
		|\mathbf{y}| \leq \frac{1}{R} \sup_{\mathbf{x} \in [-R,R]^m} \left[ \varphi^{\cM}(\mathbf{x}) - \varphi(0) \right]
		\]
	\end{lemma}
	
	\begin{proof}
		By \cite[Proposition 4.5]{JekelTypeCoupling}, there exists some $\mathbf{y} \in \underline{\nabla} \varphi^{\cM}(\mathbf{0})$ which is also in $L^2(\dcl^{\cM}(\bC))$.  By \cite[Observation 3.8]{JekelTypeCoupling}, we have
		\[
		\dcl^{\cM}(\bC) \subseteq (\bC' \cap \cM)' \cap \cM = \cM' \cap \cM = \bC.
		\]
		Thus, $\mathbf{y} \in \bC^m$.  Since $\mathbf{y} \in \underline{\nabla} \varphi^{\cM}(\mathbf{0})$, we have
		\[
		\varphi^{\cM}\left(\frac{R}{|\mathbf{y}|} \mathbf{y} \right) - \varphi^{\cM}(0) \geq \re \ip*{\frac{R}{|\mathbf{y}|} \mathbf{y}, \mathbf{y}}_{L^2(\cM)^m} = R |\mathbf{y}|.
		\]
		Hence,
		\[
		|\mathbf{y}| \leq \frac{1}{R} \left[ \varphi^{\cM}\left(\frac{R}{|\mathbf{y}|} \mathbf{y} \right) - \varphi^{\cM}(0) \right] \leq \frac{1}{R} \sup_{\mathbf{x} \in [-R,R]^m} \left[ \varphi^{\cM}(\mathbf{x}) - \varphi(0) \right].  \qedhere
		\]
	\end{proof}
	
	Next, we recall the result from \cite{JekelTypeCoupling} which will be used to control the operator norms of $\nabla \varphi(\mathbf{x})$ for certain convex definable predicates.
	
	\begin{lemma}[{\cite[Corollary 5.6]{JekelTypeCoupling}}] \label{lem: Lipschitz boundedness}
		Let $\cM$ be a tracial factor and let $F: \cM^n \to L^2(\cM)^m$ be $L$-Lipschitz with respect to $\norm{\cdot}_{L^2(\cM)}$ and equivariant under unitary conjugation.  Let $\mathbf{r} = (r_1,\dots,r_n) \in (0,\infty)^n$.  Let $C = \max_i (|\tr^{\cM}(F_i(0))|)$.  Then $F$ maps $D_{\mathbf{r}}^{\cM}$ into $(D_{C+9L|\mathbf{r}|}^{\cM})^m$.
	\end{lemma}
	
	Although this result was stated in \cite{JekelTypeCoupling} for $\mathrm{II}_1$ factors, the proof works equally well for finite-dimensional factors, i.e.\ matrix algebras, since it only requires that all projections of the same trace are Murray-von Neumann equivalent.  Similarly, the following result applies for tracial factors in general.
	
	\begin{corollary}[{\cite[Corollary 5.7]{JekelTypeCoupling}}] \label{cor: definable gradient}
		Let $\varphi$ be a definable predicate with respect to $\rT_{\tr,\factor}$ that is $c$-semiconvex and $c$-semiconcave for some $c > 0$.  Then $\varphi$ is differentiable, $\nabla \varphi$ is a definable function, and $\nabla \varphi$ is $c$-Lipschitz with respect to the $L^2$-norm.
	\end{corollary}
	
	Now we are ready to prove that the semi-concave regularization of a convex definable predicate is itself a convex definable predicate.  This will then be used in Lemma \ref{lem: Legendre transform definable predicate} to obtain the analogous result for Legendre transforms of strongly convex definable predicates.
	
	\begin{proposition} \label{prop: definable inf-convolution}
		Let $\varphi$ be a convex definable predicate with respect to $\rT_{\tr,\factor}$.  For $\cM \models \rT_{\tr,\factor}$, let
		\[
		\varphi_t^{\cM}(\mathbf{x}) = \inf_{\mathbf{y} \in \cM^m} \left[ \varphi^{\cM}(\mathbf{y}) + \frac{1}{2t} \norm{\mathbf{x} - \mathbf{y}}_{L^2(\cM)^m}^2 \right].
		\]
		\begin{enumerate}[(1)]
			\item Let $C = \sup_{\cM \models \rT_{\tr,\factor}} \sup_{\mathbf{x} \in (D_1^{\cM})^m} [\varphi^{\cM}(\mathbf{x}) - \varphi^{\cM}(0)]$.  Then for $R > 0$, we have
			\[
			\mathbf{x} \in (D_R^{\cM})^m \implies \varphi_t^{\cM}(\mathbf{x}) = \inf_{\mathbf{y} \in (D_{2Ct+9\sqrt{m}R}^{\cM})^m} \left[ \varphi^{\cM}(\mathbf{y}) + \frac{1}{2t} \norm{\mathbf{x} - \mathbf{y}}_{L^2(\cM)^m}^2 \right].
			\]
			\item $\varphi_t$ is a definable predicate with respect to $\rT_{\tr,\factor}$ which is convex and $1/t$-semiconcave.
			\item $\nabla \varphi_t$ is a definable function with respect to $\rT_{\tr,\factor}$, it is $1/t$-Lipschitz with respect to $L^2(\cM)^m$-norm, and it satisfies $\nabla \varphi_t^{\cM}(0) \in \bC^m$ with $|\nabla \varphi_t(0)| \leq 2C$.
			\item The minimizer $\mathbf{y}$ in the definition of $\varphi_t^{\cM}(\mathbf{x})$ is given by $\mathbf{y} = \mathbf{x} - t \nabla \varphi_t(\mathbf{x})$.
		\end{enumerate}
		
	\end{proposition}
	
	\begin{proof}
		(1) Write
		\[
		\psi_t^{\cM}(\mathbf{x},\mathbf{y}) = \varphi^{\cM}(\mathbf{y}) + \frac{1}{2t} \norm{\mathbf{x} - \mathbf{y}}_{L^2(\cM)^m}^2.
		\]
		Given $R' > 0$, define
		\[
		\varphi_{t,R'}^{\cM}(\mathbf{x}) = \inf_{\mathbf{y} \in (D_{R'}^{\cM})^m} \psi_t^{\cM}(\mathbf{x},\mathbf{y}).
		\]
		which is a definable predicate by Fact \ref{fact: definable predicate operations}.  The existence and uniqueness of the minimizer in the definition of $\varphi_{t,R'}^{\cM}$ follows from standard arguments about convex functions on a closed convex subset of a Hilbert space as follows.  First, we claim that for $\mathbf{y}$, $\mathbf{y}' \in (D_{R'}^{\cM})^m$, we have
		\begin{equation} \label{eq: almost minimizers}
			\psi_t^{\cM}(\mathbf{x},\mathbf{y}) - \varphi_{t,R'}^{\cM}(\mathbf{x}) + \psi_t^{\cM}(\mathbf{x},\mathbf{y}') - \varphi_{t,R'}^{\cM}(\mathbf{x}) \geq \frac{1}{4t} \norm{\mathbf{y} - \mathbf{y}'}_{L^2(\cM)^m}^2.
		\end{equation}
		To see this, note that by \cite[Proposition 4.5]{JekelTypeCoupling}, there exists some $\mathbf{z} \in L^2(\cM)$ which is in $\underline{\nabla}_{\mathbf{y}} \psi_t^{\cM}(\mathbf{x}, (\mathbf{y} + \mathbf{y}')/2)$.  Since $\psi_t^{\cM}(\mathbf{x},\cdot)$ is $1/t$-strongly convex in $\mathbf{y}$, we have
		\[
		\psi_t^{\cM}(\mathbf{x}, \mathbf{y}') - \psi_t^{\cM}(\mathbf{x}, \tfrac{1}{2}(\mathbf{y} + \mathbf{y}')) \geq \re \ip{\mathbf{x}, \tfrac{1}{2}(\mathbf{y}' - \mathbf{y}) }_{L^2(\cM)^m} + \frac{1}{2t} \norm{\tfrac{1}{2}(\mathbf{y}' - \mathbf{y})}_{L^2(\cM)^m}^2
		\]
		and symmetrically
		\[
		\psi_t^{\cM}(\mathbf{x}, \mathbf{y}) - \psi_t^{\cM}(\mathbf{x}, \tfrac{1}{2}(\mathbf{y} + \mathbf{y}')) \geq \re \ip{\mathbf{x}, \tfrac{1}{2}(\mathbf{y} - \mathbf{y}') }_{L^2(\cM)^m} + \frac{1}{2t} \norm{\tfrac{1}{2}(\mathbf{y} - \mathbf{y}')}_{L^2(\cM)^m}^2.
		\]
		Adding together these inequalities,
		\[
		\psi_t^{\cM}(\mathbf{x}, \mathbf{y}') - \psi_t^{\cM}(\mathbf{x}, \tfrac{1}{2}(\mathbf{y} + \mathbf{y}')) + \psi_t^{\cM}(\mathbf{x}, \mathbf{y}) - \psi_t^{\cM}(\mathbf{x}, \tfrac{1}{2}(\mathbf{y} + \mathbf{y}')) \geq \frac{1}{4t} \norm{\mathbf{y}' - \mathbf{y}}_{L^2(\cM)^m}^2.
		\]
		This implies \eqref{eq: almost minimizers} since  $\psi_t^{\cM}(\mathbf{x},\frac{1}{2}(\mathbf{y} + \mathbf{y}')) \geq \varphi_{t,R'}^{\cM}(\mathbf{x})$ by definition of the latter.  By \eqref{eq: almost minimizers}, any sequence $\mathbf{y}_n \in (D_{R'}^{\cM})^m$ such that $\psi_t^{\cM}(\mathbf{x},\mathbf{y}_n) \to \varphi_t^{\cM}(\mathbf{x})$ will be Cauchy in $L^2(\cM)^m$.  Since $(D_{R'}^{\cM})^m$ is a closed subset, it converges to a minimizer $\mathbf{y}$ over $(D_{R'}^{\cM})^m$.  The inequality \eqref{eq: almost minimizers} also shows uniqueness of the minimizer.  Thus, we denote the minimizer by $\mathbf{f}_{t,R'}^{\cM}(\mathbf{x})$.
		
		We want to show that $\mathbf{f}_{t,R'}^{\cM}(\mathbf{x})$ is $1$-Lipschitz in $\mathbf{x}$ in order to apply Lemma \ref{lem: Lipschitz boundedness}.  First, we note the following inequality.  Letting $\mathbf{y} = \mathbf{f}_{t,R'}(\mathbf{x})$ and $\mathbf{y}' \in (D_{R'}^{\cM})^m$, the $1/t$-strong convexity of $\psi_t$ in $\mathbf{y}$ shows that for $\alpha \in [0,1]$,
		\begin{align*}
			\psi_t^{\cM}(\mathbf{x},\mathbf{y}) &\leq \psi_t^{\cM}(\mathbf{x},(1-\alpha)\mathbf{y} + \alpha \mathbf{y}') \\
			&\leq (1 - \alpha) \psi_t^{\cM}(\mathbf{x},\mathbf{y}) + \alpha \psi_t^{\cM}(\mathbf{x},\mathbf{y}') - \frac{1}{2t} \alpha(1 - \alpha) \norm{\mathbf{y}' - \mathbf{y}}_{L^2(\cM)^2}^2,
		\end{align*}
		so by rearranging and dividing by $\alpha$,
		\[
		0 \leq \psi_t^{\cM}(\mathbf{x},\mathbf{y}') -\psi_t^{\cM}(\mathbf{x},\mathbf{y}) - \frac{1}{2t}(1 - \alpha) \norm{\mathbf{y}' - \mathbf{y}}_{L^2(\cM)^2}^2,
		\]
		so taking $\alpha \to 1$, we obtain
		\[
		\psi_t^{\cM}(\mathbf{x},\mathbf{y}') - \psi_t^{\cM}(\mathbf{x},\mathbf{y}) \geq \frac{1}{2t} \norm{\mathbf{y}' - \mathbf{y}}_{L^2(\cM)}^2.
		\]
		Now in addition to $\mathbf{y} = \mathbf{f}_{t,R'}^{\cM}(\mathbf{x})$, assume that $\mathbf{y}' = \mathbf{f}_{t,R'}^{\cM}(\mathbf{x}')$.  Then symmetrically
		\[
		\psi_t^{\cM}(\mathbf{x}',\mathbf{y}) - \psi_t^{\cM}(\mathbf{x}',\mathbf{y}') \geq \frac{1}{2t} \norm{\mathbf{y}' - \mathbf{y}}_{L^2(\cM)}^2.
		\]
		Hence, adding the inequalities
		\begin{align*}
			\frac{1}{t} \norm{\mathbf{y}' - \mathbf{y}}_{L^2(\cM)^m}^2 &\leq \psi_t^{\cM}(\mathbf{x},\mathbf{y}') - \psi_t^{\cM}(\mathbf{x},\mathbf{y}) + \psi_t^{\cM}(\mathbf{x}',\mathbf{y}) - \psi_t^{\cM}(\mathbf{x}',\mathbf{y}') \\
			&= \varphi^{\cM}(\mathbf{y}') + \frac{1}{2t} \norm{\mathbf{x} - \mathbf{y}'}_{L^2(\cM)^m}^2 - \varphi^{\cM}(\mathbf{y}) - \frac{1}{2t} \norm{\mathbf{x} - \mathbf{y}}_{L^2(\cM)^m}^2 \\
			&\quad + \varphi^{\cM}(\mathbf{y}) + \frac{1}{2t} \norm{\mathbf{x}' - \mathbf{y}}_{L^2(\cM)^m}^2 - \varphi^{\cM}(\mathbf{y}') - \frac{1}{2t} \norm{\mathbf{x}' - \mathbf{y}'}_{L^2(\cM)^m}^2 \\
			&= \frac{1}{t} \ip{\mathbf{x}' - \mathbf{x}, \mathbf{y}' - \mathbf{y}}_{L^2(\cM)^m} \\
			&\leq \frac{1}{t} \norm{\mathbf{x'} - \mathbf{x}}_{L^2(\cM)^m} \norm{\mathbf{y}' - \mathbf{y}}_{L^2(\cM)^m},
		\end{align*}
		where we have used cancellation of the $\varphi^{\cM}$ terms, expanded each of the inner products, and then cancelled and recombined them.  Therefore,
		\[
		\norm{\mathbf{f}_{t,R'}^{\cM}(\mathbf{x}') - \mathbf{f}_{t,R'}^{\cM}(\mathbf{x})}_{L^2(\cM)^m} \leq \norm{\mathbf{y}' - \mathbf{y}}_{L^2(\cM)^m} \leq \norm{\mathbf{x'} - \mathbf{x}}_{L^2(\cM)^m},
		\]
		and so $\mathbf{f}_{t,R'}^{\cM}$ is clearly $1$-Lipschitz as desired.  Moreover, the uniqueness of the minimizer implies that $\mathbf{f}_{t,R'}$ is equivariant under unitary conjugation.
		
		All that remains is to estimate $\norm{\mathbf{f}_{t,R'}(\mathbf{0})}$.  Note that since $\varphi$ is a definable predicate, it has a universal upper and lower bound on each operator norm ball, and hence $s$ as defined in the statement of the proposition is finite.  By Lemma \ref{lem: gradient at zero}, there exists $\mathbf{y}_0 \in \bC^m \cap \underline{\nabla} \varphi^{\cM}(0)$ with $|\mathbf{y}_0| \leq C$.  Let $\mathbf{y} = \mathbf{f}_{t,R'}^{\cM}(\mathbf{0})$.  Since $\mathbf{f}_{t,R'}$ is equivariant under unitary conjugation, we know $\mathbf{y}$ is invariant under unitary conjugation and hence is in $\bC^m$.  Moreover,
		\begin{align*}
			0 &\geq \psi_t^{\cM}(\mathbf{0},\mathbf{y}) - \psi_t^{\cM}(\mathbf{0},\mathbf{0}) \\
			&= \varphi^{\cM}(\mathbf{y}) + \frac{1}{2t} \norm{\mathbf{y}}_{L^2(\cM)^m}^2 - \varphi^{\cM}(\mathbf{0}) \\
			&\geq \re \ip{\mathbf{y},\mathbf{y}_0}_{L^2(\cM)^m} + \frac{1}{2t} \norm{\mathbf{y}}_{L^2(\cM)^m}^2,
		\end{align*}
		which implies $|\mathbf{y}| \leq 2t|\mathbf{y}_0|/2t \leq 2Ct$.
		
		Thus, by Lemma \ref{lem: Lipschitz boundedness}, we obtain that for each $R > 0$, the function $\mathbf{f}_{t,R'}^{\cM}$ maps $(D_R^{\cM})^m$ into $(D_{2Ct+9\sqrt{m}R}^{\cM})^m$.  Thus, if $\mathbf{x} \in (D_R^{\cM})^m$, then
		\[
		\varphi_{t,R'}^{\cM}(\mathbf{x}) = \varphi_{t,C/2t+9\sqrt{m}R}^{\cM}(\mathbf{x}) \text{ for } R' \geq 2Ct +9\sqrt{m}R.
		\]
		Thus, the infimum over $(D_{2Ct+9\sqrt{m}R}^{\cM})^m$ is actually the global infimum, or
		\[
		\varphi_t^{\cM}(\mathbf{x}) = \varphi_{t,2Ct +9\sqrt{m}R}^{\cM}(\mathbf{x}) \text{ for } \mathbf{x} \in (D_R^{\cM})^m.
		\]
		
		(2) We know that $\varphi_t$ is a definable predicate with respect to $\rT_{\tr,\factor}$ since it agrees on each domain of quantification with a definable predicate $\varphi_{t,R'}$.  The convexity of $\varphi_t$ follows because it is the infimum over $\mathbf{y}$ of $\psi_t$ which is jointly convex in $(\mathbf{x},\mathbf{y})$.  For semiconcavity, note that \cite[Proposition 5.8 (1)]{JekelTypeCoupling}) shows that $\varphi_{t,R'}$ is $1/t$-semiconcave for each $R'$, and hence since $\varphi_t$ agrees with $\varphi_{t,R'}$ for sufficiently large $R'$, it follows that $\varphi_t$ is $1/t$-semiconcave.
		
		(3), (4) From Corollary \ref{cor: definable gradient}, since $\varphi$ is convex and $1/t$-semiconcave, we see that $\nabla \varphi_t$ is a definable function which is $1/t$-Lipschitz.
		
		We can relate $\nabla \varphi_t$ and the minimizer in the definition of $\varphi_t$ as follows.  First, note that for $\mathbf{x} \in (D_R^{\cM})^m$, the minimizer $\mathbf{f}_{t,R'}^{\cM}(\mathbf{x})$ is independent of $R'$ provided it is larger than $2Ct + 9\sqrt{m}R$; this follows because $\varphi_{t,R'}^{\cM}(\mathbf{x})$ is independent of $R'$ and for each $R'$, the minimizer is unique.  Therefore, denote by $\mathbf{f}_t^{\cM}(\mathbf{x})$ the common value of $\mathbf{f}_{t,R'}^{\cM}(\mathbf{x})$ for $R > 2Ct + 9\sqrt{m}R$.  Then one can show that
		\[
		\mathbf{f}_{t,R'}^{\cM}(\mathbf{x}) = \mathbf{x} - t \nabla \varphi_t^{\cM}(\mathbf{x}).
		\]
		This is a classical fact about inf-convolutions for functions on a Hilbert space, which is proved as follows:  Let $\mathbf{y} = \mathbf{f}_t^{\cM}(\mathbf{x})$ be the minimizer associated to $\mathbf{x}$.  Then for $\mathbf{x}' \in \cM^m$, we have
		\begin{align*}
			\varphi_t^{\cM}(\mathbf{x}') &\leq \varphi^{\cM}(\mathbf{y}) + \frac{1}{2t} \norm{\mathbf{x}' - \mathbf{y}}_{L^2(\cM)^m}^2 \\
			&= \varphi_t^{\cM}(\mathbf{x}) - \frac{1}{2t} \norm{\mathbf{x} - \mathbf{y}}_{L^2(\cM)^m}^2 + \frac{1}{2t} \norm{\mathbf{x}' - \mathbf{y}}_{L^2(\cM)^m}^2 \\
			&= \varphi_t^{\cM}(\mathbf{x}) + \frac{1}{t} \re \ip{\mathbf{x}' - \mathbf{x}, \mathbf{y} - \mathbf{x}}_{L^2(\cM)^m} + \frac{1}{2t} \norm{\mathbf{x}' - \mathbf{x}}_{L^2(\cM)^m}^2.
		\end{align*}
		Therefore, $\frac{1}{t}(\mathbf{y} - \mathbf{x}) \in \overline{\nabla} \varphi_t^{\cM}(\mathbf{x})$.  Since $\varphi_t^{\cM}$ is differentiable, $\frac{1}{t}(\mathbf{y} - \mathbf{x}) = \nabla \varphi_t^{\cM}(\mathbf{x})$ as desired.
		
		In particular, we have $\nabla \varphi_t^{\cM}(\mathbf{0}) = (1/t) \mathbf{f}_t^{\cM}(\mathbf{0})$.  Hence, our earlier argument for (1) shows that $\mathbf{f}_t^{\cM}(\mathbf{0}) \in \bC^m$ and $|\nabla \varphi_t^{\cM}(\mathbf{0})| \leq 2C$.
	\end{proof}
	
	\begin{lemma} \label{lem: Legendre transform definable predicate}
		Let $\varphi$ be an $m$-variable definable predicate for $\rT_{\tr,\factor}$ that is $c$-strongly convex for some $c > 0$.  Let
		\[
		\mathcal{L} \varphi^{\cM}(\mathbf{y}) = \sup_{\mathbf{x} \in \cM^m} \left[ \re \ip{\mathbf{x}, \mathbf{y}}_{L^2(\cM)^m} - \varphi^{\cM}(\mathbf{x}) \right].
		\]
		Then $\mathcal{L} \varphi$ is a definable predicate that is convex and $1/c$-semiconcave.
	\end{lemma}
	
	\begin{proof}
		Since $\varphi$ is uniformly convex, let $\tilde{\varphi} = \varphi - cq$ where $q$ is the quadratic function $(1/2) \sum_j \tr(x_j^*x_j)$.  Then
		\begin{align*}
			\re \ip{\mathbf{x},\mathbf{y}}_{L^2(\cM)^m} - \varphi^{\cM}(\mathbf{x}) &= \re \ip{\mathbf{x},\mathbf{y}}_{L^2(\cM)^m} - \frac{c}{2} \norm{\mathbf{x}}_{L^2(\cM)^m}^2 - \tilde{\varphi}^{\cM}(\mathbf{x}) \\
			&= \frac{1}{2c} \norm{\mathbf{y}}_{L^2(\cM)^m}^2 - \left[ \frac{c}{2} \norm{\mathbf{x} - c^{-1} \mathbf{y}}_{L^2(\cM)^m}^2 + \tilde{\varphi}^{\cM}(\mathbf{x}) \right].
		\end{align*}
		Thus, if $\tilde{\varphi}_{1/c}$ is the inf-convolution described in Proposition \ref{prop: definable inf-convolution}, we obtain
		\[
		\mathcal{L} \varphi^{\cM}(\mathbf{y}) = \frac{1}{2c} \norm{\mathbf{y}}_{L^2(\cM)^m}^2 - \tilde{\varphi}_{1/c}^{\cM}(c^{-1} \mathbf{y}).
		\]
		This shows that $\mathcal{L} \varphi$ is finite everywhere and is a definable predicate.  Moreover, since $\tilde{\varphi}_{1/c}$ is convex by Fact \ref{lem: inf-convolution convexity}, hence also $\tilde{\varphi}_{1/c}(c^{-1} (\cdot))$ is convex, we see that $\mathcal{L} \varphi^{\cM}$ is $1/c$-semiconcave.  Moreover, since $\tilde{\varphi}_{1/c}$ is $c$-semiconcave by Proposition \ref{prop: definable inf-convolution}, we obtain that $\tilde{\varphi}_{1/c}(c^{-1}(\cdot))$ is $c/c^2 = 1/c$-semiconcave, and therefore $\mathcal{L} \varphi$ is convex.  (Alternatively, convexity of $\mathcal{L} \varphi$ follows directly because it is a supremum of affine functions.)
	\end{proof}

	\begin{proposition} \label{prop: displacement duality}
		Let $(\mathbf{x}_0,\mathbf{x}_1)$ be an optimal coupling of types $\mu, \nu \in \mathbb{S}_m(\rT)$ where $\rT$ is the theory of some tracial factor, and let $\mathbf{x}_t = (1-t)\mathbf{x}_0 + t \mathbf{x}_1$.  Then for $0 \leq s \leq t \leq 1$, the pair $(\mathbf{x}_s,\mathbf{x}_t)$ is an optimal coupling of the associated types. Moreover, there exist convex definable predicates $\varphi_{s,t}$ and $\psi_{s,t}$ with respect to $\rT_{\tr,\factor}$ witnessing the Monge-Kantorovich duality for $\tp^{\cM}(\mathbf{x}_s)$ and $\tp^{\cM}(\mathbf{x}_t)$ such that
		\begin{enumerate}
			\item $\varphi_{s,t}$ is $t/s$-semiconcave for $s > 0$ and $(1-t) / (1-s)$-strongly convex for $t < 1$,
			\item $\psi_{s,t}$ is $(1-s)/(1-t)$-semiconcave for $t < 1$ and $s/t$-strongly convex for $s > 0$.
		\end{enumerate}
	\end{proposition}
	
	\begin{proof}
		Let $\varphi$, $\psi$ be convex definable predicates associated to the optimal coupling as in Theorem \ref{thm: MK duality}.  Then define $\varphi_{s,t}$ and $\psi_{s,t}$ as in Proposition \ref{prop: abstract interpolation 2}, namely
		\begin{align*}
			\varphi_{s,t}^{\cM}(\mathbf{x}) &= \inf_{\mathbf{x}' \in \cM^m} \biggl[ \frac{t}{2s} \norm{\mathbf{x}}_{L^2(\cM)^m}^2 - \frac{t-s}{s} \re  \ip{\mathbf{x},\mathbf{x}'}_{L^2(\cM)^m} \\
			&\quad + \frac{(t-s)(1-s)}{2s} \norm{\mathbf{x}'}_{L^2(\cM)^m}^2 + (t-s) \varphi^{\cM}(\mathbf{x}') \biggr] \text{ when } s > 0, \\
			\varphi_{0,t}^{\cM}(\mathbf{x}) &= \frac{1-t}{2} \norm{\mathbf{x}}_{L^2(\cM)^m}^2 + t \varphi^{\cM}(x),
		\end{align*}
		and
		\begin{align*}
			\psi_{s,t}^{\cM}(\mathbf{y}) &= \inf_{\mathbf{y}' \in \cM^m} \biggl[ \frac{1-s}{2(1-t)} \norm{\mathbf{y}}_{L^2(\cM)^m}^2 - \frac{t-s}{1-t} \re \ip{\mathbf{y},\mathbf{y}'}_{L^2(\cM)^m} \\
			& \quad + \frac{(t-s)t}{2(1-t)} \norm{y'}_{L^2(\cM)^m}^2 + (t-s) \psi^{\cM}(\mathbf{y}') \biggr] \text{ when } t < 1 \\
			\psi_{s,1}^{\cM}(\mathbf{y}) &= \frac{s}{2} \norm{\mathbf{y}}_{L^2(\cM)^m}^2 + (1-s) \psi^{\cM}(y).
		\end{align*}
		To show that $\varphi_{s,t}$ is a definable predicate, note that for $0 < s \leq t$ and $s < 1$, we have
		\[
		\varphi_{s,t}^{\cM}(\mathbf{x}) = \left( \frac{t}{2s} - \frac{t-s}{2s(1-s)} \right) \norm{\mathbf{x}}_{L^2(\cM)^m}^2 + \frac{t-s}{1-s} \inf_{\mathbf{x}' \in \cM^m} \left[ \frac{1}{2s} \norm{\mathbf{x} - \mathbf{x}''}_{L^2(\cM)^m}^2  + (1-s) \varphi^{\cM}((1-s)^{-1}\mathbf{x}'') \right],
		\]
		as in the proof of Proposition \ref{prop: abstract interpolation 2}.  This is a quadratic plus an inf-convolution of a convex definable predicate, which is a definable predicate by Proposition \ref{prop: definable inf-convolution}.  In the case $s = 1$, we also have $t = 1$, so that $\varphi_{s,t}^{\cM}$ is the quadratic $q$.  In the case $s = 0$, $\varphi_{s,t}$ is a linear combination of $\varphi$ and $q$ and hence is a definable predicate.  The argument that $\psi_{s,t}$ is a definable predicate is symmetrical.  The fact that $\varphi_{s,t}$ and $\psi_{s,t}$ witness Monge-Kantorovich duality for $\tp^{\cM}(\mathbf{x}_s)$ and $\tp^{\cM}(\mathbf{x}_t)$, and that $\mathbf{x}_s$ and $\mathbf{x}_t$ are an optimal couplings, follows from Proposition \ref{prop: abstract interpolation 2} (1) and (4).  Likewise, the asserted statements about semi-convexity and semi-concavity follow from Proposition \ref{prop: abstract interpolation 2} (2) and (3).
	\end{proof}

	\section{Free entropy and geodesics} \label{sec: entropy along geodesics}
	
	\subsection{Definitions of free entropy quantities} \label{subsec: entropy definitions}
	
	The microstates framework for free entropy measures the amount of matrix approximations for a tuple $\mathbf{x}$, or more precisely, the matrix $m$-tuples with approximately the same non-commutative distribution as $\mathbf{x}$ \cite{VoiculescuFE2}.  In this work, we use neighborhoods in the space of types $\mathbb{S}_{m,R}(\rT_{\tr,\factor})$ to define these microstate spaces, as in \cite{JekelCoveringEntropy,JekelModelEntropy} rather than neighborhoods in the space of non-commutative laws as in the original definition of free entropy.  We therefore also attach the subscript $\full$ to indicate that these are the versions of entropy for the full type as opposed to the quantifier-free type (i.e.\ non-commutative law).
	
	\begin{definition}[Microstate spaces]
		Let $\cO$ be an open subset of $\mathbb{S}_{m,R}(\rT_{\tr,\factor})$.  Then we set
		\[
		\Gamma_R^{(n)}(\cO) = \{\mathbf{X} \in \bM_n^m: \norm{X_j} \leq R, \tp^{\bM_n}(\mathbf{X}) \in \cO \}.
		\]
	\end{definition}
	
	Various versions of free entropy can be defined by measuring the ``size'' of $\Gamma_R^{(n)}(\cO)$ as $n \to \infty$, or in the present work, as $n \to \cU$ for some free ultrafilter $\cU$ on $\bN$.   Free entropy itself is defined using the Lebesgue measure of the microstate spaces, while free entropy dimension and metric entropy are defined in terms of covering numbers.
	
	\begin{definition}[Covering numbers]
		For each $\varepsilon > 0$ and a subset $\Omega$ of $\bM_n^m$, let $K_{\varepsilon}(\Omega)$ be the minimum number of $\varepsilon$-balls with respect to $\norm{\cdot}_{\tr_n}$ that cover $\Omega$.
		
		We also use \emph{orbital covering numbers} defined as follows.  For $\mathbf{X} \in \bM_n^m$ and $U \in U(\bM_n)$, let
		\[
		U\mathbf{X}U^* = (UX_1U^*,\dots,UX_mU^*).
		\]
		For $\Omega \subseteq \bM_n^m$, let $N_{\varepsilon}^{\orb}(\Omega)$ be the set of $\mathbf{X}$ such that there exists $\mathbf{Y} \in \Omega$ and $U \in U(\bM_n)$ with $\norm{\mathbf{X} - U\mathbf{Y}U^*}_{\tr_n} < \varepsilon$.
	\end{definition}
	
	We start with the definition of metric entropy from \cite{JekelCoveringEntropy}, which is the analog for full types of the Jung--Hayes entropy from \cite{Jung2007S1B,Hayes2018}.
	
	\begin{definition}[Metric entropy]
		The metric entropy (or $1$-bounded entropy) is defined as follows.  Fix a free ultrafilter $\cU$ on $\bN$.  For $\cO \subseteq \mathbb{S}_{m,R}(\rT_{\tr,\factor})$, let
		\[
		\Ent_\varepsilon^{\cU}(\cO) = \lim_{n \to \cU} \frac{1}{n^2} \log K_{\varepsilon}^{\orb}(\Gamma_R^{(n)}(\cO)).
		\]
		Then for $\mu \in \mathbb{S}_{m,R}(\rT_{\tr,\factor})$, let
		\[
		\Ent_{\varepsilon}^{\cU}(\mu) = \inf_{\cO \ni \mu} \Ent_\varepsilon^{\cU}(\cO),
		\]
		where $\cO$ ranges over all neighborhoods of $\mu$ in $\mathbb{S}_{m,R}(\rT_{\tr,\factor})$.  Finally, let
		\[
		\Ent_{\full}^{\cU}(\mu) = \sup_{\varepsilon > 0} \Ent_{\varepsilon}^{\cU}(\mu).
		\]
		We remark that this is independent of $R$ on account of \cite[Corollary 4.9]{JekelCoveringEntropy}.
	\end{definition}
	
	Next, the free entropy dimension for types is the analog of Voiculescu's free entropy dimension from \cite{VoiculescuFE2}, which was re-expressed in terms of covering numbers by Jung \cite{Jung2003FED}.
	
	\begin{definition}[Free entropy dimension for types]
		Fix a free ultrafilter $\cU$ on $\bN$.  For $\mu \in \mathbb{S}_{m,R}(\rT_{\tr,\factor})$, let
		\[
		\delta_{\full}^{\cU}(\mu) = \limsup_{\varepsilon \searrow 0} \frac{1}{\log(1/\varepsilon)} \inf_{\cO \ni \mu} \lim_{n \to \cU} \frac{1}{n^2} \log K_{\varepsilon}(\cO).
		\]
		One can also show that this definition is independent of $R$.
	\end{definition}
	
	Finally, the free entropy for types is defined as follows \cite{JekelModelEntropy}.
	
	\begin{definition}[Free entropy for types]
		For $\mu \in \mathbb{S}_{m,R}(\rT_{\tr,\factor})$, let
		\[
		\chi_{\full}^{\cU}(\mu) = \inf_{\cO \ni \mu} \lim_{n \to \cU} \left[ \frac{1}{n^2} \log \vol \Gamma_R^{(n)}(\cO) + 2m \log n \right],
		\]
		where $\vol$ denotes the Lebesgue measure.\footnote{The normalization for the Lebesgue measure is based on the inner product associated to $\tr_n$ rather than $\Tr_n$, as described in \S \ref{subsec: operator algebras}; this results in the coefficient of $\log n$ being $2m$ rather than $m$ for the non-self-adjoint setting (similarly, it would be $m$ rather than $m/2$ for the self-adjoint setting).  See \cite{JekelPi2024} for further discussion.}  Again, this is independent of $R$.
	\end{definition}
	
	The free entropy $\chi_{\full}^{\cU}$ intuitively describes the large-$n$ limit of normalized entropy for classical random matrix models.  Recall that the differential entropy $h$ of a probability distribution $\mu$ on $\bM_n^m$ with density $\rho$ is
	\[
	h(\mu) = -\int_{\bM_n^m} \rho \log \rho\,d\mathbf{X}.
	\]
	Moreover, if $\mu$ does not admit a density, then $h(\mu) = -\infty$ by definition.  Then consider a normalized version of entropy
	\begin{equation} \label{eq: normalized entropy}
		h^{(n)}(\mu) = \frac{1}{n^2} h(\mu) + 2m \log n.
	\end{equation}
	For ease of notation, we also write $h^{(n)}(\mathbf{X}) = h^{(n)}(\mu)$ when $\mathbf{X}$ is a random variable with distribution $\mu$.
	
	\begin{proposition} \label{prop: entropy of matrix model}
		Fix $\cU$ and $\mu \in \mathbb{S}_{m,R}(\rT_{\cU})$.  Then there exists a multi-matrix model $\mathbf{X}^{(n)}$ with $\norm{X_j^{(n)}} \leq R$ such that
		\[
		\lim_{n \to \cU} h^{(n)}(\mathbf{X}^{(n)}) = \chi_{\full}^{\cU}(\mu).
		\]
		Moreover, for every random multi-matrix model $\mathbf{X}^{(n)}$ with $\norm{X_j^{(n)}} \leq R$ and such that $\lim_{n \to \cU} \tp^{\bM_n}(\mathbf{X}^{(n)}) = \mu$ in probability,\footnote{See Lemma \ref{lem: Wasserstein distance of matrix models} for the definition of convergence in probability for an ultrafilter.} we have
		\[
		\lim_{n \to \cU} h^{(n)}(\mathbf{X}^{(n)}) \leq \chi_{\full}^{\cU}(\mu).
		\]
		This in fact holds more generally when $\norm{\mathbf{X}_j^{(n)}}$ is not necessarily uniformly bounded, but satisfies an estimate of the form
		\[
		P(\norm{\mathbf{X}_j^{(n)}} \geq R + \delta) \leq e^{-cn \delta^2} \text{ for } \delta > 0,
		\]
		for some constants $c, R > 0$.
	\end{proposition}
	
	This proposition is proved exactly as in \cite[Proposition B.7]{ST2022} and \cite[Theorem 4.8]{JekelPi2024}.  In order to achieve equality, one can take $\mathbf{X}^{(n)}$ to have the uniform distribution on $\Gamma_R^{(n)}(\cO_{k(n)})$ where $\cO_{k(n)}$ is an appropriately chosen neighborhood.  This is the analog of a ``microcanonical ensemble'' in the theory of Shannon entropy.
	
	\subsection{Entropy along geodesics} \label{subsec: entropy along geodesics}
	
	Now we proceed with the proof of Theorem \ref{thm: entropy along geodesics}. Because in Theorem \ref{thm: entropy along geodesics} $\mathbf{x}$ and $\mathbf{y}$ are definable functions of $\mathbf{x}_t$ for $t \in (0,1)$, claim (1) about the metric entropy $\Ent_{\full}^{\cU}$ will be immediate from the following fact, which is a special case of \cite[Proposition 4.7]{JekelCoveringEntropy}.
	
	\begin{proposition}[Monotonicity of $\Ent^{\cU}$] \label{prop: monotonicity of metric entropy}
		Let $\mu \in \mathbb{S}_{m}(\rT_{\cU})$. Let $\mathbf{f} = (f_1,\dots,f_{m'})$ be a definable function with respect to $\rT_{\tr,\factor}$.  Then
		\[
		\Ent^{\cU}(\mathbf{f}_* \mu) \leq \Ent^{\cU}(\mu).
		\]
	\end{proposition}
	
	For claim (2) of the theorem concerning the microstates free entropy dimension, we will show a similar result about the behavior of free entropy dimension under pushforwards, provided that the definable function in question is Lipschitz.
	
	\begin{lemma} \label{lem: MFED Lipschitz pushforward}
		Let $\mathbf{f} = (f_1,\dots,f_{m'})$ be an $m'$-tuple of definable functions in $m$ variables.  Suppose that $f$ is $L$-Lipschitz.  Fix an ultrafilter $\cU$, let $\cQ$ be the associated matrix ultraproduct, and let $\mu \in \mathbb{S}_m(\rT_{\cU})$.  Then $\delta_{\full}^{\cU}(\mathbf{f}_* \mu) \leq \delta_{\full}^{\cU}(\mu)$.
	\end{lemma}
	
	\begin{proof}
		Fix $R$ large enough that $\norm{X_j} \leq R$ when $\tp^{\cQ}(\mathbf{X}) = \mu$.  By Lemma \ref{lem: Lipschitz boundedness}, since $\cQ$ is a $\mathrm{II}_1$ factor, we have $\norm{f_j(\mathbf{X})} \leq \operatorname{const} L R$.  Let $\nu = f_* \mu$.  Fix $\varepsilon > 0$, and fix a neighborhood $\cO$ of $\mu$ in $\Sigma_{m,R}(\Th(\cQ))$.  Since $\mathbb{S}_{m,R}(\Th(\cQ))$ is a compact metrizable space, there is a nonnegative definable predicate $\varphi$ such that $(\varphi,\mu) = 1$ and $\varphi$ vanishes on types in $\mathbb{S}_{m,R}(\Th(\cQ)) \setminus \cO$.  Define
		\[
		\psi(\mathbf{Y}) = \sup_{\mathbf{X} \in (D_R)^m} \left[ \varphi(\mathbf{X})\left(1 - \frac{1}{\varepsilon^2} \sum_{j=1}^{m'} \tr(|Y_j - f_j(\mathbf{X})|^2) \right) \right].
		\]
		Then let
		\[
		\cO' = \{ \lambda \in \mathbb{S}_{m,LR}(\Th(\cQ)): (\psi,\lambda) > 0 \}.
		\]
		Observe that $\nu \in \cO'$; indeed, if $\mathbf{x}$ has type $\mu$, then $\psi(\mathbf{f}(\mathbf{x})) \geq \varphi(\mathbf{x}) = 1$.
		Moreover, if $\mathbf{x}$ is an $m'$-tuple in some tracial von Neumann algebra $\cM$ and $\tp^{\cM}(\mathbf{x}) \in \cO'$, then by construction of $\psi$, there exists $\mathbf{x} \in (D_R^{\cM})^m$ with $\norm{\mathbf{f}(\mathbf{x}) - \mathbf{y}}_{L^2(\cM)} < \varepsilon$ and $\varphi(\mathbf{x}) > 0$, hence $\tp^{\cM}(\mathbf{x}) \in \cO$.  Therefore, we have
		\[
		\Gamma_{LR}^{(n)}(\cO') \subseteq N_\varepsilon(\mathbf{f}^{\bM_n}(\Gamma_R^{(n)}(\cO))).
		\]
		In particular, using the Lipschitz nature of $\mathbf{f}$,
		\[
		K_{(L+1)\varepsilon}^{\orb}(\Gamma_{LR}^{(n)}(\cO')) \leq K_{L \varepsilon}(\mathbf{f}^{\bM_n}(\Gamma_R^{(n)}(\cO))) \leq K_{\varepsilon}(\Gamma_R^{(n)}(\cO)).
		\]
		Since $\cO$ was arbitrary, we obtain
		\[
		\inf_{\cO'} \lim_{n \to \cU} \frac{1}{n^2} K_{(L+1)\varepsilon}(\Gamma_R^{(n)}(\cO')) \leq \inf_{\cO} \lim_{n \to \cU} \frac{1}{n^2} \log K_{\varepsilon}(\Gamma_R^{(n)}(\cO)).
		\]
		Hence,
		\begin{align*}
			\delta_{\full}(\nu) &= \limsup_{\varepsilon \searrow 0} \frac{1}{-\log((L+1) \varepsilon)} \inf_{\cO'} \lim_{n \to \cU} \frac{1}{n^2} K_{(L+1)\varepsilon}(\Gamma_R^{(n)}(\cO')) \\
			&\leq \limsup_{\varepsilon \searrow 0} \frac{-\log \varepsilon}{-\log((L+1) \varepsilon)} \frac{1}{-\log \varepsilon} \inf_{\cO} \lim_{n \to \cU} \frac{1}{n^2} \log K_{\varepsilon}(\Gamma_R^{(n)}(\cO)) \\
			&\leq 1 \cdot \delta_{\full}(\mu).
		\end{align*}
	\end{proof}
	
	Claim (3) of Theorem \ref{thm: entropy along geodesics} requires more work to prove.  Recall we are trying to estimate $\chi_{\full}^{\cU}(\mu_t)$ in terms of the value $\chi_{\full}^{\cU}(\nu)$ at the endpoint; we know $\nu$ is a Lipschitz pushforward of $\mu_t$ but not the other way around.  While for free entropy dimension, it was sufficient to arrange that $\Gamma_{LR}^{(n)}(\cO')$ was in an $\varepsilon$-neighborhood of $f^{\bM_n}(\Gamma_R^{(n)}(\cO))$, this is not sufficient to control the Lebesgue measure of the microstate spaces in order to estimate $\chi_{\full}^{\cU}$.  Indeed, two sets can be contained in $\varepsilon$-neighborhoods of each other but have vastly different Lebesgue measures.  To ameliorate this issue, we want some sort of inverse for the function $f$ that gives the pushforward, but the existence of a definable function inverse to $f$ is not guaranteed since we are concerned with the type $\nu$ at the endpoint of the geodesic.
	
	However, we will actually show that for the random matrix models there is something like a measurable right inverse of $f^{\bM_n}$.  The following lemma accomplishes this by ``lifting'' an optimal coupling of types in $\cQ$ to a classical optimal coupling of probability measures on $\bM_n^m$.  The key ingredient is the Monge-Kantorovich duality for types, and here it is essential that the functions in the Monge-Kantorovich duality are definable predicates, which are weak-$*$ continuous on the space of types.  The tracial $\mathrm{W}^*$-functions of \cite{GJNS2021} used in the Monge-Kantorovich duality for quantifier-free types do not satisfy such weak-$*$ continuity, and we would not be able to accomplish the lifting construction with them.
	
	%\begin{theorem} \label{thm: displacement concavity}
	%Let $\mathbf{r} \in (0,\infty)^m$, and let $\mu, \nu \in \mathbb{S}_{\mathbf{r}}(\mathcal{Q})$.  Let $X, Y$ be an optimal coupling of $\mu$, $\nu$. Then we have continuity and concavity of the mapping
	%\[
	%[0,1] \to [-\infty,\infty): \qquad t \mapsto \chi^{\cU}(\tp^{\cQ}((1-t)X + tY)).
	%\]
	%\end{theorem}
	
	%While the proposition mostly follows from the same methods as the classical case, special care has to be taken in estimating $\chi$ for the endpoints $\mu$ and $\nu$.  For this we will use the following lemma which shows that an optimal coupling in $\cQ$ can be lifted to optimal couplings of random matrix models.
	
	\begin{lemma} \label{lem: lifting}
		Let $R > 0$, and let $\mu, \nu \in \mathbb{S}_{m,R}(\rT_{\cU})$.  Let $(\mathbf{x}, \mathbf{y})$ be an optimal coupling of $\mu$, $\nu$ in $\cQ = \prod_{n \to \cU} \bM_n$.  For $t \in (0,1)$, let $\mathbf{x}_t = (1-t)\mathbf{x} + t\mathbf{y}$.  Let $\mathbf{Y}^{(n)}$ be any sequence of random matrix models such that $\norm{Y_j^{(n)}} \leq R$ for each $j$ and $\tp^{\bM_n}(\mathbf{Y}^{(n)}) \to \nu$ in probability.  Fix $t \in (0,1)$.  Then there exists a random matrix tuple $\mathbf{X}_t^{(n)}$ such that
		\begin{enumerate}[(1)]
			\item $\tp^{\bM_n}(\mathbf{X}_t^{(n)},\mathbf{Y}^{(n)}) \to \tp^{\cQ}(\mathbf{X}_t,\mathbf{Y})$ in probability as $n \to \cU$.
			\item For each $n$, $(\mathbf{X}_t^{(n)},\mathbf{Y}^{(n)})$ is an optimal coupling of the probability distributions of $\mathbf{X}_t^{(n)}$ and $\mathbf{Y}^{(n)}$ on $\bM_n^m$.
			\item $\mathbf{X}_t^{(n)} = F^{(n)}(\mathbf{Y}^{(n)})$ for some Borel measurable function $F^{(n)}$.
		\end{enumerate}
	\end{lemma}
	
	\begin{proof}
		Fix convex definable predicates $\phi$ and $\psi$ witnessing the Monge-Kantorovich duality for $\mu$ and $\nu$.  Let $\psi_t(\mathbf{y}) = (1-t)\psi(\mathbf{y}) + (t/2)\norm{\mathbf{y}}_{\tr_n}^2$, and let $\phi_t$ be its Legendre transform.  Thus, $\phi_t$ and $\psi_t$ satisfy that
		\[
		\phi_t^{\cM}(\mathbf{x}) + \psi_t^{\cM}(\mathbf{y}) \geq \re \ip{\mathbf{x},\mathbf{y}}_{L^2(\cM)}
		\]
		for all tracial factors $\cM$, and $\psi_t$ is $t$-strongly convex and $\phi_t$ is $1/t$-semiconcave.
		
		Furthermore, let $\mu_t$ be the type of $(1-t) \mathbf{x} + t \mathbf{y}$.  By Remark \ref{rem: definable predicates and continuous functions}, fix a nonnegative definable predicate $\eta$ with values such that for $\sigma \in \mathbb{S}_{m,R}(\mathrm{T}_{\tr,\factor})$, we have $(\sigma,\eta) = 0$ if and only if $\sigma = \mu_t$.  Note also that by compactness of $\mathbb{S}_{\mathbf{r}}(\mathrm{T}_{\text{factor}})$, we have that for any neighborhood $\mathcal{O}$ of $\mu_t$, there is some $\delta > 0$ such that $(\sigma,\eta) < \delta$ implies $\sigma \in \mathcal{O}$ for $\sigma \in \mathbb{S}_{\mathbf{r}}(\mathrm{T}_{\tr,\factor})$.
		
		Let
		\[
		\widehat{\psi}^{\cM}(\mathbf{y}') = \sup_{\mathbf{x}' \in (D_{R}^{\cM})^m} \left[ \re \ip{\mathbf{x}',\mathbf{y}'}_{L^2(\cM)} -   \phi_t^{\cM}(\mathbf{x}') - \eta^{\cM}(\mathbf{x}') \right],
		\]
		which is convex definable predicate.
		Furthermore, for each $\mathbf{Y} \in (D_R^{\bM_n})^m$, let $A^{(n)}(\mathbf{Y})$ be the set of $\mathbf{X} \in (D_{R}^{\bM_n})^m$ where the supremum is achieved, which is nonempty and compact because the functions in the optimization problem are continuous and the domain $(D_{R}^{\bM_n})^m$ is compact.  Our next goal is to make a Borel-measurable selection of some $\mathbf{X} \in A^{(n)}(\mathbf{Y})$ for each $\mathbf{Y}$.  By the Kuratowski--Ryll-Nardzewski theorem, it suffices to show that for each open set $O$, the set
		\[
		\{\mathbf{Y}: A^{(n)}(\mathbf{Y}) \cap O \neq \varnothing \}
		\]
		is Borel-measurable.  In fact, since an open set can be written as a countable union of closed sets, we can replace the open set $O$ with a closed set $K$.  Note that $\widehat{\psi}^{\bM_n}$ is continuous and hence
		\[
		G := \{(\mathbf{X},\mathbf{Y}): \mathbf{X} \in A^{(n)}(\mathbf{Y})\} = \{\ip{\mathbf{X},\mathbf{Y}}_{\tr_n} -   \phi_t^{\bM_n}(\mathbf{X}) - \eta^{\bM_n}(\mathbf{X}) = \widehat{\psi}^{\bM_n}(\mathbf{Y})\}
		\]
		is a closed set.  Now for a closed set $K$,
		\[
		\{\mathbf{Y}: A^{(n)}(\mathbf{Y}) \cap K \neq \varnothing \} = \pi_2(G \cap (K \times (D_{\mathbf{r}}^{\bM_n})),
		\]
		where $\pi_2(\mathbf{X},\mathbf{Y}) = \mathbf{Y})$, and this is a continuous image of a compact set, hence closed.  Thus, by the Kuratowski--Ryll-Nardzewski measurable selection theorem \cite{KRN1965}, there exists a Borel-measurable $F^{(n)}: D_{\mathbf{r}}^{\bM_n} \to D_{\mathbf{r}}^{\bM_n}$ such that $F^{(n)}(\mathbf{Y}) \in A^{(n)}(\mathbf{Y})$ for all $\mathbf{Y} \in D_{\mathbf{r}}^{\bM_n}$.
		
		Now consider the random matrix ensemble $\mathbf{X}^{(n)} := F^{(n)}(\mathbf{Y}^{(n)})$.  Because $\mathbf{X}^{(n)}$ is a maximizer in the definition of $\widehat{\psi}^{\bM_n}$, we have that $\mathbf{X}^{(n)} \in \underline{\nabla} \widehat{\psi}^{\bM_n}$, and thus since $\widehat{\psi}$ is convex, the classical Monge-Kantorovich duality implies that $(\mathbf{X}^{(n)},\mathbf{Y}^{(n)})$ is an optimal coupling.
		
		It remains to show that $\tp^{\bM_n}(\mathbf{X}^{(n)},\mathbf{Y}^{(n)})$ converges in probability to $\tp^{\cQ}(\mathbf{x}_t,\mathbf{y})$.  Fix a definable predicate $\tilde{\eta}$ with values in $[0,1]$ such that for $\sigma \in \mathbb{S}(\mathrm{T}_{\text{factor}})$ we have $\sigma[\tilde{\eta}] = 0$ if and only if $\sigma = \nu$.  Then consider the definable predicate
		\[
		\omega^{\cM}(\mathbf{x}',\mathbf{y}') = \phi_t^{\cM}(\mathbf{x}') + \psi_t^{\cM}(\mathbf{y}') - \re \ip{\mathbf{x}',\mathbf{y}'}_{L^2(\cM)} + \eta^{\cM}(\mathbf{x}') + \tilde{\eta}(\mathbf{y}')
		\]
		We claim that $\omega^{\cM}(\mathbf{x}',\mathbf{y}') \geq 0$ with equality if and only if $\tp^{\cM}(\mathbf{x}',\mathbf{y}') = \tp^{\cQ}(\mathbf{x}_t,\mathbf{y})$.  Nonnegativity is immediate from the fact that $\phi_t^{\cM}(\mathbf{x}') + \psi_t^{\cM}(\mathbf{y}') - \re \ip{\mathbf{x}',\mathbf{y}'}_{L^2(\cM)^m} \geq 0$.  Also, by construction this becomes zero when we evaluate on $(\mathbf{x}_t,\mathbf{y})$ in $\cQ$.  Lastly, suppose that $\omega^{\cM}(\mathbf{x}',\mathbf{y}') = 0$.  From nonnegativity of $\phi_t^{\cM}(\mathbf{x}') + \psi_t^{\cM}(\mathbf{y}') - \re \ip{\mathbf{x}',\mathbf{y}'}_{L^2(\cM)}$, we obtain that $\eta^{\cM}(\mathbf{x}') = 0$ and $\tilde{\eta}^{\cM}(\mathbf{y}') = 0$ and hence $\tp^{\cM}(\mathbf{x}') = \mu_t$ and $\tp^{\cM}(\mathbf{y}') = \nu$.  Furthermore, since $\phi_t^{\cM}(\mathbf{x}') + \psi_t^{\cM}(\mathbf{y}') - \re \ip{\mathbf{x}',\mathbf{y}'}_{L^2(\cM)}$ must be zero, we have that $\mathbf{y}' \in \underline{\nabla} \phi_t(\mathbf{x}')$.  Since $\phi_t$ is differentiable and $\nabla \phi_t$ is a definable function by semiconcavity and Corollary \ref{cor: definable gradient}, we get that $\mathbf{y}' = \nabla \phi_t^{\cM}(\mathbf{x}')$.  Thus, since $\tp^{\cM}(\mathbf{x}') = \tp^{\cQ}(\mathbf{x}_t)$, we obtain
		\[
		\tp^{\cM}(\mathbf{x}',\mathbf{y}') = \tp^{\cM}(\mathbf{x}',\nabla \phi_t^{\cM}(\mathbf{y}')) = \tp^{\cQ}(\mathbf{x}_t,\nabla \phi_t^{\cM}(\mathbf{x}_t)) = \tp^{\cQ}(\mathbf{x}_t,\mathbf{y}).
		\]
		Hence, $\omega$ vanishes only when the input has the same type as $(\mathbf{x}_t,\mathbf{y})$.  Furthermore, because $\mathbb{S}_{2m,R}(\mathrm{T}_{\tr,\factor})$ is compact, we deduce that for every open neighborhood $\mathcal{O}$ of $\tp^{\cQ}(\mathbf{x}_t,\mathbf{y})$, there exists $\delta > 0$ such that $\omega^{\cM}(\mathbf{x}',\mathbf{y}') < \delta$ implies that $\tp^{\cM}(\mathbf{x}',\mathbf{y}') \in \mathcal{O}$.  In other words, for a net of types in $\mathbb{S}_{2m,R}(\mathrm{T}_{2m,R})$, convergence of $\omega$ to zero implies convergence of the types to $\tp^{\cQ}(\mathbf{X}_t,\mathbf{Y})$.
		
		Hence, it suffices to show that $\omega^{\bM_n}(\mathbf{X}^{(n)},\mathbf{Y}^{(n)}) \to 0$ in probability as $n \to \cU$.  By our choice of $\mathbf{X}^{(n)}$ as a maximizer in the definition of $\widehat{\psi}^{\bM_n}(\mathbf{Y}^{(n)})$, we have that
		\begin{align*}
			\omega^{\bM_n}(\mathbf{X}^{(n)}, \mathbf{Y}^{(n)}) &= \phi_t^{\bM_n}(\mathbf{X}^{(n)}) + \psi_t^{\bM_n}(\mathbf{Y}^{(n)}) - \re \ip{\mathbf{X}^{(n)},\mathbf{Y}^{(n)}}_{L^2(\cM)} + \eta^{\cM}(\mathbf{X}^{(n)}) + \tilde{\eta}(\mathbf{Y}^{(n)}) \\
			&= \psi_t^{\bM_n}(\mathbf{Y}^{(n)}) -\widehat{\psi}^{\bM_n}(\mathbf{Y}^{(n)}) + \tilde{\eta}^{\bM_n}(\mathbf{Y}^{(n)}).
		\end{align*}
		By our assumption on $\mathbf{Y}^{(n)}$, we have that
		\[
		\widehat{\psi}^{\bM_n}(\mathbf{Y}^{(n)}) \to \widehat{\psi}^{\cQ}(\mathbf{Y}^{(n)}), \qquad \tilde{\eta}^{\bM_n}(\mathbf{Y}^{(n)}) \to \tilde{\eta}^{\cQ}(\mathbf{Y}) = 0
		\]
		in probability as $n \to \cU$.  Note that $\widehat{\psi}^{\cQ}(\mathbf{X}) = \psi^{\cQ}(\mathbf{X})$ because on the one hand
		\[
		\sup_{\mathbf{X}' \in D_{\mathbf{r}}^{\cQ}} \left[ \re \ip{\mathbf{X}',\mathbf{Y}}_{L^2(\cQ)} -   \phi_t^{\cQ}(\mathbf{X}') - \eta^{\cQ}(\mathbf{X}') \right] \leq \sup_{\mathbf{X}' \in D_{\mathbf{r}}^{\cQ}} \left[ \ip{\mathbf{X}',\mathbf{Y}}_{L^2(\cQ)} -   \phi_t^{\cQ}(\mathbf{X}') \right] \leq \psi_t^{\cQ}(\mathbf{Y})
		\]
		while on the other hand
		\[
		\widehat{\psi}^{\cQ}(\mathbf{Y}) \geq \ip{\mathbf{X},\mathbf{Y}}_{L^2(\cQ)} -   \phi_t^{\cQ}(\mathbf{X}) - \eta^{\cQ}(\mathbf{X}) = \psi_t^{\cQ}(\mathbf{Y}) - 0.
		\]
		Therefore,
		\[
		\omega^{\bM_n}(\mathbf{X}^{(n)},\mathbf{Y}^{(n)}) = \psi_t^{\bM_n}(\mathbf{Y}^{(n)}) -\widehat{\psi}^{\bM_n}(\mathbf{Y}^{(n)}) + \tilde{\eta}^{\bM_n}(\mathbf{Y}^{(n)}) \to 0
		\]
		in probability as $n \to \cU$, as desired.
	\end{proof}
	
	\begin{lemma} \label{lem: displacement entropy bound}
		Let $(\mathbf{X},\mathbf{Y})$ be an optimal coupling in $\cQ$ of $\mu, \nu \in \mathbb{S}_{m,R}(\rT_{\cU})$.  Let $\mathbf{x}_t = (1-t)\mathbf{x} + t \mathbf{y}$.  Then
		\[
		\chi_{\full}^{\cU}(\mathbf{x}_t) \geq \chi_{\full}^{\cU}(\mathbf{y}) + 2m \log t.
		\]
	\end{lemma}
	
	\begin{proof}
		If $\chi_{\full}^{\cU}(\mathbf{y}) = -\infty$ or if $t = 0$, the claim is vacuously true.  Suppose that $\chi^{\cU}(\mathbf{y}) > -\infty$ and $t > 0$.  By Proposition \ref{prop: entropy of matrix model}, there exists a sequence of random matrix models $\mathbf{Y}^{(n)}$ such that $\norm{Y_j^{(n)}} \leq r_j$ and $\tp^{\bM_n}(\mathbf{Y}^{(n)})$ converges in probability to $\tp^{\cQ}(\mathbf{y})$ and
		\[
		\lim_{n \to \cU} h^{(n)}(\mathbf{Y}^{(n)}) = \chi_{\full}^{\cU}(\mathbf{y}).
		\]
		Fix $0 < t' < t < 1$.  By Lemma \ref{lem: lifting}, there exists some Borel measurable function $F^{(n)}$ such that $\mathbf{X}_{t'}^{(n)} := F^{(n)}(\mathbf{Y}^{(n)})$ satisfies that $(\mathbf{X}_{t'}^{(n)},\mathbf{Y}^{(n)})$ is an optimal coupling of two probability measures on $D_{\mathbf{r}}^{\bM_n}$ and $\tp^{\bM_n}(\mathbf{X}_{t'}^{(n)},\mathbf{Y}^{(n)}) \to \tp^{\cQ}(\mathbf{x}_{t'},\mathbf{y})$.  Let us write
		\[
		\mathbf{X}_t^{(n)} = \frac{1-t}{1-t'} \mathbf{X}_{t'}^{(n)} + \frac{t-t'}{1-t'} \mathbf{Y}^{(n)}.
		\]
		
		By the classical Monge-Kantorovich duality, there is a convex function $\phi_{t'}^{(n)}$ such that $\mathbf{Y}^{(n)} \in \underline{\nabla} \phi_{t'}^{(n)}$ almost surely.  Moreover, let
		\[
		\phi_t^{(n)} = \mathcal{L}\left(\frac{1-t}{1-t'} \mathcal{L}\phi_{t'}^{(n)} + \frac{t-t'}{1-t'} q_1^{(n)}\right),
		\]
		where $q^{(n)}(\mathbf{X}) = (1/2) \norm{\mathbf{X}}_{\tr_n}^2$.  Thus, by similar reasoning as in \S \ref{subsec: classical convex}, $\phi_t^{(n)}$ is convex and $(1-t')/(t-t')$-semiconcave, and in particular $\nabla \psi_t^{(n)}$ is $(1-t')/(t-t')$-Lipschitz.  Moreover, $\nabla \phi_t^{(n)}(\mathbf{X}_t^{(n)}) = \mathbf{Y}^{(n)}$ almost surely.
		
		Let $\mu_t^{(n)}$ be the distribution of $\mathbf{X}_t^{(n)}$.  Since $\mathbf{Y}$ is a Lipschitz function of $\mathbf{X}_t^{(n)}$, it follows that the distribution of $\mu_t^{(n)}$ is absolutely continuous (see e.g.\ \cite[Theorem 8.7]{Villani2008}) and hence has a density $\rho_t^{(n)}$.  By Rademacher's theorem, $\nabla \psi_t^{(n)}$ is differentiable almost everywhere.  Moreover, since $\mathbf{X}_{t'}^{(n)}$ and hence $\mathbf{X}_t^{(n)}$ are given as Borel functions of $\mathbf{Y}^{(n)}$, we see that $\nabla \phi_t^{(n)}$ is injective on the support of $\rho_t^{(n)}$.  By the measurable change-of-variables theorem, the densities $\rho_t^{(n)}$ for $\mathbf{X}_t^{(n)}$ and $\rho_1^{(n)}$ for $\mathbf{Y}^{(n)}$ satisfy
		\[
		\rho_t^{(n)}(\mathbf{x}) = \rho_1^{(n)}(\nabla \phi_t^{(n)}(\mathbf{x})) |\det D(\nabla \phi_t^{(n)})(\mathbf{x})| \text{ for } \mathbf{x}\in \supp(\rho_t^{(n)}).
		\]
		Hence, we have
		\begin{align*}
			h(\mathbf{X}_t^{(n)}) &= -\int \rho_t^{(n)}(\mathbf{x}) \log \rho_t^{(n)}(\mathbf{x})\,d\mathbf{x} \\
			&= \int_{\supp(\rho_t^{(n)})} [-\log \rho_1^{(n)} \circ \nabla \phi_t^{(n)}(\mathbf{x})] \rho_1^{(n)} \circ \nabla \phi_t^{(n)}(\mathbf{x}) |\det D(\nabla \phi_t^{(n)})(\mathbf{x})| \,d\mathbf{x} \\
			& \quad - \int_{\supp(\rho_t^{(n)})} \log |\det D(\nabla \phi_t^{(n)})(\mathbf{x})| \rho_1^{(n)} \circ \nabla \phi_t^{(n)}(\mathbf{x}) |\det D(\nabla \phi_t^{(n)})(\mathbf{x})|\,d\mathbf{x} \\
			&= \int [-\log \rho_1^{(n)} \circ \nabla \phi_t^{(n)}(\mathbf{x})] \rho_t^{(n)}(\mathbf{x})\,d\mathbf{x} \\
			&\quad - \int |\det D(\nabla \phi_t^{(n)})(\mathbf{x})| \rho_t(\mathbf{x})\,d\mathbf{x} \\
			&= \int [-\log \rho_1^{(n)}(\mathbf{y})] \rho_1^{(n)}(\mathbf{y})\,d\mathbf{y} - \int \log |\det D(\nabla \phi_t^{(n)})(\mathbf{x})| \rho_t(\mathbf{x})\,d\mathbf{x},
		\end{align*}
		and thus
		\[
		h(\mathbf{X}_t^{(n)}) = h(\mathbf{Y}^{(n)}) - \int \log |\det D(\nabla \phi_t^{(n)})(\mathbf{x})| \rho_t(\mathbf{x})\,d\mathbf{x}.
		\]
		Now almost surely $\norm{D(\nabla \phi_t^{(n)})} \leq (1-t')/(t-t')$, and hence the determinant (as a real linear transformation) is bounded by $[(1-t')/(t-t')]^{2mn^2}$.  Thus,
		\[
		h(\mathbf{X}_t^{(n)}) \geq h(\mathbf{Y}^{(n)}) - 2mn^2 \log \frac{1-t'}{t-t'}. 
		\]
		Hence,
		\[
		h^{(n)}(\mathbf{X}_t^{(n)}) \geq h^{(n)}(\mathbf{Y}^{(n)}) - 2m \log \frac{1-t'}{t-t'}.
		\]
		Then by Proposition \ref{prop: entropy of matrix model},
		\begin{align*}
			\chi_{\full}^{\cU}(\mathbf{x}_t) &\geq \lim_{n \to \cU} h^{(n)}(\mathbf{X}_t^{(n)}) \\
			&\geq \lim_{n \to \cU} h^{(n)}(\mathbf{Y}^{(n)}) - 2m \log \frac{1-t'}{t-t'} \\
			&= \chi_{\full}^{\cU}(\mathbf{Y}) - 2m \log \frac{1-t'}{t-t'}.
		\end{align*}
		Finally, letting $t' \to 0$, we obtain $\chi^{\cU}(\mathbf{x}_t) \geq \chi^{\cU}(\mathbf{y}) - 2m \log(1/t)$, which is the inequality we wanted to prove.
	\end{proof}
	
	Finally, we put the pieces together to conclude the proof of Theorem \ref{thm: entropy along geodesics}.
	
	\begin{proof}[Proof of Theorem \ref{thm: entropy along geodesics}]
		Let $(\mathbf{x},\mathbf{y})$ be an optimal coupling in $\cQ$ of $\mu, \nu \in \mathbb{S}_m(\rT_{\cU})$ (which always exists by Remark \ref{rem: optimal coupling exists}) and let $\mathbf{x}_t = (1-t) \mathbf{x} + t \mathbf{y}$ and $\mu_t = \tp^{\cQ}(\mathbf{x}_t)$.
		
		(1) By Proposition \ref{prop: monotonicity of metric entropy}, since $\mathbf{x}, \mathbf{y} \in \dcl^{\cQ}(\mathbf{x}_t)$, we have $\Ent_{\full}^{\cU}(\tp^{\cQ}(\mathbf{x})) \leq \Ent_{\full}^{\cU}(\tp^{\cQ}(\mathbf{x}_t))$ and $\Ent_{\full}^{\cU}(\tp^{\cQ}(\mathbf{y})) \leq \Ent_{\full}^{\cU}(\tp^{\cQ}(\mathbf{x}_t))$.
		
		(2) Let $\varphi_{s,t}$ and $\psi_{s,t}$ be as in Proposition \ref{prop: displacement duality}.  For $t > 0$, $\psi_{0,t}$ is $1/t$-semiconcave, and hence $\nabla \psi_{0,t}$ is a $1/t$-Lipschitz definable function by \cite[Corollary 5.7]{JekelTypeCoupling}.  Since $\varphi_{0,t}$ and $\psi_{0,t}$ witness Monge-Kantorovich duality for $\mathbf{x}$ and $\mathbf{x}_t$, we have $\mathbf{x} = \nabla \psi_{0,t}(\mathbf{x}_t)$.  By Lemma \ref{lem: MFED Lipschitz pushforward}, we have $\delta_{\full}^{\cU}(\mu) \leq \delta_{\full}^{\cU}(\mu_t)$.  Symmetrically, $\delta_{\full}^{\cU}(\nu) \leq \delta_{\full}^{\cU}(\mu_t)$.
		
		(3) By Lemma \ref{lem: displacement entropy bound}, we have $\chi_{\full}^{\cU}(\mu_t) \geq \chi_{\full}^{\cU}(\nu) + 2m \log t$.  By symmetry, namely by switching $\mathbf{x}$ and $\mathbf{y}$ and substituting $1-t$ instead of $t$, we have $\chi_{\full}^{\cU}(\mu_t) \geq \chi_{\full}^{\cU}(\mu) + 2m \log(1-t)$.
	\end{proof}
	
	\begin{remark}[Upper bounds in the setting of Theorem \ref{thm: entropy along geodesics}]
		We also have the easy upper bound
		\[
		\delta_{\full}^{\cU}(\mu_t) \leq \delta_{\full}^{\cU}(\mu) + \delta_{\full}^{\cU}(\nu).
		\]
		To prove this, first note that $\delta_{\full}^{\cU}(\tp^{\cU}(\mathbf{x},\mathbf{y})) \leq \delta_{\full}^{\cU}(\mu) + \delta_{\full}^{\cU}(\nu)$ since the Cartesian product of any two microstate spaces for $\mu$ and $\nu$ is a microstate space for $\tp^{\cU}(\mathbf{x},\mathbf{y})$; see \cite[Proposition 6.2]{VoiculescuFE2} for the analogous property for free entropy dimension.  Then since $\mathbf{x}_t$ is the image of $(\mathbf{x},\mathbf{y})$ under a Lipschitz definable predicate, $\delta_{\full}^{\cU}(\mu_t) \leq \delta_{\full}^{\cU}(\tp^{\cU}(\mathbf{x},\mathbf{y})) \leq \delta_{\full}^{\cU}(\mu) + \delta_{\full}^{\cU}(\nu)$.
		
		This upper bound is also sharp in certain cases:  Indeed, suppose that $(x,y)$ is a pair of self-adjoint operators whose type has free entropy dimension $2$ (for instance if $x$ and $y$ are the limiting type of independent GUE random matrices along the ultrafilter $\cU$).  Let $\mathbf{x} = (x,0)$ and $\mathbf{y} = (0,y)$.  It is easy to check that $\mathbf{x}$ and $\mathbf{y}$ are optimally coupled.  Also, $\mathbf{x}_t = ((1-t)x, ty)$ has full free entropy dimension for any $t \in (0,1)$.  Hence,
		\[
		2 = \delta_{\full}^{\cU}(\mathbf{x}_t) \leq \delta_{\full}^{\cU}(\mathbf{x}) + \delta_{\full}^{\cU}(\mathbf{y}) \leq 1 + 1,
		\]
		so that equality is achieved.
		
		The same example shows that it is impossible to have an upper bound for $\Ent_{\full}^{\cU}(\mu_t)$ in terms of $\Ent_{\full}^{\cU}(\mu)$ and $\Ent_{\full}^{\cU}(\nu)$ in general.  Indeed, 
		$\Ent_{\full}^{\cU}(\mu) = \Ent_{\full}^{\cU}(\nu) = 0$ since the corresponding von Neumann algebras are commutative.  However, $\Ent_{\full}^{\cU}(\mu_t) = \infty$ since $\delta_{\full}^{\cU}(\mu_t) > 1$.  It would be interesting to investigate whether an upper bound on $\Ent_{\full}^{\cU}(\mu_t)$ can be obtained under additional conditions on $\mu$ and $\nu$.
	\end{remark}
	
	\begin{proof}[Proof of Proposition \ref{prop: geodesic continuity}]
		Consider the same setup as in Theorem \ref{thm: entropy along geodesics}.  We will show that 
		\begin{equation} \label{eq: geodesic continuity goal}
			\chi_{\full}^{\cU}(\mu_t) - \chi_{\full}^{\cU}(\mu_s) \leq 2m \log \frac{t}{s}.
		\end{equation}
		Note that the other asserted inequality can be written as
		\[
		\chi_{\full}^{\cU}(\mu_s) - \chi_{\full}^{\cU}(\mu_t) \leq 2m \log \frac{1-s}{1-t},
		\]
		and hence it follows from \eqref{eq: geodesic continuity goal} by the switching $\mu$ and $\nu$ and substituting $1-t$ for $s$ and $1-s$ for $t$.
		
		Recall $0 \leq s < t \leq 1$.  Note that if $s = 0$, the right-hand side of \eqref{eq: geodesic continuity goal} is $+\infty$, so there is nothing to prove.  Moreover, in the case that $t = 1$, the claim follows from Theorem \ref{thm: entropy along geodesics} (3).  Therefore, assume that $0 < s < t < 1$.
		
		Let $R > \norm{\mathbf{x}_t}_\infty$.  By Proposition \ref{prop: entropy of matrix model}, there exist random matrix models $\mathbf{X}_t^{(n)}$ such that $\norm{\mathbf{X}_t^{(n)}}_\infty \leq R$ and $\lim_{n \to \cU} \tp^{\bM_n}(\mathbf{X}_t^{(n)}) = \mu_t$ uniformly and
		\[
		\lim_{n \to \cU} h^{(n)}(\mathbf{X}^{(n)}) = \chi_{\full}^{\cU}(\mu_t).
		\]
		Let $\varphi_{s,t}$ and $\psi_{s,t}$ be as in Proposition \ref{prop: displacement duality}.  Thus, $\psi_{s,t}$ is $s/t$-strongly convex and $(1-s)/(1-t)$-semiconcave.  We also have $(\nabla \psi_{s,t})_* \mu_t = \mu_s$.  Since $\nabla \psi_{s,t}$ is bounded on operator norm balls, we have that $\nabla \psi_{s,t}^{\bM_n}(\mathbf{X}_t^{(n)})$ is bounded in operator norm.  Because the pushforward by definable functions is continuous, the type of $\nabla \psi_{s,t}^{\bM_n}(\mathbf{X}_t^{(n)})$ converges uniformly to $\mu_s$ as $n \to \cU$.  Therefore,
		\[
		\chi_{\full}^{\cU}(\mu_s) \geq \lim_{n \to \cU} h^{(n)}(\nabla \psi_{s,t}^{\bM_n}(X_t^{(n)})).
		\]
		Since $\psi_{s,t}$ is strongly convex and semiconcave, $\nabla \psi_{s,t}$ is Lipschitz with a Lipschitz inverse, so we can apply measurable change of variables by $\nabla\psi_{s,t}$ to the entropy.  Because $\psi_{s,t}$ is $s/t$-uniformly convex, we have
		\[
		h^{(n)}(\nabla \psi_{s,t}^{\bM_n}(\mathbf{X}_t^{(n)})) \geq h^{(n)}(\mathbf{X}_t^{(n)}) + 2m \log \frac{s}{t}.
		\]
		Thus,
		\[
		\chi_{\full}^{\cU}(\mu_s) \geq \chi_{\full}^{\cU}(\mu_t) + 2m \log \frac{s}{t},
		\]
		which is equivalent to \eqref{eq: geodesic continuity goal}.
	\end{proof}
	
	\subsection{Topological properties of free entropy} \label{subsec: topological properties}
	
	Next, we turn our attention to establishing the topological properties of various free entropy quantities.
	
	\begin{proof}[Proof of Proposition \ref{prop: topological properties}] ~
		
		(1)	First, we must show $\Ent_{\full}^{\cU}$ is lower semi-continuous on $\mathbb{S}_{m,R}(\rT_{\cU})$ with respect to the Wasserstein distance.  Fix a type $\mu \in \mathbb{S}_{m,R}(\rT_{\cU})$.  We claim that for $\varepsilon > 0$ and $\delta \in (0,1)$, we have
		\begin{equation} \label{eq: covering inequality}
			\forall \nu \in \mathbb{S}_{m,R}(\rT_{\cU}), \quad d_{W,\full}(\nu,\mu) < \delta \implies \Ent_{\varepsilon+\delta}^{\cU}(\nu) \geq \Ent_{\varepsilon}^{\cU}(\mu).
		\end{equation}
		Let $d_{W,\full}(\mu,\nu) < \delta$.  Fix a weak-$*$ neighborhood $\cO$ of $\nu$.  By Urysohn's lemma, there exists a nonnegative continuous function on $\mathbb{S}_{m,R}(\rT_{\cU})$, i.e., a nonnegative definable predicate $\varphi$, such that $(\nu,\varphi) = 0$ and $(\nu',\varphi) = 1$ for $\nu' \in \mathbb{S}_{m,R}(\rT_{\cU}) \setminus \cO$.  Define a new definable predicate $\psi$ by
		\[
		\psi^{\cM}(\mathbf{x}) = \inf_{\mathbf{y} \in (D_R^{\cM})^m} \left[ \varphi^{\cM}(\mathbf{y}) + \norm{\mathbf{x} - \mathbf{y}}_{L^2(\cM)^m} \right].
		\]
		Let
		\[
		\cO' = \{\mu' \in \mathbb{S}_{m,R}(\rT_{\cU}): (\mu',\psi) < \delta \}.
		\]
		Note that $\mu \in \cO$; indeed, fixing an optimal coupling $(\mathbf{x},\mathbf{y})$ of $(\mu,\nu)$ in $\cM$, we have
		\[
		(\mu,\varphi) = \psi^{\cM}(\mathbf{x}) \leq \varphi^{\cM}(\mathbf{y}) + \norm{\mathbf{x} - \mathbf{y}}_{L^2(\cM)} < 0 + \delta.
		\]
		Next, observe that
		\begin{equation} \label{eq: microstate containment}
			\Gamma_R^{(n)}(\cO') \subseteq N_\delta(\Gamma_R^{(n)}(\cO)).
		\end{equation}
		Indeed, if $\mathbf{X}$ is a matrix tuple in $\Gamma_R^{(n)}(\cO')$, then since $\psi^{\bM_n}(\mathbf{X}) < \delta$, there exists some $\mathbf{Y} \in (D_R^{\bM_n})^m$ with $\varphi^{\bM_n}(\mathbf{Y}) + \norm{\mathbf{X} - \mathbf{Y}}_{\tr_n} < \delta$, which implies that $\norm{\mathbf{X} - \mathbf{Y}}_{\tr_n} < \delta$ as well as $\mathbf{Y} \in \Gamma_R^{(n)}(\cO)$ since $\varphi^{\bM_n}(\mathbf{Y}) < \delta < 1$.  Now \eqref{eq: microstate containment} implies that
		\[
		K_{\varepsilon+\delta}^{\orb}(\Gamma_R^{(n)}(\cO')) \leq K_{\varepsilon}^{\orb}(\Gamma_R^{(n)}(\cO)).
		\]
		Hence, taking the logarithm and dividing by $n^2$ and taking the limit as $n \to \cU$, we have
		\[
		\Ent_{\full,\varepsilon+\delta}^{\cU}(\mu) \leq \Ent_{\full,\varepsilon}^{\cU}(\cO') \leq \Ent_{\full,\varepsilon}^{\cU}(\cO).
		\]
		Since $\cO$ was arbitrary, we obtain $\Ent_{\full,\varepsilon+\delta}^{\cU}(\mu) \leq \Ent_{\full,\varepsilon}^{\cU}(\nu)$ as desired, and thus we have proved \eqref{eq: covering inequality}.
		
		Next, from \eqref{eq: covering inequality}, it follows that if $d_{W,\full}(\mu,\nu) < \delta$, then
		\[
		\Ent_{\full}^{\cU}(\nu) \geq \Ent_{\full,\varepsilon}^{\cU}(\nu) \geq \Ent_{\full,\varepsilon+\delta}^{\cU}(\mu).
		\]
		Hence,
		\[
		\liminf_{d_{W,\full}(\nu,\mu) \to 0} \Ent_{\full}^{\cU}(\nu) \geq \Ent_{\full,\varepsilon+\delta}^{\cU}(\mu).
		\]
		Next, taking the supremum over $\varepsilon$ and $\delta$ on the right-hand side, we obtain
		\[
		\liminf_{d_{W,\full}(\nu,\mu) \to 0} \Ent_{\full}^{\cU}(\nu) \geq \Ent_{\full}^{\cU}(\mu),
		\]
		which is the desired lower semi-continuity.
		
		(2) Note that
		\[
		\{\mu \in \mathbb{S}_{m,R}(\rT_{\cU}): \Ent_{\full}^{\cU}(\mu) = +\infty \} = \bigcap_{k \in \bN} \{\mu \in \mathbb{S}_{m,R}(\rT_{\cU}): \Ent_{\full}^{\cU}(\mu) > k \}.
		\]
		Each set on the right-hand side is open with respect to $d_{W,\full}$ on account of the lower semi-continuity of $\Ent_{\full}^{\cU}$ with respect to $d_{W,\full}$.  Therefore, the set on the left-hand side is $G_\delta$.
		
		To show Wasserstein density of this set, fix $\mu \in \mathbb{S}_{m,R}(\rT_{\cU})$.  Let $\nu \in \mathbb{S}_{m,R}(\rT_{\cU})$ be any type with $\Ent_{\full}^{\cU}(\nu) = +\infty$ (for instance, the limiting type of Gaussian random matrices scaled so that the norm is bounded by $R$).  Let $(\mathbf{x},\mathbf{y})$ be an optimal coupling of $\mu$ and $\nu$, and let $\mathbf{x}_t = (1-t) \mathbf{x} + t \mathbf{y}$.  By Theorem \ref{thm: entropy along geodesics} (1), $\Ent_{\full}^{\cU}(\mu_t) = +\infty$ and $d_{W,\full}(\mu_t,\mu) \to 0$ as $t \searrow 0$.
		
		(3) The weak-$*$ upper semi-continuity of $\chi_{\full}^{\cU}$ is shown in \cite[Lemma 3.6]{JekelModelEntropy}.
		
		(4) First, observe that
		\[
		\{\mu \in \mathbb{S}_{m,R}(\rT_{\cU}): \chi_{\full}^{\cU}(\mu) = -\infty \} = \bigcap_{k \in \bN} \{\mu \in \mathbb{S}_{m,R}(\rT_{\cU}): \chi_{\full}^{\cU}(\mu) < -k \}.
		\]
		Each set on the right-hand side is weak-$*$ open on account of the weak-$*$ upper semi-continuity of $\chi_{\full}^{\cU}$.  Consequently, it is also open with respect to Wasserstein distance because the Wasserstein topology is stronger than the weak-$*$ topology.  Therefore, the set of types with $\chi_{\full}^{\cU} = -\infty$ is a $G_\delta$ set with respect to both topologies.
		
		Next, we show density of this set with respect to Wasserstein distance.  Fix some full type $\mu$, and let $\mathbf{x} \in \cQ^m$ be some element with this type.  Fix a projection $p_k \in \cQ$ with trace $1/k$.  Let $\mathbf{x}_k = (x_{k,1},\dots,x_{k,m})$ where
		\[
		x_{k,j} = (1 - 1/k) (1 - p_k)x_j(1-p_k) + R p_k.
		\]
		Note that $p_k \in \mathrm{W}^*(\mathbf{x}_k)$ since $\lim_{\ell \to \infty} (x_{k,j}/R)^\ell \to p_k$ in strong operator topology.  Clearly, $p_k$ commutes with $\mathbf{x}_{k,j}$.  Thus, $\mathrm{W}^*(\mathbf{x}_k)$ has nontrivial center, so by \cite[Theorem 4.1]{VoiculescuFE3}, the plain free entropy $\chi(\mathbf{x}_k)$ is $-\infty$.  Since $\chi_{\full}^{\cU}(\mathbf{x}_k) \leq \chi^{\cU}(\mathbf{x}_k) \leq \chi(\mathbf{x}_k)$ by \cite[Corollary 4.4]{JekelModelEntropy}, we also have $\chi_{\full}^{\cU}(\mathbf{x}_k) = -\infty$, and $\mathbf{x}_k \to \mathbf{x}$ in $L^2(\cM)^m$ since $\norm{p_k}_{L^2(\cM)} \to 0$ as $k \to \infty$.  Thus, the set where $\chi_{\full}^{\cU} = -\infty$ is dense with respect to $d_{W,\full}$, hence also dense with respect to the weak-$*$ topology.
		
		(5) Wasserstein density of the set where $\chi_{\full}^{\cU} > -\infty$ follows by the same argument used in (2) for $\Ent_{\full}^{\cU} = +\infty$.  Fix $\mu \in \mathbb{S}_{m,R}(\rT_{\cU})$.  Let $\nu \in \mathbb{S}_{m,R}(\rT_{\cU})$ be any type with $\chi_{\full}^{\cU}(\nu) > -\infty$.  The type $\mu_t$ defined as before has finite $\chi_{\full}^{\cU}$ by Theorem \ref{thm: entropy along geodesics} (3), and also $\lim_{t \searrow 0} d_{W,\full}(\mu_t,\mu) = 0$.
	\end{proof}
	
	\subsection{A counterexample in the setting of laws} \label{subsec: counterexample}
	
	This section will show the impossibility of simultaneously approximating the Wasserstein distance and the microstates entropy of non-commutative laws by the same random matrix models.  More precisely, we will show the following.
	
	\begin{proposition} \label{prop: counterexample}
		There exist non-commutative laws $\mu$ and $\nu$ of self-adjoint $3$-tuples with finite free entropy $\chi^{\cU}$, such that there \textbf{do not} exist any random matrix tuples $\mathbf{X}^{(n)}$ and $\mathbf{Y}^{(n)}$ satisfying simultaneously:
		\begin{itemize}
			\item $\lim_{n \to \cU} h^{(n)}(\mathbf{X}^{(n)}) = \chi^{\cU}(\mu)$.
			\item $\lim_{n \to \cU} h^{(n)}(\mathbf{Y}^{(n)}) = \chi^{\cU}(\nu)$.
			\item $\lim_{n \to \cU} \norm{\mathbf{X}^{(n)} - \mathbf{Y}^{(n)}}_{L^2} = d_{W,\CEP}(\mu,\nu)$.
			\item $\norm{X_j^{(n)}} \leq R$ and $\norm{Y_j^{(n)}} \leq R$ for some constant $R$.
		\end{itemize}
	\end{proposition}
	
	In fact, we actually can obtain a contradiction without even including the first condition on the entropy of $\mathbf{X}^{(n)}$, which, as noted in Remark \ref{rem: qf failure} below, also shows that Lemma \ref{lem: lifting} is false for quantifier-free types, or non-commutative laws. The last condition that the random matrices are uniformly bounded in operator norm is assumed mostly for technical convenience; the statement can be extended to random matrix models satisfying reasonable tail bounds by arguing as in \cite[Proposition B.7]{ST2022} and \cite[Theorem 4.8]{JekelPi2024}, a side quest which we leave to the reader.  We also remark that throughout this section, $\chi^{\cU}$ refers to the free entropy defined for \emph{self-adjoint} tuples as originally formulated in \cite{VoiculescuFE2}, even though in the rest of this paper, we have used the version for non-self-adjoint operators.
	
	The laws $\mu$ and $\nu$ will be the distributions of tuples $\mathbf{X}$ and $\mathbf{Y}$ defined as follows.  Let $\varepsilon \in (0,1)$.  Let $S_1$ and $S_2$ be freely independent semicirculars, and let $S_3$ be a standard semicircular operator that is tensor independent of $S_1$ and $S_2$.  In other words,
	\[
	\mathrm{W}^*(S_1,S_2,S_3) \cong (\mathrm{W}^*(S_1) * \mathrm{W}^*(S_2)) \overline{\otimes} \mathrm{W}^*(S_3).
	\]
	Next, let $S_1'$, $S_2'$, $S_3'$ be standard semicirculars freely independent of each other and of $\mathrm{W}^*(S_1,S_2,S_3)$, and let $\cM$ be the tracial von Neumann algebra generated by $S_1$, $S_2$, $S_3$, $S_1'$, $S_2'$, $S_3'$.  Then let
	\[
	(X_1,X_2,X_3) = (1-\varepsilon)^{1/2} (S_1,S_2,S_3) + \varepsilon^{1/2} (S_1',S_2',S_3').
	\]
	Note that each $X_j$ is a standard semicircular operator.  Moreover, $X_1$ and $X_2$ are freely independent.  However, $X_1$ almost commutes with $X_3$.  Namely,
	\[
	[X_1,X_3] = 0 + (1-\varepsilon)^{1/2} \varepsilon^{1/2} [S_1,S_3'] + (1-\varepsilon)^{1/2} \varepsilon^{1/2} [S_1',S_3] + \varepsilon [S_1',S_3'],
	\]
	and hence since each semicircular has operator norm $2$, we have
	\[
	\norm{[X_1,X_3]} \leq 24 \varepsilon^{1/2},
	\]
	and similarly for $X_2$ instead of $X_1$.  We also define
	\[
	(Y_1,Y_2,Y_3) = (S_1,S_2,\varepsilon S_3').
	\]
	Let $\mu$ and $\nu$ be the non-commutative laws of $\mathbf{X}$ and $\mathbf{Y}$ respectively.
	
	\begin{claim}
		The Biane-Voiculescu-Wasserstein distance of $\mu$ and $\nu$ satisfies
		\[
		1 - \varepsilon \leq d_{W,\CEP}(\mu,\nu) \leq 1 - \varepsilon^{3/2}.
		\]
	\end{claim}
	
	\begin{proof}
		For the upper bound, note
		\begin{align*}
			d_{W,\operatorname{qf}}(\mu,\nu) &\leq \norm{\mathbf{X} - \mathbf{Y}}_{L^2(\cM)^3} \\
			&= \left( \norm{S_1 - S_1}_{L^2(\cM)^3}^2 + \norm{S_2 - S_2}_{L^2(\cM)^3}^2 + \norm{(1-\varepsilon)^{1/2}S_3 + (\varepsilon^{1/2} - \varepsilon)S_3' - \varepsilon^{1/2}S_3'}_{l^2(\cM)^3}^2 \right)^{1/2} \\
			&= [(1 - \varepsilon) + \varepsilon(1 - \varepsilon^{1/2})^2]^{1/2} \\
			&= [1 - 2 \varepsilon^{3/2} + \varepsilon^2]^{1/2} \\
			&\leq (1 - 2 \varepsilon^{3/2})^{1/2} \\
			&\leq 1 - \frac{1}{2} 2 \varepsilon^{3/2},
		\end{align*}
		where the last inequality follows from concavity of the square root.  For the lower bound, note that $d_{W,\operatorname{qf}}(\mu,\nu)$ is greater than or equal to the Wasserstein distance between the distributions of $X_3$ and $Y_3$, which is lower bounded by $| \norm{X_3}_{L^2(\cM)} - \norm{Y_3}_{L^2(\cM)} | = 1 - \varepsilon$.
	\end{proof}
	
	\begin{claim}
		We have $\chi^{\cU}(\nu) = (3/2) \log 2\pi e + (1/2) \log \varepsilon$.
	\end{claim}
	
	\begin{proof}
		Note that $\chi^{\cU}(Y_1) = \chi^{\cU}(Y_2) = \chi^{\cU}(\varepsilon^{-1/2}Y_3) = (1/2) \log (2\pi e)$ since these are standard semicirculars.  By change of variables $\chi^{\cU}(Y_3) = \chi^{\cU}(\varepsilon^{-1/2} Y_3) + (1/2) \log \varepsilon$.  By free independence, $\chi^{\cU}(Y_1,Y_2,Y_3) = \chi^{\cU}(Y_1) + \chi^{\cU}(Y_2) + \chi^{\cU}(Y_3)$.
	\end{proof}
	
	\begin{claim}
		We have $\chi^{\cU}(\mu) > -\infty$ and more specifically $\chi^{\cU}(\mu) \geq (3/2) \log 2\pi e + (3/2) \log \varepsilon$.
	\end{claim}
	
	\begin{proof}
		One can define random matrix models for $\mathbf{X}$ as follows.  Since $\mathrm{W}^*(S_1,S_2,S_3)$ is Connes-embeddable, there are some deterministic matrices $(S_1^{(n)},S_2^{(n)},S_3^{(n)})$ which converge in non-commutative law to $S_1$, $S_2$, $S_3$.  Then let $\hat{S}_1^{(n)}$, $\hat{S}_2^{(n)}$, $\hat{S}_3^{(n)}$ be independent standard GUE matrices truncated to the operator norm ball of radius $4$.  By Voiculescu's asymptotic freeness theorem, $(S_1^{(n)},S_2^{(n)},S_3^{(n)},\hat{S}_1^{(n)}$, $\hat{S}_2^{(n)}$, $\hat{S}_3^{(n)})$ converge in non-commutative law to $(S_1,S_2,S_3,S_1',S_2',S_3')$.  In particular, letting
		\[
		X_j^{(n)} = (1 - \varepsilon)^{1/2} S_j^{(n)} + \varepsilon^{1/2} \hat{S}_j^{(n)},
		\]
		we see that $\mathbf{X}^{(n)}$ converges in non-commutative law $\mathbf{X}$.  Also,
		\[
		\lim_{n \to \cU} h^{(n)}(\mathbf{X}^{(n)}) \geq \lim_{n \to \cU} h^{(n)}(\varepsilon^{1/2} \hat{\mathbf{S}}^{(n)}) = (3/2) \log(2 \pi e) + (3/2) \log \varepsilon.
		\]
		By Proposition \ref{prop: entropy of matrix model}, we have
		\[
		\chi^{\cU}(\mathbf{X}) \geq \lim_{n \to \cU} h^{(n)}(\mathbf{X}^{(n)}) = (3/2) \log(2 \pi e) + (3/2) \log \varepsilon.  \qedhere
		\]
	\end{proof}
	
	\begin{claim} \label{claim: counterexample}
		Let $\mathbf{X}^{(n)}$ and $\mathbf{Y}^{(n)}$ be random matrix models on the same probability space such that
		\[
		\norm{X_j^{(n)}} \leq R,   \qquad \norm{Y_j^{(n)}} \leq R,
		\]
		such that the non-commutative laws of $\mathbf{X}^{(n)}$ and $\mathbf{Y}^{(n)}$ converge in probability to those of $\mathbf{X}$ and $\mathbf{Y}$ as $n \to \cU$.  Let
		\[
		a = \chi^{\cU}(\mathbf{Y}) - \lim_{n \to \cU} h^{(n)}(\mathbf{Y}^{(n)})
		\]
		and
		\[
		b = \lim_{n \to \cU} \norm{\mathbf{X}^{(n)} - \mathbf{Y}^{(n)}}_{L^2(\Omega,\bM_n)}^2 - d_{W,CEP}(\mu,\nu)^2.
		\]
		Then
		\begin{equation} \label{eq: uncertainty principle}
			e^{a/2} \left[ (25 + 6R) \varepsilon^{1/2} + 2R b^{1/2} \right] \geq \varepsilon^{1/4},
		\end{equation}
		and in particular, provided that $\varepsilon$ is sufficiently small, $a$ and $b$ cannot both be zero.
	\end{claim}
	
	\begin{proof}
		First, we want to give an upper bound on $h^{(n)}(Y_1^{(n)},Y_2^{(n)})$.  Assume without loss of generality that $Y_1^{(n)}$, $Y_2^{(n)}$ has finite entropy. We want to bound this entropy through conditioning.  Recall that for a random variable $(Z,W)$ on a product space with joint density $\rho_{Z,W}$, if $\rho_W$ is the marginal density of $W$ and $\rho_{Z,W}(z,w) = \rho_{Z \mid W}(z,w) \rho_W(w)$, then
		\[
		h(Z \mid W) = \int \left( -\int \rho_{Z \mid W} \log \rho_{Z \mid W}(z,w)\,dz \right) \rho_W(w) \,dw.
		\]
		We also have
		\[
		h(Z, W) = h(Z \mid W) + h(W).
		\]
		We want to apply this with $Z$ being $(Y_1^{(n)},Y_2^{(n)})$ and $W$ being a perturbation of $X_3^{(n)}$.  Let $S_4^{(n)}$ be a GUE matrix independent of $\mathbf{X}^{(n)}$ and $\mathbf{Y}^{(n)}$.  Let $h^{(n)}(\cdot \mid \cdot)$ denote the conditional entropy with the same normalizations as we used for $h^{(n)}$, thus for instance
		\[
		h^{(n)}(X_3^{(n)} + \varepsilon^{1/2} S_4^{(n)} \mid Y_1^{(n)}, Y_2^{(n)}) = \frac{1}{n^2} h(X_3^{(n)} + \varepsilon^{1/2} S_4^{(n)} \mid Y_1^{(n)}, Y_2^{(n)}) + \log n,
		\]
		where the additive constant is $\log n$ times the number of self-adjoint matrix variables in the argument to the left of the $\mid$.  Then we have
		\begin{align}
			h^{(n)}(X_3^{(n)} + \varepsilon^{1/2} S_4^{(n)} \mid Y_1^{(n)}, Y_2^{(n)}) &\geq h^{(n)}(\varepsilon^{1/2} S_4^{(n)} \mid Y_1^{(n)}, Y_2^{(n)}) \label{eq: lower estimate for h of Wn} \\
			&= h^{(n)}(\varepsilon^{1/2} S_4^{(n)}) \nonumber \\
			&= (1/2) \log(2 \pi e) + (1/2) \log \varepsilon, \nonumber
		\end{align}
		where the first step follows because $S_4^{(n)}$ is independent of $X_3^{(n)}$ conditioned on $Y_1^{(n)}$ and $Y_2^{(n)}$, and the second step follows because $S_4^{(n)}$ is independent of $Y_1^{(n)}$ and $Y_2^{(n)}$.  In particular, $h^{(n)}(Y_1^{(n)},Y_2^{(n)}, X_3^{(n)} + \varepsilon^{1/2} S_4^{(n)})$ is finite.  Write $W^{(n)} = X_3^{(n)} + \varepsilon^{1/2} S_4^{(n)}$.
		
		Our goal is to obtain an upper bound on $h^{(n)}(Y_1^{(n)},Y_2^{(n)} \mid W^{(n)})$ using the fact that $Y_1^{(n)}$ and $Y_2^{(n)}$ almost commute with $W^{(n)}$.  First note that because of convergence in law,
		\[
		\lim_{n \to \cU} \norm{[X_1^{(n)},X_3^{(n)}]}_{L^2(\Omega,\bM_n)} = \norm{[X_1,X_3]}_{L^2(\cM)} \leq 24 \varepsilon^{1/2},
		\]
		and so for $\cU$-many $n$, we can assume this is bounded by $25 \varepsilon^{1/2}$.  Because $\norm{X_3^{(n)}} \leq R$, the triangle inequality and non-commutative H{\"o}lder inequality yield
		\[
		\norm{[Y_1^{(n)} - X_1^{(n)}, X_3^{(n)}]}_{L^2(\Omega,\mathbb{M}_n)} \leq 2R \norm{Y_1^{(n)} - X_1^{(n)}}_{L^2(\Omega,\mathbb{M}_n)},
		\]
		and hence by the triangle inequality, for $\cU$-many $n$, we have
		\[
		\norm{[Y_1^{(n)},X_3^{(n)}]}_{L^2(\Omega,\bM_n)} \leq 25 \varepsilon^{1/2} +  \norm{Y_1^{(n)} - X_1^{(n)}}_{L^2(\Omega,\mathbb{M}_n)}.
		\]
		Moreover, because $\norm{Y_1^{(n)}} \leq R$, we have
		\[
		\norm{[Y_1^{(n)}, \varepsilon^{1/2}S_4^{(n)}]}_{L^2(\Omega,\mathbb{M}_n)} \leq 2R \varepsilon^{1/2} \norm{S_4^{(n)}}_{L^2(\Omega,\mathbb{M}_n)} = 2R \varepsilon^{1/2}.
		\]
		By the triangle equality, for $\cU$-many $n$,
		\[
		\norm{[Y_1^{(n)},X_3^{(n)}+ \varepsilon^{1/2} S_4^{(n)}]}_{L^2(\Omega,\bM_n)} \leq (25 + 2R) \varepsilon^{1/2} + 2R \norm{Y_1^{(n)} - X_1^{(n)}}_{L^2}.
		\]
		Similarly,
		\[
		\norm{[Y_2^{(n)},X_3^{(n)}+ \varepsilon^{1/2} S_4^{(n)}]}_{L^2(\Omega,\bM_n)} \leq (25 + 2R) \varepsilon^{1/2} + 2R \norm{Y_2^{(n)} - X_2^{(n)}}_{L^2}.
		\]
		Recall we set $W^{(n)} = X_3^{(n)}+ \varepsilon^{1/2} S_4^{(n)}$, and let $T^{(n)} = W^{(n)} \otimes 1 - 1 \otimes W^{(n)}$ denote the operator in $\bM_n \otimes \bM_n^{\op}$ which acts on $\cM_n$ by left-right multiplication.  Fix $\eta > 0$, and for $j = 1$, $2$, let
		\[
		V_j^{(n)} = (\eta^2 + |T^{(n)}|^2)^{1/2} Y_j^{(n)}.
		\]
		Note
		\[
		\norm{V_j^{(n)}}_{\tr_n}^2 = \norm{T^{(n)} Y_j^{(n)}}_{\tr_n}^2 + \eta^2 \norm{Y_j^{(n)}}_{\tr_n}^2 \leq \norm{[Y_j^{(n)},W^{(n)}]}_{\tr_n}^2 + R^2 \eta^2,
		\]
		hence
		\[
		\norm{V_j^{(n)}}_{L^2(\Omega,\bM_n)} \leq (25 + 2R) \varepsilon^{1/2} + 2R \norm{X_j^{(n)} - Y_j^{(n)}}_{L^2(\Omega,\bM_n)} + R \eta.
		\]
		Now $(\eta^2 + |T^{(n)}|^2)^{1/2}$ is an invertible linear tranformation which only depends on $W^{(n)}$.  Thus, we can perform a change of variables for the entropy conditioned on $W^{(n)}$ and obtain
		\begin{align*}
			h^{(n)}(Y_j^{(n)} \mid W^{(n)}) &= h^{(n)}(V_j^{(n)} \mid W^{(n)}) + \frac{1}{n^2} \mathbb{E} \log \det (\eta^2 + |T^{(n)}|^2)^{-1/2} \\
			&\leq (1/2) \log(2 \pi e) + \frac{1}{2} E \log E[\norm{V_j^{(n)}}_2^2 \mid W^{(n)}] - \tr_n \otimes \tr_n[\log(\eta^2 + |T^{(n)}|^2)^{1/2}],
		\end{align*}
		where in the second inequality we applied the standard upper bound for entropy in terms of variance to $h^{(n)}(V_j^{(n)} \mid W^{(n)})$.  By Jensen's inequality,
		\begin{align*}
			\frac{1}{2} E \log E[\norm{V_j^{(n)}}_{\tr_n}^2 \mid W^{(n)}] &\leq \frac{1}{2} \log E \norm{V_j^{(n)}}_{\tr_n}^2 \\
			&= \frac{1}{2} \log \norm{V_j^{(n)}}_{L^2(\Omega,\bM_n)}^2 \\
			&\leq \log [(25 + 2R) \varepsilon^{1/2} + 2R \norm{X_j^{(n)} - Y_j^{(n)}}_{L^2(\Omega,\bM_n)} + R \eta].
		\end{align*}
		Meanwhile, $\tr_n \otimes \tr_n[\log(\eta^2 + |T^{(n)}|^2)]$ can be approximated in the large-$n$ limit using the convergence of the spectral distribution of $W^{(n)}$ to a semicircular of variance $1 + \varepsilon$.  Indeed, $W^{(n)} \otimes 1$ and $1 \otimes W^{(n)}$ are in tensor position, and hence for any two variable smooth function $f$, we have
		\[
		\mathbb{E} \tr_n \otimes \tr_n[f(W^{(n)} \otimes 1, 1 \otimes W^{(n)})] \to \int_{\bR \otimes \bR} f(s,t)\,d\sigma_{1+\varepsilon}(s)\,d\sigma_{1 + \varepsilon}(t),
		\]
		where $\sigma_{1+\varepsilon}$ is the semicircular density of variance $1 + \varepsilon$.  In particular,
		\begin{align*}
			\lim_{n \to \cU} \mathbb{E} \tr_n \otimes \tr_n[\log(\eta^2 + |T^{(n)}|^2)^{1/2}] &= \int_{\bR \times \bR} \log (\eta^2 + |s - t|^2)^{1/2} \,d\sigma_{1 + \varepsilon}(s) \,d\sigma_{1 + \varepsilon}(t) \\
			&\geq \int_{\bR \times \bR} \log |s - t| \,d\sigma_{1 + \varepsilon}(s) \,d\sigma_{1 + \varepsilon}(t),
		\end{align*}
		which by \cite{VoiculescuFE2} is exactly the free entropy $\chi$ of this semicircular distribution minus the constant term $(3/4) + (1/2) \log(2 \pi)$ by \cite[Proposition 4.5]{VoiculescuFE2}.  By change of variables \cite[Proposition 3.5]{VoiculescuFE2}, $\chi(W^{(n)})$ is $(1/2) \log(1 + \varepsilon)$ plus the entropy of a standard semicircular, which is $(1/2) \log(2 \pi e)$.  Hence, $\int_{\bR \times \bR} \log |s - t| \,d\sigma_{1 + \varepsilon}(s) \,d\sigma_{1 + \varepsilon}(t)$ evaluates to $(1/4) + (1/2) \log(1 + \varepsilon)$.  Therefore,
		\begin{multline*}
			\lim_{n \to \cU} h^{(n)}(Y_j^{(n)} \mid W^{(n)}) \\
			\leq (1/2) \log(2 \pi e) + \log [(25 + 2R) \varepsilon^{1/2} + 2R \norm{X_j^{(n)} - Y_j^{(n)}}_{L^2(\Omega,\bM_n)} + R \eta] - (1/2) \log(1 + \varepsilon) - 1/4,
		\end{multline*}
		and since $\eta$ was arbitrary,
		\begin{multline*}
			\lim_{n \to \cU} h^{(n)}(Y_j^{(n)} \mid W^{(n)}) \\
			\leq (1/2) \log(2 \pi e) + \log [(25 + 2R) \varepsilon^{1/2} + 2R \norm{X_j^{(n)} - Y_j^{(n)}}_{L^2(\Omega,\bM_n)}] - (1/2) \log(1 + \varepsilon) - 1/4,
		\end{multline*}
		Then we note that
		\[
		h^{(n)}(Y_1^{(n)}, Y_2^{(n)} \mid W^{(n)}) \leq h^{(n)}(Y_1^{(n)} \mid W^{(n)}) + h^{(n)}(Y_2^{(n)} \mid W^{(n)}).
		\]
		Therefore, using conditioning for entropy and recalling the estimate for $h^{(n)}(W^{(n)} \mid Y_1^{(n)}, Y_2^{(n)})$ in \eqref{eq: lower estimate for h of Wn}, we have
		\begin{align*}
			h^{(n)}(Y_1^{(n)}, Y_2^{(n)}) &= h^{(n)}(Y_1^{(n)},Y_2^{(n)},W^{(n)}) - h(W^{(n)} \mid Y_1^{(n)},Y_2^{(n)}) \\
			&\leq h^{(n)}(Y_1^{(n)},Y_2^{(n)},W^{(n)}) - (1/2) \log 2 \pi e - (1/2) \log \varepsilon \\
			&= h^{(n)}(Y_1^{(n)},Y_2^{(n)} \mid W^{(n)}) + h^{(n)}(W^{(n)}) - (1/2) \log(2 \pi e)  - (1/2) \log \varepsilon.
		\end{align*}
		Then we use the fact that $h^{(n)}(W^{(n)})$ is bounded in the limit by $(1/2) \log(2 \pi e) + (1/2) \log(1 + \varepsilon)$ since $W^{(n)}$ approximates a semicircular of variance $1 + \varepsilon$.  Together with our previous estimates on $h^{(n)}(Y_1^{(n)},Y_2^{(n)} \mid W^{(n)})$, we obtain
		\begin{align*}
			\lim_{n \to \cU} h^{(n)}(Y_1^{(n)},Y_2^{(n)}) &\leq \log(2 \pi e) + \sum_{j=1}^2 \log \left[(25 + 2R) \varepsilon^{1/2} + 2R \lim_{n \to \cU} \norm{X_j^{(n)} - Y_j^{(n)}}_{L^2(\Omega,\bM_n)}\right] \\
			& \quad - \log(1 + \varepsilon) -1/2 + (1/2) \log(1 + \varepsilon) - (1/2) \log \varepsilon \\
			&\leq \log(2 \pi e) + \sum_{j=1}^2 \log \left[(25 + 2R) \varepsilon^{1/2} + 2R \lim_{n \to \cU} \norm{X_j^{(n)} - Y_j^{(n)}}_{L^2(\Omega,\bM_n)} \right] \\
			&\quad - (1/2) \log \varepsilon.
		\end{align*}
		By subadditivity,
		\begin{align*}
			\lim_{n \to \cU} h^{(n)}(Y_1^{(n)}, Y_2^{(n)}, Y_3^{(n)}) &\leq \lim_{n \to \cU} h^{(n)}(Y_1^{(n)},Y_2^{(n)}) + \lim_{n \to \cU} h^{(n)}(Y_3^{(n)}) \\
			&= \lim_{n \to \cU} h^{(n)}(Y_1^{(n)},Y_2^{(n)}) + \chi^{\cU}(Y_3).
		\end{align*}
		Also, $\chi^{\cU}(Y_1,Y_2,Y_3) = \chi^{\cU}(Y_1,Y_2) + \chi^{\cU}(Y_3)$ and $\chi^{\cU}(Y_1,Y_2) = \log(2 \pi e)$, and hence
		\begin{align*}
			a &= \chi^{\cU}(Y_1,Y_2,Y_3) - \lim_{n \to \cU} h^{(n)}(Y_1^{(n)},Y_2^{(n)},Y_3^{(n)}) \\
			&\geq \chi^{\cU}(Y_1,Y_2) - \lim_{n \to \cU} h^{(n)}(Y_1^{(n)},Y_2^{(n)}) \\
			&\geq - \sum_{j=1}^2 \log \left[(25 + 2R) \varepsilon^{1/2} + 2R \lim_{n \to \cU} \norm{X_j^{(n)} - Y_j^{(n)}}_{L^2(\Omega,\bM_n)} \right] + (1/2) \log \varepsilon.
		\end{align*}
		To relate this equation with the quantity $b$ in the statement, observe that
		\begin{align*}
			\norm{\mathbf{X} - \mathbf{Y}}_{L^2(\Omega,\bM_n^3)}^2 - d_{W,\full}(\mu,\nu)^2 &=
			\sum_{j=1}^3 \norm{X_j^{(n)} - Y_j^{(n)}}_{L^2}^2 - d_{W,\full}(\mu,\nu)^2 \\
			&\geq \sum_{j=1}^2 \norm{X_j^{(n)} - Y_j^{(n)}}_{L^2}^2 + (1 - \varepsilon)^2 - (1 - \varepsilon^{3/2})^2 \\
			&\geq \sum_{j=1}^2 \norm{X_j^{(n)} - Y_j^{(n)}}_{L^2}^2 - 2 \varepsilon + \varepsilon^2 + 2 \varepsilon^{3/2} - \varepsilon^3 \\
			&\geq \sum_{j=1}^2 \norm{X_j^{(n)} - Y_j^{(n)}}_{L^2}^2 - 2 \varepsilon.
		\end{align*}
		Hence, for $j = 1$, $2$,
		\[
		\lim_{n \to \cU} \norm{X_j^{(n)} - Y_j^{(n)}}_{L^2}^2 \leq b + 2 \varepsilon \leq (b^{1/2} + 2 \varepsilon^{1/2})^2.
		\]
		Therefore,
		\[
		a \geq -2 \log \left[ (25 + 6R) \varepsilon^{1/2} + 2R b^{1/2} \right] + (1/2) \log \varepsilon.
		\]
		This rearranges to yield \eqref{eq: uncertainty principle}.
	\end{proof}
	
	\begin{remark} \label{rem: qf failure}
		Claim \ref{claim: counterexample} shows that Lemma \ref{lem: lifting} fails for quantifier-free types.  Of course, we restrict our attention to those that arise in Connes-embeddable von Neumann algebras, and use the Connes-embeddable Wasserstein distance.  Now let $(\mathbf{X},\mathbf{Y})$ be a Connes-embeddable optimal coupling of $(\mu,\nu)$, and let $\mathbf{X}_t = (1-t) \mathbf{X} + t \mathbf{Y}$.  Let $\mathbf{Y}^{(n)}$ be a random matrix model chosen to realize the $\chi^{\cU}(\nu)$ asymptotically.  Suppose for contradiction that there exists a random matrix model $\mathbf{X}_t^{(n)}$ such that $(\mathbf{X}_t^{(n)},\mathbf{Y}^{(n)})$ converges in law to $(\mathbf{X}_t,\mathbf{Y})$, then of course $\mathbf{X}^{(n)} = (1-t)^{-1} (\mathbf{X}_t^{(n)} - t \mathbf{Y}^{(n)})$ gives a compatible random matrix model for $\mathbf{X}^{(n)}$, and Claim \ref{claim: counterexample} with $a = 0$, then gives a lower bound on $b$ and hence on $\lim_{n \to \infty} \norm{\mathbf{X}^{(n)} - \mathbf{Y}^{(n)}}_{L^2(\Omega,\bM_n^3)}$.  This yields a lower bound on $\norm{\mathbf{X}_t^{(n)} - \mathbf{X}^{(n)}}_{L^2(\Omega,\bM_n^3)}$, showing that it cannot realize the Connes-embeddable Wasserstein distance.
	\end{remark}

	\section{Gibbs types for strongly convex potentials} \label{sec: Gibbs types}
	
	\subsection{Existence via matrix models} \label{subsec: Gibbs existence}
	
	\begin{definition}
		Let $\varphi$ be a definable predicate in $m$ free variables for tracial von Neumann algebras.  Fix an ultrafilter $\cU$ and let $\cQ = \prod_{n \to \cU} \mathbb{M}_n$ and $\mathrm{T}_{\cU} = \Th(\cQ)$.  We say that $\mu \in \mathbb{S}_m(\mathrm{T}_{\cU})$ is a \emph{Gibbs type} for $\varphi$ (with respect to $\cU$) if
		\[
		\chi_{\full}^{\cU}(\mu) - (\mu,\varphi) = \sup_{\nu \in \mathbb{S}_m(\rT_{\cU})} \left[ \chi_{\full}^{\cU}(\nu) - (\nu,\varphi) \right].
		\]
	\end{definition}
	
	Several remarks are in order:
	\begin{enumerate}
		\item For general $\varphi$, the supremum may be infinite, in which case there cannot be a Gibbs type.  However, recall that
		\[
		\chi_{\full}^{\cU}(\nu) \leq C + \log \sum_{j=1}^m \nu(x_j^*x_j),
		\]
		and therefore if $\varphi^{\cM}(\mathbf{X}) \geq C' - \log \norm{\mathbf{X}}_{L^2(\cM)^m}^2$ for some constant $C'$, then the supremum will be finite.
		\item For each $R > 0$, the supremum over $\mathbb{S}_{m,R}(\rT_{\cU})$ will be achieved because $\mathbb{S}_{m,R}(\rT_{\cU})$ is compact in the weak-$*$ topology and $\chi_{\full}^{\cU}$ is upper semi-continuous, while of course $\nu \mapsto (\nu,\varphi)$ is continuous.  However, this is not enough to determine whether a supremum is achieved over all of $\mathbb{S}_m(\rT_{\cU})$ (or over a suitable completion of this space).
	\end{enumerate}
	
	Our first goal is to show existence and uniqueness of Gibbs types for a given ultrafilter $\cU$ and definable predicate $\varphi$ that is $c$-strongly convex.  The proof is based on studying the associated random multi-matrix models.  Let $\mu^{(n)}$ be the probability measure on $(\bM_n)^m$ given by
	\begin{equation} \label{eq: RMME}
		d\mu^{(n)}(\mathbf{X}) = \frac{1}{Z^{(n)}} e^{-n^2 \varphi^{\bM_n}(\mathbf{X})}\,d\mathbf{X} \text{ where } Z^{(n)} = \int_{\bM_n^m} e^{-n^2 \varphi^{\bM_n}(\mathbf{X})}\,d\mathbf{X}.
	\end{equation}
	We will study this model using concentration of measure techniques which are now standard in random matrix theory.  We first recall the following results on concentration inequalities.
	
	\begin{theorem}[Various authors]
		Let $H$ be a finite-dimensional real inner-product space and let $m$ be its canonical Lebesgue measure (obtained by identification with $\bR^m$ using an orthonormal basis).  Let $c > 0$ and $\mu \in \mathcal{P}(V)$.  Consider the following inequalities:
		\begin{enumerate}[(1)]
			\item \label{item log concave} \textbf{$c$-strong log-concavity}:
			\[
			d\mu(x) = e^{-V(x)} \,dx \text{ where } V \text{ is }c \text{-strongly convex.}
			\]
			\item \label{item log Sobolev} \textbf{Log-Sobolev inequality with constant $c$ \cite{Gross1975}:} Whenever $\nu \in \mathcal{P}(H)$ is given by density $\rho$ with respect to $\mu$, that is, $d\nu(x) = \rho(x)\,d\mu(x)$, then
			\[
			\int_H \rho \log \rho\,d\mu \leq \frac{2}{c} \int_H |\nabla \log \rho|^2\,d\mu
			\]
			\item \label{item Talagrand} \textbf{Talagrand inequality with constant $c$ \cite{Talagrand1996}:}  If $\nu \in \mathcal{P}(H)$ with $d\nu(x) = \rho(x)\,d\mu(x)$, then
			\[
			d_{W,\operatorname{class}}(\mu,\nu)^2 \leq \frac{2}{c} \int_H \rho \log \rho\,d\mu.
			\]
			\item \label{item Herbst} \textbf{Herbst concentration inequality with constant $c$:}  Whenever $f: H \to \bR$ is Lipschitz, then
			\[
			\mu\left( \{x: |f(x) - \smallint_H f\,d\mu| \geq \delta \} \right) \leq 2 e^{-c \delta^2 / 2 \norm{f}_{\Lip}^2}.
			\]
		\end{enumerate}
		Then
		\begin{itemize}
			\item \eqref{item log concave} $\implies$ \eqref{item log Sobolev}; see \cite{BE1985}, \cite{BL2000}, \cite[\S 4.4.2]{AGZ2009}.
			\item \eqref{item log Sobolev} $\implies$ \eqref{item Talagrand}; see \cite{OV2000}.
			\item \eqref{item log Sobolev} $\implies$ \eqref{item Herbst}; see \cite[\S 2.3.2]{AGZ2009}.
		\end{itemize}
	\end{theorem}
	
	An important consequence of the Herbst concentration inequality the following tail bound for the operator norm of random multi-matrices.  This follows from a standard $\varepsilon$-net argument as in \cite[\S 2.3.1]{Tao2012}.  For a proof of this specific statement, see e.g.\ \cite[Lemma 2.12]{JekelExpectation}.
	
	\begin{lemma} \label{lem: operator norm tail}
		Let $\mu^{(n)} \in \mathcal{P}(\bM_n^m)$ be a probability measure satisfying the Herbst concentration inequality with constant $cn^2$, and let $\mathbf{X}^{(n)}$ be a random multi-matrix with probability distribution $\mu^{(n)}$.  Then
		\[
		P(\norm{X_j^{(n)} - \mathbb{E} X_j^{(n)} } \geq c^{-1/2}(\Theta + \delta)) \leq 2 e^{-n\delta^2}.
		\]
		for all $\delta > 0$, where $\Theta$ is a universal constant.
	\end{lemma}
	
	In order to use this result, we will also estimate for $\mathbb{E} X_j^{(n)}$ in terms of the gradient of the function at $0$.  This uses strong convexity and standard integration tricks.  The only subtlety here is to avoid assuming any more smoothness for $V$.
	
	\begin{lemma} \label{lem: expectation estimate}
		Let $H$ be a finite-dimensional real inner-product space and $V: H \to \bR$ be $c$-strongly convex.  Let $X$ be a random variable with density proportional to $e^{-V}$.  Let $y \in \underline{\nabla} V(0)$.  Then
		\[
		(\mathbb{E} \norm{X}_H^2)^{1/2} \leq (c^{-1} \dim H)^{1/2} + c^{-1} \norm{y}_H.
		\]
	\end{lemma}
	
	\begin{proof}
		Since $V$ is convex and $H$ is finite-dimensional, $V$ is locally bounded and locally Lipschitz, hence differentiable almost everywhere.  Let $B_R$ be the ball of radius $R$.  Then by the dominated convergence theorem,
		\[
		\frac{d}{dt} \biggr|_{t=1} \int_{B_R} e^{-V(tx)}\,dx = \int_{B_R} \frac{d}{dt}\biggr|_{t=1} e^{-V(tx)}\,dx = -\int_{B_R} \ip{\nabla V(x),x}_H\,dx.
		\]
		On the other hand, by change of variables and polar coordinates,
		\[
		\int_{B_R} e^{-V(tx)}\,dx = t^{-\dim H} \int_{B_{Rt}} e^{-V(x)}\,dx = t^{-\dim H} \int_0^{Rt} \int_{\partial B_s} e^{-V(x)}\,d\sigma(x) \,ds,
		\]
		where $d\sigma(x)$ is the $\dim H - 1$ dimensional surface measure on $\partial B_{Rt}$.  Therefore,
		\[
		\frac{d}{dt} \biggr|_{t=1} \int_{B_R} e^{-V(tx)}\,dx = -\dim H \int_{B_R} e^{-V(x)}\,dx + \int_{\partial B_R} e^{-V(x)}\,d\sigma(x).
		\]
		Hence,
		\[
		\int_{B_R} \ip{\nabla V(x),x}_H e^{-V(x)}\,dx = \dim H \int_{B_R} e^{-V(x)}\,dx - \int_{\partial B_R} e^{-V(x)}\,d\sigma(x).
		\]
		Next, since $V$ is $c$-strongly convex, we have that whenever $b \in \underline{\nabla} V(a)$ and $b' \in \underline{\nabla} V(a')$, then $\ip{b-b',a-a'}_H \geq c \norm{a - a'}_H^2$ (see, e.g., \cite[Fact 4.8]{JekelTypeCoupling}).  Therefore, whenever $V$ is differentiable at $x$, we have
		\[
		\ip{\nabla V(x) - y, x - 0}_H \geq \norm{x - 0}_H^2.
		\]
		Thus,
		\[
		c \int_{B_R} \norm{x}_H^2 e^{-V(x)}\,dx \leq \int_{B_R} \ip{\nabla V(x)-y,x}_H e^{-V(x)}\,dx.
		\]
		We rewrite this as
		\[
		c \int_{B_R} \left( \norm{x - (2c)^{-1}y}_H^2 - \norm{(2c)^{-1}y}_H^2 \right) e^{-V(x)}\,dx = \dim H \int_{B_R} e^{-V(x)}\,dx - \int_{\partial B_R} e^{-V(x)}\,d\sigma(x).
		\]
		Now take $R \to +\infty$.  From strong convexity, we have $V(x) \geq V(0) + \ip{x,y}_H + \frac{c}{2} \norm{x}_H^2$, which easily shows that $\int_{\partial B_R} e^{-V(x)}\,d\sigma(x) \to 0$.  Hence, upon dividing by $\int e^{-V}$, we obtain
		\[
		c \mathbb{E} \norm{X - (2c)^{-1}y}_H^2 - (4c)^{-1} \norm{y}_H^2 \leq \dim H,
		\]
		hence
		\[
		\mathbb{E} \norm{X - (2c)^{-1} y}_H^2 \leq c^{-1} \dim H + (2c)^{-2} \norm{y}_H^2 \leq [(c^{-1} \dim H)^{1/2} + (2c)^{-1} \norm{y}_H]^2.
		\]
		Then by the triangle inequality,
		\begin{align*}
			(\mathbb{E} \norm{x}_H^2)^{1/2} &\leq (\mathbb{E} \norm{X - (2c)^{-1} y}_H^2)^{1/2} + (2c)^{-1} \norm{y}_H \\
			&\leq (c^{-1} \dim H)^{1/2} + (2c)^{-1} \norm{y}_H + (2c)^{-1} \norm{y}_H \\
			&= (c^{-1} \dim H)^{1/2} + c^{-1} \norm{y}_H. \qedhere
		\end{align*}
	\end{proof}
	
	\begin{corollary} \label{cor: convex Gibbs norm estimate}
		Let $\varphi$ be a $c$-strongly convex definable predicate with respect to $\rT_{\tr,\factor}$, and let 
		\[
		C = \sup_{\cM \models \rT_{\tr,\factor}} \sup_{\mathbf{x} \in (D_1^{\cM})^m} [\varphi^{\cM}(\mathbf{x}) - \varphi^{\cM}(0)].
		\]
		Let $\mu^{(n)}$ be the measure on $(\bM_n)^m$ given by \eqref{eq: RMME}, and let $\mathbf{X}^{(n)}$ be a corresponding random variable in $(\bM_n)^m$.  Then
		\[
		P(\norm{X_j^{(n)}} \geq m^{1/2}[c^{-1/2} + c^{-1} C] + c^{1/2}(\Theta + \delta)) \leq e^{-n \delta^2}.
		\]
	\end{corollary}
	
	\begin{proof}
		By Lemma \ref{lem: gradient at zero}, there exists $y^{(n)} \in \underline{\nabla} \varphi^{\bM_n}(0) \cap \bC^m$ with $|y_j^{(n)}| \leq C$, hence $\norm{y^{(n)}}_{\tr_n} \leq Cm^{1/2}$.  Now we apply \ref{lem: expectation estimate} to the function $n^2 \varphi^{\bM_n}$ and the constant $n^2 c$ and the vector $n^2 y^{(n)}$ to conclude that
		\[
		\norm{\mathbb{E} \mathbf{X}^{(n)}}_{\tr_n} \leq (\mathbb{E} \norm{\mathbf{X}^{(n)}}_{\tr_n}^2)^{1/2} \leq c^{-1/2}m^{1/2} + c^{-1} C m^{1/2}.
		\]
		Since $\mu^{(n)}$ is invariant under unitary conjugation, $\mathbb{E} X_j^{(n)}$ is a scalar multiple of the identity, so its operator norm is the same as its $2$-norm.  Therefore, $\norm{\mathbb{E} X_j^{(n)}} \leq m^{1/2}[c^{-1/2} + c^{-1} C]$.  Combining this with Lemma \ref{lem: operator norm tail} and the triangle inequality completes the proof.
	\end{proof}
	
	Another consequence of concentration is that the type will be close to a constant with high probability.  Since we have fixed an ultrafilter $\cU$, the convergence of the expectation of some definable predicate evaluated on $\mathbf{X}^{(n)}$ comes for free, and hence concentration leads to convergence in probability.  We state the next lemma in greater generality, but note that the measures in \eqref{eq: RMME} satisfy the assumptions because of Lemma \ref{lem: expectation estimate} and the fact that $\mathbb{E} X_j^{(n)}$ is a multiple of $1$.
	
	\begin{lemma} \label{lem: limit type}
		For each $n$, let $\mu^{(n)} \in \mathcal{P}(\bM_n^m)$ be a probability measure satisfying the Herbst concentration inequality with constant $cn^2$, and let $\mathbf{X}^{(n)}$ be a random multi-matrix with probability distribution $\mu^{(n)}$.  Suppose that
		\[
		C := \sup_n \max_j \norm{\mathbb{E} X_j^{(n)}}_\infty < \infty.
		\]
		Let $R = C + c^{-1/2} \Theta$ where $\Theta$ is as in Lemma \ref{lem: operator norm tail}.  Then there is a type $\mu \in \bS_{m,R}$ such that
		\[
		\lim_{n \to \cU} \tp^{\bM_n}(\mathbf{X}^{(n)}) = \mu \text{ in probability,}
		\]
		where the limit occurs in the weak-$*$ topology on $\bS_{m,R}(\rT_{\tr})$.
	\end{lemma}
	
	\begin{proof}
		We claim that for every formula $\varphi$, the random variable $\varphi^{\bM_n}(\mathbf{X}^{(n)})$ converges in probability to a constant as $n \to \cU$.  In order to apply the Herbst concentration inequality, we want to approximate a formula $\varphi$ by a Lipschitz function.  It is straightforward to show by induction on complexity of formulas that $\varphi^{\cM}$ is Lipschitz with respect to the $L^2(\cM)^m$ norm on the domain $(D_{R'}^{\cM})^m$ for each tracial von Neumann algebra $\cM$.  We therefore want to compose $\varphi^{\cM}$ with a cutoff function that maps the entire von Neumann algebra into the operator norm ball and is also Lipschitz.
		
		Let $R' > R$, and consider the formula $\delta_{R'}(x) = \inf_{y \in D_{R'}} \frac{1}{2} \tr(|x - y|^2)$ or equivalently $\delta_{R'}^{\cM}(x) = \inf_{y \in D_{R'}^{\cM}} \frac{1}{2} \norm{x - y}_{L^2(\cM)}^2$.  Then $\delta_{R'}$ is convex by Fact \ref{lem: inf-convolution convexity} since it is the Hopf-Lax semigroup at $t = 1$ applied to the convex function which is zero on $D_R^{\cM}$ and $+\infty$ outside $D_{R'}^{\cM}$ (or see \cite[Proof of Theorem 1.1, p.~33]{JekelTypeCoupling}).  Similarly, $\delta_{R'}$ is $1$-semiconcave.  Therefore, by Corollary \ref{cor: definable gradient}, $\nabla \delta_{R'}$ is a definable function that is a $1$-Lipschitz.  Similarly, as $\frac{1}{2} \norm{x}_{L^2(\cM)}^2 - \delta_{R'}^{\cM}(x)$ is a convex and $1$-semiconcave, $x - \nabla \delta_{R'}(x)$ is a $1$-Lipschitz definable function.  Since $\delta_{R'}^{\cM}(x) = 0$ on $D_R^{\cM}$ by inspection, we also have that $\nabla \delta_{R'}^{\cM}(x) = 0$ on $D_{R'}^{\cM}$.  Furthermore,
		\[
		\frac{1}{2} \norm{x}_{L^2(\cM)}^2 - \delta_{R'}^{\cM}(x) = \sup_{y \in D_{R'}^{\cM}} \left[ \re \ip{x,y}_{L^2(\cM)} - \frac{1}{2} \norm{y}_{L^2(\cM)}^2 \right].
		\]
		Since $y \in D_{R'}^{\cM}$, each of the functions inside the supremum is $1$-Lipschitz in $x$ with respect to $\norm{\cdot}_{L^1(\cM)}$, hence also $\frac{1}{2} \norm{x}_{L^2(\cM)}^2 - \delta_{R'}^{\cM}(x)$ is $R'$-Lipschitz with respect to $\norm{\cdot}_{L^1(\cM)}$.  By taking the directional derivatives of this function, this implies that
		\[
		|\re \ip{x - \nabla \delta_{R'}^{\cM}(x), h}| \leq R' \norm{h}_{L^1(\cM)} \text{ for } h \in \cM,
		\]
		and hence by the duality between $\cM$ and $L^1(\cM)$, we have $\norm{x - \nabla \delta_{R'}^{\cM}(x)}_{\cM} \leq R$.   Therefore, $f^{\cM}(x) := x - \nabla \delta_{R'}^{\cM}(x)$ maps $\cM$ into $D_{R'}^{\cM}$ and is a $1$-Lipschitz definable function.
		
		Let $\mathbf{f}$ be the $m$-tuple of definable functions $(f,\dots,f)$.  By Lemma \ref{lem: definable composition}, $\varphi \circ \mathbf{f}$ is a definable predicate for every formula $\varphi$.  Define a linear functional $\mu$ on $\mathcal{F}_n$ as follows:  For each formula $\varphi$, let
		\[
		(\mu,\varphi) := \lim_{n \to \cU} \mathbb{E} \varphi \circ \mathbf{f}(\mathbf{X}^{(n)}).
		\]
		If the formula $\varphi$ is $L$-Lipschitz on $R'$-ball, then we have for $\delta > 0$ that
		\[
		\lim_{n \to \cU} P(|\varphi \circ \mathbf{f}(\mathbf{X}^{(n)}) - \mathbb{E} \varphi \circ \mathbf{f}(\mathbf{X}^(n))| \geq \delta) = 0
		\]
		using the Herbst concentration inequality.  Also, by Lemma \ref{lem: operator norm tail},
		\[
		\lim_{n \to \cU} P(\max_j \norm{X_j^{(n)}}_\infty \geq R') = 0.
		\]
		Since $\varphi = \varphi \circ \mathbf{f}$ on the $R'$-ball and since $(\mu,\varphi) := \lim_{n \to \cU} \mathbb{E} \varphi \circ \mathbf{f}(\mathbf{X}^{(n)})$, we thus obtain that for each $\delta > 0$,
		\[
		\lim_{n \to \cU} P(|\varphi(\mathbf{X}^{(n)}) - \mathbb{E} \varphi \circ \mathbf{f}(\mathbf{X}^(n))| \geq \delta) = 0.
		\]
		Since this is true for every formula $\varphi$, we have $\lim_{n \to \cU} \tp^{\bM_n}(\mathbf{X}^{(n)}) = \mu$ in probability by definition of the weak-$*$ topology.
		
		From this, we deduce in turn that $\mu$ is actually in $\mathbb{S}_{m,R'}(\rT_{\tr,\factor})$.  Indeed, convergence in probability implies that every weak-$*$ neighborhood of $\mu$ contains the type of $\mathbf{X}^{(n)}$ for a suitable choice of $n$ and a suitable outcome in the probability space.  Hence, every weak-$*$ neighborhood of $\mu$ in the vector-space dual of $\mathcal{F}_m$ intersects $\mathbb{S}_{m,R'}(\rT_{\tr,\factor})$.  Since the space of types $\mathbb{S}_{m,R'}(\rT_{\tr,\factor})$ is weak-$*$ compact, it is a closed subset of the dual of $\mathcal{F}_m$, and therefore, $\mu \in \mathbb{S}_{m,R'}(\rT_{\tr,\factor})$.  Finally, since $R' > R$ was arbitrary, we have $\mu \in \mathbb{S}_{m,R}(\rT_{\tr,\factor})$.
	\end{proof}
	
	\begin{proposition}
		Fix $\cU$.  Let $\varphi$ be a $c$-strongly convex definable predicate with respect to $\rT_{\tr}$.  Let $\mu^{(n)}$ be the associated probability measure on $\bM_n^m$, and let $\mu$ be the limit type given by Lemma \ref{lem: limit type}.  Then $\mu$ is the unique Gibbs type for $\varphi$, and we also have
		\begin{equation} \label{eq: entropy of Gibbs type}
			\chi_{\full}^{\cU}(\mu) = \lim_{n \to \infty} h^{(n)}(\mu^{(n)}) = \lim_{n \to \cU} \left( \frac{1}{n^2} \log Z_\varphi^{(n)} + 2m \log n \right) + (\mu,\varphi)
		\end{equation}
	\end{proposition}
	
	\begin{proof}
		First, we prove \eqref{eq: entropy of Gibbs type} by showing a cycle of inequalities.  First, because of Proposition \ref{prop: entropy of matrix model} and the operator norm tail bounds from Lemma \ref{lem: operator norm tail}, we have
		\[
		\chi_{\full}^{\cU}(\mu) \geq \lim_{n \to \cU} h^{(n)}(\mu^{(n)}).
		\]
		Next, fix $R$ large enough that $\mu \in \mathbb{S}_{m,R}(\rT_{\tr,\factor})$ and
		\[
		\lim_{n \to \cU} P(\max_j \norm{X_j^{(n)}} \leq R) = 1,
		\]
		and let $\cO$ be a neighborhood of $\mu$ in $\mathbb{S}_{m,R}(\rT_{\tr,\factor})$.  Given $\varepsilon > 0$, then for $\cO$ sufficiently small, we have
		\[
		\cO \subseteq \{ \mu': |(\mu',\varphi) - (\mu,\varphi)| < \varepsilon \}.
		\]
		Now write
		\begin{align*}
			h(\mu^{(n)}) &= -\int_{\bM_n^m} \frac{1}{Z_\varphi^{(n)}} e^{-n^2 \varphi^{\bM_n}} \log \left( \frac{1}{Z_\varphi^{(n)}} e^{-n^2 \varphi^{\bM_n}} \right) \\
			&= \log Z_\varphi^{(n)} + n^2 \int_{\bM_n} \varphi^{\bM_n}\,d\mu^{(n)}.
		\end{align*}
		Since $\varphi^{\bM_n}$ is strongly convex, it is bounded below by some constant $a$.  We then break up the domain of integration into $\Gamma_R^{(n)}(\cO)$ and its complement:
		\begin{align*}
			\int_{\bM_n} \varphi^{\bM_n}\,d\mu^{(n)} &= \int_{\Gamma_R^{(n)}(\cO)} \varphi^{\bM_n} + \int_{\bM_n \setminus \Gamma_R^{(n)}(\cO)} \varphi^{\bM_n}\,d\mu^{(n)} \\
			&\geq [(\mu,\varphi) - \varepsilon]\mu^{(n)}(\Gamma_R^{(n)}(\cO)) + a \mu^{(n)}(\bM_n \setminus \Gamma_R^{(n)}(\cO)),
		\end{align*}
		and hence
		\[
		\lim_{n \to \cU} \int_{\bM_n} \varphi^{\bM_n}\,d\mu^{(n)} \geq (\mu,\varphi) - \varepsilon.
		\]
		Hence, after renormalizing, we obtain
		\[
		\lim_{n \to \cU} h^{(n)}(\mu^{(n)}) \geq \lim_{n \to \cU} \left[ \frac{1}{n^2} \log Z_\varphi^{(n)} + 2m \log n \right] + (\mu,\varphi) - \varepsilon,
		\]
		and since $\varepsilon$ was arbitrary,
		\[
		\lim_{n \to \cU} h^{(n)}(\mu^{(n)}) \geq \lim_{n \to \cU} \left[ \frac{1}{n^2} \log Z_\varphi^{(n)} + 2m \log n \right] + (\mu,\varphi).
		\]
		For the final inequality, again consider a neighborhood $\cO$ as above.  Then
		\begin{align*}
			Z_\varphi^{(n)} &= \int_{\bM_n^m} e^{-n^2 \varphi^{\bM_n}} \\
			&\geq \int_{\Gamma_R^{(n)}(\cO)} e^{-n^2 \varphi^{\bM_n}} \\
			&\geq \vol(\Gamma_R^{(n)}(\cO)) e^{-n^2[(\mu,\varphi) + \varepsilon]},
		\end{align*}
		hence
		\[
		\frac{1}{n^2} \log Z_\varphi^{(n)} \geq \frac{1}{n^2} \log \vol(\Gamma_R^{(n)}(\cO)) - (\mu,\varphi) - \varepsilon,
		\]
		so that
		\begin{align*}
			\lim_{n \to \cU} \left[ \frac{1}{n^2} \log Z_\varphi^{(n)} + 2m \log n \right] + (\mu,\varphi) &\geq \lim_{n \to \cU} \left[ \frac{1}{n^2} \log \vol(\Gamma_R^{(n)}(\cO)) + 2 m \log n \right] - \varepsilon \\
			&\geq \chi_{\full}^{\cU}(\mu) - \varepsilon.
		\end{align*}
		Since $\varepsilon$ was arbitrary, we obtain
		\[
		\lim_{n \to \cU} \left[ \frac{1}{n^2} \log Z_\varphi^{(n)} + 2m \log n \right] + (\mu,\varphi) \geq \chi_{\full}^{\cU}(\mu),
		\]
		which completes the cycle of inequalities to show \eqref{eq: entropy of Gibbs type}.
		
		Next, we prove that $\mu$ is the unique Gibbs type for $\varphi$ with respect to $\cU$.  Let $\nu$ be some type other than $\mu$, and we will show that
		\[
		\chi^{\cU}(\nu) + (\nu,\varphi) < \chi^{\cU}(\mu) + (\mu,\varphi).
		\]
		Fix $R$ large enough that $\mu, \nu \in \mathbb{S}_{m,R}(\rT_{\tr,\factor})$. Since types are defined as linear functionals on formulas, there exists some formula $\eta$ such that $(\mu,\eta) \neq (\nu,\eta)$, and let $\varepsilon = |(\mu,\eta) - (\nu,\eta)|$.  Furthermore, as in the previous lemma, let $f$ be a cut-off function so that $\psi = \varphi \circ f = \varphi$ on $D_R^m$ and $\psi$ is $L$-Lipschitz.  Then if $n$ is $\cU$-large enough that $|\mathbb{E} \psi^{\bM_n}(\mathbf{X}^{(n)}) - (\mu,\psi)| < \varepsilon/4$, we then have
		\[
		P(|\psi^{\bM_n}(\mathbf{X}^{(n)}) - (\mu,\psi)| \geq \varepsilon / 2) \leq e^{-n^2 \varepsilon^2 / 32c}.
		\]
		Let $\cO$ be the neighborhood of $\nu$ given by
		\[
		\cO = \{\nu': |(\nu', \psi) - (\nu,\psi)| < \varepsilon/2, |(\nu',\varphi) - (\nu,\varphi)| < \varepsilon^2/64Lc \}.
		\]
		Thus, since $\Gamma_R^{(n)}(\cO) \subseteq \{|\psi^{\bM_n} - (\mu,\psi)| \geq \varepsilon/2 \}$, we have
		\[
		\mu^{(n)}(\Gamma_R^{(n)}(\cO)) \leq e^{-n^2 \varepsilon^2 / 32Lc}.
		\]
		By definition of $\mu^{(n)}$ and $\cO$, we have
		\begin{align*}
			\mu^{(n)}(\Gamma_R^{(n)}(\cO)) &= \frac{1}{Z_\varphi^{(n)}} \int_{\Gamma_R^{(n)}(\cO)} e^{-n^2 \varphi^{\bM_n}} \\
			&\geq \frac{1}{Z_\varphi^{(n)}} e^{-n^2[(\nu,\varphi) + \varepsilon^2/64Lc]} \vol(\Gamma_R^{(n)}(\cO)).
		\end{align*}
		Therefore,
		\begin{align*}
			\frac{1}{n^2} \log \vol(\Gamma_R^{(n)}(\cO)) &\leq (\nu,\varphi) + \frac{\varepsilon^2}{64Lc} + \log Z_\varphi^{(n)} + \frac{1}{n^2} \mu^{(n)}(\Gamma_R^{(n)}(\cO)) \\
			&\leq (\nu,\varphi) + \frac{\varepsilon^2}{64Lc} + \log Z_\varphi^{(n)} - \frac{\varepsilon^2}{32Lc}.
		\end{align*}
		Therefore, adding $2m \log n$ and taking the limit as $n \to \cU$, we obtain
		\[
		\chi^{\cU}(\nu) \leq (\nu,\varphi) + \chi^{\cU}(\mu) - (\mu,\varphi) - \frac{\varepsilon^2}{64Lc},
		\]
		and hence $\chi^{\cU}(\nu) - (\nu,\varphi) < \chi^{\cU}(\mu) - (\mu,\varphi)$ as desired.
	\end{proof}

	\subsection{Talagrand inequality for types via matrix models} \label{subsec: Talagrand}
	
	In this section, we show an analog of the Talagrand transportation-cost inequality for free Gibbs types for strongly convex definable predicates.  This is generalization of a similar result for non-commutative laws \cite[Theorem 2.2]{HU2006} which was also proved using random matrix approximations.
	
	\begin{definition}
		Let $\varphi$ be a definable predicate with respect to $\rT_{\tr}$ and $\mu \in \mathbb{S}_m(\rT_{\cU})$.  We say that the pair $(\mu,\varphi)$ satisfies the \emph{Talagrand inequality with constant $c$} if for all $\nu \in \mathbb{S}_m(\rT_{\cU})$, we have
		\[
		d_{W,\full}(\mu,\nu)^2 \leq \frac{2}{c} \left[ (\nu,\varphi) - \chi_{\full}^{\cU}(\nu) - (\mu,\varphi) + \chi_{\full}^{\cU}(\mu) \right].
		\]
	\end{definition}
	
	\begin{remark}
		If $(\mu,\varphi)$ satisfies the Talagrand inequality, it follows immediately that $\mu$ is the unique maximizer of $\nu \mapsto \chi(\nu) - (\nu,\varphi)$, that is, $\mu$ is the unique Gibbs type for $\varphi$.
	\end{remark}
	
	\begin{proposition}[Talagrand inequality for certain Gibbs types] \label{prop: Talagrand for Gibbs}
		Fix $\cU$.  Let $\varphi$ be a $c$-strongly convex definable predicate with respect to $\rT_{\tr}$, and let $\mu$ be its associated free Gibbs type.  Then $(\mu,\varphi)$ satisfies the Talagrand inequality with constant $c$.
	\end{proposition}
	
	\begin{proof}
		Fix $R$ such that $\mu, \nu \in \mathbb{S}_{m,R}(\rT_{\cU})$.  By Proposition \ref{prop: entropy of matrix model}, there exist measures $\nu^{(n)}$ supported on the operator norm ball of radius $R$ such that when $\mathbf{Y}^{(n)} \sim \nu^{(n)}$, then $\lim_{n \to \cU} \tp^{\bM_n}(\mathbf{Y}^{(n)}) \to \nu$ in probability and also $\chi^{\cU}(\nu) = \lim_{n \to \cU} h^{(n)}(\nu^{(n)})$.
		
		Let $\rho^{(n)}$ be the density of $\nu^{(n)}$ with respect to Lebesgue measure, and so its density with respect to $\mu^{(n)}$ is 
		\[
		\tilde{\rho}^{(n)}(x) = \rho^{(n)}(x) e^{n^2 \varphi^{\bM_n}(x)} Z_\varphi^{(n)}.
		\]
		By the Talagrand inequality for $\mu^{(n)}$, we obtain
		\[
		d_{W,\operatorname{class}}(\mu^{(n)},\nu^{(n)})^2 \leq \frac{2}{cn^2} \int_{\bM_n^m} \tilde{\rho}^{(n)} \log \tilde{\rho}^{(n)}\,d\mu^{(n)}.
		\]
		We compute
		\begin{align*}
			\int_{\bM_n^m} \tilde{\rho}^{(n)} \log \tilde{\rho}^{(n)} \,d\mu^{(n)} 
			&= \int_{\bM_n^m} \rho^{(n)} \log \tilde{\rho}^{(n)}\,dm^{(n)} \\
			&= \int_{\bM_n^m} \rho^{(n)} \log \rho^{(n)}\,dm^{(n)} + \int_{\bM_n^m} (n^2 \varphi^{\bM_n} + \log Z_\varphi^{(n)}) \rho^{(n)} \,dm^{(n)} \\
			&= -h(\nu^{(n)}) + n^2 \int \varphi^{\bM_n} \,d\nu^{(n)} + \log Z_\varphi^{(n)} \\
			&= -\left( h(\nu^{(n)}) + 2mn^2 \log n \right) + n^2 \int \varphi_{\bM_n}\,d\nu^{(n)} + \left( \log Z_\varphi^{(n)} - 2mn^2 \log n \right).
		\end{align*}
		By Lemma \ref{lem: Wasserstein distance of matrix models}, we have
		\begin{align*}
			d_{W,\full}(\mu,\nu)^2 &\leq \lim_{n \to \cU} d_{W,\operatorname{class}}(\mu^{(n)},\nu^{(n)})^2 \\
			&\leq \frac{2}{cn^2} \int_{\bM_n^m} \tilde{\rho}^{(n)} \log \tilde{\rho}^{(n)} \,d\mu^{(n)} \\
			&\leq \frac{2}{c} \lim_{n \to \cU} \left[ -\left( \frac{1}{n^2} h(\nu^{(n)}) + 2m \log n \right) + \int \varphi_{\bM_n}\,d\nu^{(n)} + \left( \frac{1}{n^2} \log Z_\varphi^{(n)} - 2m \log n \right) \right] \\
			&= \frac{2}{c} \left[ -\chi^{\cU}(\nu) + (\nu, \varphi) + \chi^{\cU}(\mu) - (\mu,\varphi) \right],
		\end{align*}
		which concludes the proof.
	\end{proof}
	
	\subsection{Separability} \label{subsec: separability}
	
	In this section, we show as a consequence of the Talagrand inequality that the set of free Gibbs types from uniformly convex definable predicates is Wasserstein separable.  We then show that Wasserstein separability of a set is preserved by taking the image under all definable functions.  Together, these facts show the impossibility of realizing all types through an analog of the moment measure construction, since the space of types is not Wasserstein separable.  One cannot even realize all the types with finite free entropy in this way, since those form a Wasserstein-dense subset by Proposition \ref{prop: topological properties} (5), which is therefore non-separable with respect to Wasserstein distance as well.
	
	\begin{proposition} \label{prop: Gibbs type separability}
		Fix $\cU$.  The set of free Gibbs types associated to strongly convex definable predicates comprise a $d_{W,\full}$-separable subset of $\mathbb{S}_m(\rT_{\cU})$.
	\end{proposition}
	
	\begin{proof}
		For $c > 0$ and $C > 0$, let $\Phi_{c,C}$ be the set of $c$-strongly convex definable predicates $\varphi$ such that
		\[
		\sup_{\cM \models \rT_{\tr,\factor}} \sup_{\mathbf{x} \in (D_1^{\cM})^m} |\varphi^{\cM}(\mathbf{x}) - \varphi^{\cM}(\mathbf{0})| \leq C.
		\]
		If $\varphi \in \Phi_{c,C}$, then by Proposition \ref{prop: convex Gibbs type} there is a unique associated free Gibbs type $\mu_{\varphi}$, and by Corollary \ref{cor: convex Gibbs norm estimate}, we have $\mu_{\varphi} \in \mathbb{S}_{m,R}(\rT_{\cU})$ where $R = m^{1/2}[c^{-1/2} + c^{-1} C] + c^{1/2}(\Theta + \delta)$.
		
		Let $\norm{\cdot}_R$ be as in Remark \ref{rem: formula completion}.  Then for $\varphi, \psi \in \Phi_{c,C}$, we have for $\nu \in \mathbb{S}_{m,R}(\rT_{\cU})$,
		\[
		\left| \left[ \chi_{\full}^{\cU}(\nu) - (\nu,\varphi) \right] - \left[ \chi_{\full}^{\cU}(\nu) - (\nu,\psi) \right] \right| \leq \norm{\varphi - \psi}_R.
		\]
		Hence,
		\[
		\left| \sup_{\nu \in \mathbb{S}_{m,R}(\rT_{\cU})} \left[ \chi_{\full}^{\cU}(\nu) - (\nu,\varphi) \right] - \sup_{\nu \in \mathbb{S}_{m,R}(\rT_{\cU})} \left[ \chi_{\full}^{\cU}(\nu) - (\nu,\psi) \right] \right| \leq \norm{\varphi - \psi}_R,
		\]
		or equivalently,
		\[
		\left| \left[ \chi_{\full}^{\cU}(\mu_\varphi) - (\mu_\varphi,\varphi) \right] - \sup_{\nu \in \mathbb{S}_{m,R}(\rT_{\cU})} \left[ \chi_{\full}^{\cU}(\mu_\psi) - (\mu_\psi,\psi) \right] \right| \leq \norm{\varphi - \psi}_R.
		\]
		Replacing $\psi$ by $\varphi$ in the last term on the left results in an additional error of no more than $\norm{\varphi - \psi}_R$.  Thus,
		\[
		\left| \left[ \chi_{\full}^{\cU}(\mu_\varphi) - (\mu_\varphi,\varphi) \right] - \sup_{\nu \in \mathbb{S}_{m,R}(\rT_{\cU})} \left[ \chi_{\full}^{\cU}(\mu_\psi) - (\mu_\psi,\varphi) \right] \right| \leq 2 \norm{\varphi - \psi}_R.
		\]
		Then by the Talagrand inequality,
		\begin{equation} \label{eq: Gibbs type distance estimate}
			d_{W,\full}(\mu_\varphi,\mu_\psi)^2 \leq \frac{4}{c} \norm{\varphi - \psi}_R.
		\end{equation}
		Since definable predicates restrict to continuous functions on $\mathbb{S}_{m,R}(\rT_{\cU})$, which is compact and metrizable in the weak-$*$ topology, there is some countable dense subset of $\Phi_{c,C}$ with respect to the seminorm $\norm{\cdot}_R$.  Thus, \eqref{eq: Gibbs type distance estimate} shows that $\{ \mu_\varphi: \varphi \in \Phi_{c,C} \}$ admits a countable dense subset with respect to $d_{W,\full}$.  This in turn implies separability of
		\[
		\bigcup_{c > 0} \bigcup_{C > 0} \{ \mu_\varphi: \varphi \in \Phi_{c,C} \},
		\]
		since this can be expressed as the union over a countable collection of values of $c$ and $C$.
	\end{proof}
	
	Next, we show that Wasserstein separability is preserved under the operation of definable pushforwards.
	
	\begin{proposition} \label{prop: pushforward separability}
		Let $\rT$ be the theory of some tracial von Neumann algebra.  Let $\mathcal{S}$ be a $d_{W,\full}$-separable subset of $\mathbb{S}_m(\rT)$.  Then
		\[
		\mathcal{S}' = \{\mathbf{f}_* \mu: \mu \in \mathcal{S}, \mathbf{f} = (f_1,\dots,f_m) \text{ definable function} \} \subseteq \mathbb{S}_{m'}(\rT)
		\]
		is also $d_{W,\full}$-separable.
	\end{proposition}
	
	This proposition rests on two facts:  Uniform continuity of definable functions and separability of the space of definable functions.  Both of these facts hold in general for metric structures (assuming the language is separable for the second item), but we focus on the case of tracial von Neumann algebras to minimize technical background.
	
	\begin{lemma}[{See \cite[Lemma 3.19]{JekelCoveringEntropy}}] \label{lem: definable functions uniform continuity}
		Let $\rT$ be a theory in the language of tracial von Neumann algebras containing $\rT_{\tr}$.  Let $\mathbf{f} = (f_1,\dots,f_{m'})$ be an $m'$-tuple of $m$-variable definable functions.  Then for every $R > 0$ and $\varepsilon > 0$, there exists $\delta > 0$ such that
		\[
		\cM \models \rT \text{ and } \mathbf{x}, \mathbf{y} \in (D_R^{\cM})^m \text{ and } \norm{\mathbf{x} - \mathbf{y}}_{L^2(\cM)^m} < \delta \implies \norm{\mathbf{f}^{\cM}(\mathbf{x}) - \mathbf{f}^{\cM}(\mathbf{y})}_{L^2(\cM)^{m'}} < \varepsilon.
		\]
	\end{lemma}
	
	\begin{lemma}[Separability of definable functions] \label{lem: definable functions separable}
		Let $\rT$ be a theory in the language of tracial von Neumann algebras containing $\rT_{\tr}$.  For $m'$-tuples of definable functions with respect to $\rT$, let
		\[
		\norm{\mathbf{f}}_R = \sup \{ \norm{\mathbf{f}^{\cM}(\mathbf{x})}_{L^2(\cM)^{m'}}: \cM \models \rT, \mathbf{x} \in (D_R^{\cM})^m \}.
		\]
		Then the space of definable functions is separable with respect to the family of seminorms $\norm{\cdot}_R$.
	\end{lemma}
	
	\begin{proof}
		Let $\mathbf{f}$ and $\mathbf{g}$ be definable functions, and let $R > 0$.  then there is some $R'$ such that $\mathbf{f}$ and $\mathbf{g}$ both map $D_R^m$ into $D_{R'}^{m'}$.  Moreover, there exist $(m+m')$-variable definable predicates $\varphi$ and $\psi$ with respect to $\rT$ such that for $\cM \models \rT$ and $\mathbf{x} \in \cM^m$ and $\mathbf{y} \in \cM^{m'}$,
		\begin{align*}
			\norm{\mathbf{f}^{\cM}(\mathbf{x}) - \mathbf{y}}_{L^2(\cM)^{m'}} &= \varphi^{\cM}(\mathbf{x},
			\mathbf{y}) \\
			\norm{\mathbf{g}^{\cM}(\mathbf{x}) - \mathbf{y}}_{L^2(\cM)^{m'}} &= \psi^{\cM}(\mathbf{x},
			\mathbf{y}).
		\end{align*}
		Therefore, for $\mathbf{x} \in (D_R^{\cM})^m$,
		\begin{align*}
			\norm{\mathbf{f}^{\cM}(\mathbf{x}) - \mathbf{g}^{\cM}(\mathbf{x})}_{L^2(\cM)^{m'}} &= \varphi^{\cM}(\mathbf{x},\mathbf{g}^{\cM}(\mathbf{x})) \\
			&\leq \psi^{\cM}(\mathbf{x},\mathbf{g}^{\cM}(\mathbf{x})) + |\varphi^{\cM}(\mathbf{x},\mathbf{g}^{\cM}(\mathbf{x})) -  \psi^{\cM}(\mathbf{x},\mathbf{g}^{\cM}(\mathbf{x}))| \\
			&\leq \norm{\mathbf{g}^{\cM}(\mathbf{x}) - \mathbf{g}^{\cM}(\mathbf{x})}_{L^2(\cM)^{m'}} + \norm{\varphi - \psi}_{\max(R,R')} \\
			&= \norm{\varphi - \psi}_{\max(R,R')}.
		\end{align*}
		And thus $\norm{\mathbf{f} - \mathbf{g}}_R \leq \norm{\varphi - \psi}_{\max(R,R')}$.  Since the space of definable predicates is separable with respect to $\norm{\cdot}_{\max(R,R')}$, we see that the space of definable functions mapping $D_R^m$ into $D_{R'}^{m'}$ is separable with respect to $\norm{\cdot}_R$.  Then taking the union over a sequence of $R'$-values tending to $\infty$, we see that the set of all definable functions with $m$ inputs and $m'$ outputs is separable with respect to $\norm{\cdot}_R$.   Since $R$ was arbitrary, the proof is complete.
	\end{proof}
	
	\begin{proof}[Proof of Proposition \ref{prop: pushforward separability}]
		For $R > 0$, let
		\[
		\mathcal{S}_R' = \{\mathbf{f}_* \mu: \mu \in \mathcal{S} \cap \mathbb{S}_{m,R}(\rT), \mathbf{f} = (f_1,\dots,f_m) \text{ definable function} \}.
		\]
		It suffices to show that $\mathcal{S}_R'$ is separable for each $R > 0$.  If $\mu \in \mathbb{S}_{m,R}(\rT)$ and $\mathbf{f}$ and $\mathbf{g}$ are definable,
		\[
		d_{W,\full}(\mathbf{f}_* \mu, \mathbf{g}_* \mu) \leq \norm{\mathbf{f} - \mathbf{g}}_R.
		\]
		Hence, if we take a countable dense collection of definable functions $(\mathbf{f}_k)_{k \in \bN}$ with respect to $\norm{\cdot}_R$ from Lemma \ref{lem: definable functions separable}, then $\bigcup_{k \in \bN} (\mathbf{f}_k)_*(\mathcal{S} \cap \mathbb{S}_{m,R}(\rT))$ is dense in $\mathcal{S}_R'$.  Thus, it suffices to show that $(\mathbf{f}_k)_*(\mathcal{S} \cap \mathbb{S}_{m,R}(\rT))$ is separable for each $k \in \bN$.  Now by Lemma \ref{lem: definable functions uniform continuity}, for every $\varepsilon > 0$, there exists $\delta > 0$ such that
		\[
		\mu, \nu \in \mathbb{S}_{m,R}(\rT) \text{ and } d_{W,\full}(\mu,\nu) < \delta \implies d_{W,\full}((\mathbf{f}_k)_* \mu, (\mathbf{f}_k)_* \nu) < \varepsilon.
		\]
		Therefore, separability of $\mathcal{S} \cap \mathbb{S}_{m,R}(\rT)$ implies separability of its pushforward under $\mathbf{f}_k$.
	\end{proof}

	\section{Quasi-moment types} \label{sec: quasi-moment types}
	
	The goal of this section is to prove Theorem \ref{thm: quasi-moment type}.  Fix $\cU$ and $\mu \in \mathbb{S}_m(\rT_{\cU})$ and $t > 0$.  In particular, since $t$ is fixed throughout the section, we will sometimes suppress the dependence on $t$ in the notation.
	
	By Remark \ref{rem: definable predicates and continuous functions}, let $\eta$ be a definable predicate such that $(\mu',\eta) \geq 0$ with equality if and only if $\mu' = \mu$.  For $\varepsilon > 0$, let
	\[
	\varphi_\varepsilon^{\cM}(\mathbf{y}) = \sup_{\mathbf{x} \in (D_R^{\cM})^m} \left[ \re \ip{\mathbf{x},\mathbf{y}}_{L^2(\cM)^m} - \frac{1}{\varepsilon} \eta^{\cM}(\mathbf{x}) \right].
	\]
	
	\begin{claim}
		We have $\varphi_\varepsilon^{\cM}(\mathbf{y}) \searrow C_{\full}(\tp^{\cM}(\mathbf{y}),\mu)$ as $\varepsilon \searrow 0$; here the limit is $-\infty$ if $\cM$ is not elementarily equivalent to $\cQ$.
	\end{claim}

	This claim is proved in \cite[Proof of Proposition 4.1]{JekelTypeCoupling}.  Note also that $\varphi_\varepsilon^{\cM}$ is convex for each $\cM$ since it is a supremum of affine functions. Next, let
	\[
	\varphi_{\varepsilon,t}^{\cM}(\mathbf{y}) = \varphi_\varepsilon^{\cM}(\mathbf{y}) + \frac{t}{2} \norm{\mathbf{y}}_{L^2(\cM)}^2,
	\]
	which is also a definable predicate over $\mathrm{T}_{\tr}$.  Then $\varphi_{\varepsilon,t}$ is $t$-strongly convex.  We thus obtain a Gibbs type for $\varphi_{\varepsilon,t}$ with respect to $\cU$, and furthermore, we will show that the Gibbs type is in $\mathbb{S}_{m,R'}(\rT_{\tr,\factor})$ where $R'$ is independent of $\varepsilon$ and only depends on $t$.
	
	\begin{claim}
		There exists a unique Gibbs type $\nu_{\varepsilon,t}$ with respect to $\cU$ for $\varphi_{\varepsilon,t}$, that is, a unique maximizer of
		\[
		\chi_{\full}^{\cU}(\nu) - (\nu,\varphi_{\varepsilon,t}).
		\]
		We also have $\nu_{\varepsilon,t} \in \mathbb{S}_{m,R'}(\rT_{\tr,\factor})$ for $R'  = t^{-1/2} + t^{-1} R m^{1/2} + t^{-1/2} \Theta$, where $\Theta$ is the constant from Lemma \ref{lem: operator norm tail}.
	\end{claim}
	
	\begin{proof}
		The existence and uniqueness of the Gibbs type follow from Proposition \ref{prop: convex Gibbs type}, and it remains to estimate the operator norm.  By Lemma \ref{lem: gradient at zero}, there exists $\mathbf{z} \in \underline{\nabla} \varphi^{\bM_n}(\mathbf{0}) \cap \bC^m$.  Moreover, using convexity,
		\[
		\varphi_\varepsilon^{\bM_n}(\mathbf{z}) - \varphi_\varepsilon^{\bM_n}(\mathbf{0}) \geq \re \ip{\mathbf{z},\mathbf{z}}_{\tr_n} = \norm{\mathbf{z}}_{\tr_n}^2
		\]
		On the other hand, for each $\mathbf{x} \in (D_R^{\bM_n})^m$,
		\[
		\left| \left[ \re \ip{\mathbf{x},\mathbf{z}}_{L^2(\cM)^m} - \frac{1}{\varepsilon} \eta^{\cM}(\mathbf{x}) \right] - \left[ \re \ip{\mathbf{x},\mathbf{0}}_{L^2(\cM)^m} - \frac{1}{\varepsilon} \eta^{\cM}(\mathbf{x}) \right] \right| \leq \norm{\mathbf{x}}_{\tr_n} \norm{\mathbf{z}}_{\tr_n} \leq R m^{1/2} \norm{\mathbf{z}},
		\]
		and since $\varphi_{\varepsilon}^{\bM_n}$ is obtained by taking the supremum over such $\mathbf{x}$,
		\[
		|\varphi_\varepsilon^{\bM_n}(\mathbf{z}) - \varphi_\varepsilon^{\bM_n}(\mathbf{0})| \leq Rm^{1/2} \norm{\mathbf{z}}_{\tr_n}
		\]
		and so $\norm{\mathbf{z}}_{\tr_n} \leq Rm^{1/2}$.  Note that $\mathbf{z}$ is also a subgradient vector for $\varphi_{\varepsilon,t}$ since the quadratic function has gradient zero at zero.
		
		Now letting $\mathbf{X}^{(n)}$ be the random matrix $m$-tuple associated to $\varphi_{\varepsilon,t}$, Lemma \ref{lem: expectation estimate} yields that $\norm{\mathbb{E} X_j^{(n)}}_\infty \leq t^{-1/2} + t^{-1} R m^{1/2}$.  Therefore, the limiting type $\nu_{\varepsilon,t}$ obtained in Lemma \ref{lem: limit type} has operator norm bounded by $R'  = t^{-1/2} + t^{-1} R m^{1/2} + t^{-1/2} \Theta$, where $\Theta$ is the constant from Lemma \ref{lem: operator norm tail}.
	\end{proof}
	
	\begin{claim}
		Let $q(x) = \frac{1}{2}\sum_{j=1}^m \tr(x_j^*x_j)$, and let
		\[
		M := \sup_{\nu \in \mathbb{S}_m(\rT_{\cU})} \left[ \chi_{\full}^{\cU}(\nu) - C_{\full}(\mu,\nu) - t (\nu,q) \right].
		\]
		Then the supremum is witnessed on the smaller set $\mathbb{S}_{m,R'}(\rT_{\cU})$.
		Moreover, if $\rho \in \mathbb{S}_m(\rT_{\cU})$ satisfies,
		\[
		\chi_{\full}^{\cU}(\rho) - C_{\full}(\rho,\mu) - \frac{t}{2} \sum_{j=1}^m \rho(x_j^*x_j) > M - \delta,
		\]
		then
		\[
		\limsup_{\varepsilon \searrow 0} d_{W,\full}(\nu_{\varepsilon,t},\rho) \leq (2\delta/t)^{1/2}.
		\]
		%In particular, the supremum of $\chi_{\full}^{\cU}(\nu) - C_{\full}(\mu,\nu) - t (\nu,q)$ over $\mathbb{S}_m(\rT_{\cU})$ agrees with the supremum over $\mathbb{S}_{m,R'}(\rT_{\cU})$.
	\end{claim}
	
	\begin{proof}
		Since $\varphi_{\varepsilon,t}$ is a monotone function of $\varepsilon$, we have
		\begin{align*}
			M &= \sup_{\nu \in \mathbb{S}_m(\rT_{\cU})} \left[ \chi_{\full}^{\cU}(\nu) - C_{\full}(\mu,\nu) - t(\nu,q) \right] \\
			&= \sup_{\nu \in \mathbb{S}_m(\rT_{\cU})} \sup_{\varepsilon > 0} \left[ \chi_{\full}^{\cU}(\nu) - (\mu,\varphi_{\varepsilon}) - t (\nu,q)  \right] \\
			&= \sup_{\varepsilon > 0} \sup_{\nu \in \mathbb{S}_m(\rT_{\cU})}  \left[ \chi_{\full}^{\cU}(\nu) - (\nu,\varphi_{\varepsilon,t})  \right] \\
			&= \sup_{\varepsilon > 0} \left[ \chi_{\full}^{\cU}(\nu_{\varepsilon,t}) - (\nu_{\varepsilon,t}, \varphi_{\varepsilon,t}) \right].
		\end{align*}
		Since $\nu_{\varepsilon,t} \in \mathbb{S}_{m,R'}(\rT_{\cU})$, we can write the same string of equalities with $\mathbb{S}_m(\rT_{\cU})$ replaced by $\mathbb{S}_{m,R'}(\rT_{\cU})$, and hence the supremum is the same if we only use $\mathbb{S}_{m,R'}(\rT_{\cU})$ as claimed.
		
		Next, let $\rho$ be given as in the statement.  Write $M_\varepsilon = \chi_{\full}^{\cU}(\nu_{\varepsilon,t}) - (\nu_{\varepsilon,t}, \varphi_{\varepsilon,t})$, so that $M_\varepsilon \nearrow M$ as $\varepsilon \searrow 0$ because of monotonicity of the previous expressions in $\varepsilon$.  Next, note that
		\[
		\lim_{\varepsilon \searrow 0} \left[ \chi_{\full}^{\cU}(\rho) - (\rho, \varphi_{\varepsilon,t}) \right] = \chi_{\full}^{\cU}(\rho) - C_{\full}(\rho,\mu) - t(\rho,q) > M - \delta,
		\]
		and hence for sufficiently small $\varepsilon$, we have
		\begin{align*}
			\chi_{\full}^{\cU}(\rho) - (\rho, \varphi_{\varepsilon,t}) &> M - \delta \\
			&\geq M_\varepsilon - \delta \\
			&= \chi_{\full}^{\cU}(\nu_{\varepsilon,t}) - (\nu_{\varepsilon,t}, \varphi_{\varepsilon,t}) - \delta.
		\end{align*}
		Applying the Talagrand inequality (Proposition \ref{prop: Talagrand for Gibbs}) to $\varphi_{\varepsilon,t}$ and $\nu_{\varepsilon,t}$, we obtain that
		\[
		d_{W,\full}(\rho,\nu_{\varepsilon,t})^2 \leq \frac{2}{t} \delta,
		\]
		as desired.
	\end{proof}
	
	\begin{claim}
		There is a unique maximizer $\nu_t$ of $\chi_{\full}^{\cU}(\nu) - C_{\full}(\mu,\nu) - t(\nu,q)$.  We have $\nu_t \in \mathbb{S}_{m,R'}(\bT_{\cU})$.  Moreover,
		\[
		\chi_{\full}^{\cU}(\rho) - C_{\full}(\mu,\rho) - t(\rho,q) \geq M - \delta \implies d_{W,\full}(\rho,\nu_t) \leq (2\delta/t)^{1/2}.
		\]
		In particular, $d_{W,\full}(\nu_t,\nu_{\varepsilon,t})^2 \leq 2(M - M_\varepsilon)/t$, so that $\nu_{\varepsilon,t} \to \nu_t$ in Wasserstein distance as $\varepsilon \to 0$.
	\end{claim}
	
	\begin{proof}
		Fix $\delta$ and let $\rho$ be such that $\chi_{\full}^{\cU}(\rho) - C_{\full}(\mu,\rho) - t(\rho,q) \geq M - \delta$.  Then we obtain
		\[
		\lim_{\varepsilon, \varepsilon' \searrow 0} d_{W,\full}(\nu_{\varepsilon,t},\nu_{\varepsilon',t}) \leq \lim_{\varepsilon, \varepsilon' \searrow 0} \left[ d_{W,\full}(\nu_{\varepsilon,t},\rho) + d_{W,\full}(\nu_{\varepsilon',t},\rho) \right] \leq 2 (2\delta/t)^{1/2}.
		\]
		Since $\delta$ was arbitrary, $(\nu_{\varepsilon,t})$ is Cauchy as $\varepsilon \searrow 0$, as desired.  Hence, it converges to some $\nu_t$ in Wasserstein distance as $\varepsilon \searrow 0$.  Since the convergence occurs in Wasserstein distance, we have $C_{\full}(\mu,\nu_{\varepsilon,t}) \to C_{\full}(\mu,\nu_t)$ as $\varepsilon \searrow 0$.  Also, since $\chi_{\full}^{\cU}$ is upper semi-continuous with respect to weak-$*$ convergence, we obtain
		\begin{align*}
			\chi_{\full}^{\cU}(\nu_t) - C_{\full}(\mu,\nu_t) - t(\nu_t,q) &\geq \limsup_{\varepsilon \searrow 0} \left[ \chi_{\full}^{\cU}(\nu_{\varepsilon,t}) - C_{\full}(\mu,\nu_{\varepsilon,t}) - t(\nu_{\varepsilon,t},q) \right] \\
			&\geq \limsup_{\varepsilon \searrow 0} \left[ \chi_{\full}^{\cU}(\nu_{\varepsilon,t}) - (\nu_{\varepsilon,t},\varphi_{\varepsilon,t}) \right] \\
			&= \limsup_{\varepsilon \searrow 0} M_\varepsilon = M.
		\end{align*}
		Therefore, $\nu_t$ achieves the maximum.
		
		Now suppose that $\chi_{\full}^{\cU}(\rho) - C_{\full}(\mu,\rho) - t(\rho,q) \geq M - \delta$.  Take $\delta' > \delta$.  Then the preceding claim shows that $d_{W,\full}(\rho,\nu_t) = \lim_{\varepsilon \searrow 0} d_{W,\full}(\rho,\nu_{\varepsilon,t}) \leq (2\delta' / t)^{1/2}$.  Since $\delta' > \delta$ was arbitrary, $d_{W,\full}(\rho,\nu_t) \leq (2\delta / t)^{1/2}$.  We can then apply this claim to $\nu_{\varepsilon,t}$ with $\delta = M - M_\varepsilon$.
	\end{proof}
	
	\begin{claim}
		For each $t$, the maximizer $\nu_t$ is the Gibbs type associated to some $t$-strongly convex definable predicate with respect to $\rT_{\tr,\factor}$.
	\end{claim}
	
	\begin{proof}
		By Theorem \ref{thm: MK duality}, there exist convex definable predicates $\varphi_t$ and $\psi_t$ such that $\varphi_t(\mathbf{x}) + \psi_t^{\cM}(\mathbf{y}) \geq \ip{\mathbf{x},\mathbf{y}}_{L^2(\cM)^m}$ for all tracial von Neumann algebras $\cM$, and $C_{\full}(\mu,\nu_t) = (\nu_t,\varphi_t) + (\mu,\psi_t)$.\footnote{The $\varphi_t$ here is not directly related to the $\varphi_{\varepsilon,t}$ above, but obtained from a separate application of the MK duality.} Then for every type $\nu$, we have
		\begin{align*}
			\chi_{\full}^{\cU}(\nu) - (\nu,\varphi_t) - (\mu,\psi_t) - t(\nu,q) &\leq \chi_{\full}^{\cU}(\nu) - C_{\full}(\mu,\nu) - t(\nu,q) \\
			&\leq \chi_{\full}^{\cU}(\nu_t) - C_{\full}(\mu,\nu_t) - t(\nu_t,q) \\
			&= \chi_{\full}^{\cU}(\nu_t) - (\nu_t,\varphi_t) - (\mu,\psi_t) - t(\nu_t,q).
		\end{align*}
		Hence, after cancelling the terms $(\mu,\psi_t)$ that are independent of $\nu$, we see that $\nu_t$ maximizes $\chi_{\full}^{\cU}(\nu) - (\nu,\varphi_t) - t(\nu,q)$, so $\nu_t$ is the Gibbs type associated to the definable predicate $\varphi_t + tq$.
	\end{proof}
	
	\section*{Acknowledgements}
	
	I thank Wuchen Li for introducing me to information geometry.  I thank Dimitri Shlyakhtenko, Wilfrid Gangbo, Kyeongsik Nam, and Aaron Palmer for continuing collaboration on non-commutative optimal transport and stochastic control theory.  I thank Jennifer Pi and Juspreet Singh Sandhu for comments on exposition in an early draft, and Charles-Philippe Diez for comments about related work on moment measures.  I thank the anonymous referee for careful reading of the manuscript and detailed corrections.
	
	\section*{Funding}
	
	This work was partially supported by Independent Research Fund of Denmark (Danmarks Frie Forskningsfond), grant 1026-00371B and the Horizon Marie Sk{\l}odowska-Curie Action FREEINFOGEOM, project id: 101209517.
	
	\section*{Data availability statement}
	
	There are no external data associated to this paper.

	\bibliographystyle{plain}
	\bibliography{ultraproduct-info-geom}
	
\end{document}